\documentclass[12pt]{amsart}
\usepackage{amssymb,amsmath,tabularx,mathrsfs,mathbbol,yfonts,upgreek}
\usepackage{amsthm,verbatim,comment}
\usepackage[bookmarks=true]{hyperref}
\usepackage{pstricks,pst-node,pst-plot}
\usepackage{geometry}
\usepackage{stmaryrd}
\usepackage{paralist}
\usepackage[all]{xy}
\usepackage{array}
\usepackage{url}
\usepackage{bbm}
\usepackage{longtable}
\numberwithin{equation}{section}

\DeclareSymbolFontAlphabet{\mathbb}{AMSb}
\DeclareSymbolFontAlphabet{\mathbbl}{bbold}

\newtheorem{thm}{Theorem}[section]

\newtheorem{lem}[thm]{Lemma}

\newtheorem{cor}[thm]{Corollary}

\theoremstyle{definition}

\newtheorem{nota}[thm]{Notation}
\newtheorem{eg}[thm]{Example}
\newtheorem{rem}[thm]{Remark}

\newtheorem*{rem*}{Remarks}

\newtheoremstyle{case}{}{}{}{}{}{:}{ }{}
\theoremstyle{case}

\newcommand{\F}{\mathbb{F}}

\title[On Terwilliger $\F$-algebras of factorial association schemes]{On Terwilliger $\F$-algebras of factorial association schemes}
\begin{document}
\author{Jiu-Yang He}
\address[J.-Y. He]{School of Mathematical Sciences, Anhui University (Jinzhai Campus), No. 128, Hean Road, Hefei, 230031, China}
\email[J.-Y. He]{a125211051@stu.ahu.edu.cn}
\author{Yu Jiang}
\address[Y. Jiang]{School of Mathematical Sciences, Anhui University (Qingyuan Campus), No. 111, Jiulong Road, Hefei, 230601, China}
\email[Y. Jiang]{jiangyu@ahu.edu.cn}
%\author{Kay Jin Lim}
%\address[K. J. Lim]{Division of Mathematical Sciences, School of Physical and Mathematical Sciences, Nanyang Technological University, 21 Nanyang Link, Singapore 637371.}
%\email[K. J. Lim]{limkj@ntu.edu.sg}
\begin{abstract}
The Terwilliger algebras of association schemes over an arbitrary field $\F$ were called the Terwilliger $\F$-algebras of association schemes in \cite{J2}. In this paper, we study the Terwilliger $\F$-algebras of factorial association schemes. We determine the centers, the semisimplicity, the Jacobson radicals and their nilpotent indices, the Wedderburn-Artin decompositions of the Terwilliger $\F$-algebras of factorial association schemes. Moreover, we determine all Terwilliger $\F$-algebras of factorial association schemes that are the symmetric $\F$-algebras or the Frobenius $\F$-algebras.%As an application, we determine all Terwilliger $\F$-algebras of factorial association schemes that are symmetric $\F$-algebras or Frobenius $\F$-algebras.%also the Frobenius algebras over the field $\F$.
\end{abstract}
\maketitle
\vspace{-1.5em}
\section{Introduction}
Association schemes on nonempty finite sets, briefly called schemes, have already been extensively studied as important research objects in algebraic combinatorics. In particular, many different tools have been introduced to study the scheme theory.

The subconstituent algebras of commutative schemes, introduced by Terwilliger in \cite{T1}, are new tools for studying schemes. They are finite-dimensional semisimple associative $\mathbb{C}$-algebras and are now known as the Terwilliger algebras of commutative schemes. In \cite{Han}, Hanaki defined the Terwilliger algebras for an arbitrary scheme and an arbitrary commutative unital ring. Following \cite{Han}, we call the Terwilliger algebras of schemes over an arbitrary field $\F$ the Terwilliger $\F$-algebras of schemes (see \cite{J2}). So the Terwilliger algebras of commutative schemes are their Terwilliger $\mathbb{C}$-algebras.

The Terwilliger $\mathbb{C}$-algebras of many commutative schemes have been investigated (for example, see \cite{BST, CD, LMP, LM, LMW, M, T1, T2, T3, TY}). However, the investigation of the Terwilliger $\F$-algebras of schemes is almost open (see \cite{Her}). In \cite{BST}, the authors suggested that studying the Terwilliger algebras of factorial schemes is interesting. In this paper, we study the Terwilliger $\F$-algebras of factorial schemes. We determine the centers, the semisimplicity, the Jacobson radicals and their nilpotent indices, the Wedderburn-Artin decompositions of the Terwilliger $\F$-algebras of factorial schemes (see Theorems \ref{T;Center}, \ref{T;Semisimplicity}, \ref{T;Jacobson}, \ref{T;Decomposition}, respectively). We also determine all Terwilliger $\F$-algebras of factorial schemes that are the symmetric $\F$-algebras or the Frobenius $\F$-algebras (see Theorem \ref{T;Frobenius}). All these results can contribute to understanding the algebraic structures of the Terwilliger $\F$-algebras of the direct products of schemes.

This paper is organized as follows: In Section 2, we collect the basic notation and preliminaries. In Section 3, we present two $\F$-bases of the Terwilliger $\F$-algebras of factorial schemes. We complete the proofs of Theorems \ref{T;Center}, \ref{T;Semisimplicity}, \ref{T;Jacobson} in Sections 4, 5, 6, respectively. Sections 7 and 8 contribute to deducing Theorems \ref{T;Decomposition} and \ref{T;Frobenius}.
\section{Basic notation and preliminaries}
For a general background on association schemes, the reader may refer to \cite{B, Z}.
\subsection{Conventions}
Let $\mathbb{N}$ be the set of all natural numbers and $\mathbb{N}_0=\mathbb{N}\cup\{0\}$. If $g, h\in\mathbb{N}_0$, let $[g, h]=\{a: a\in\mathbb{N}_0, g\leq a\leq h\}$. If $g\in\mathbb{N}$ and $\mathbb{U}_1, \mathbb{U}_2, \ldots, \mathbb{U}_g$ are sets, let $\prod_{h=1}^g\mathbb{U}_h$ be the cartesian product $\mathbb{U}_1\times \mathbb{U}_2\times\cdots\times\mathbb{U}_g$ and $\mathbf{i}_j$ be the $j$th-entry of an element $\mathbf{i}$ in $\prod_{h=1}^g\mathbb{U}_h$ for any $j\in[1, g]$. An association scheme on a nonempty finite set is briefly called a scheme. Fix a field $\F$ of characteristic $p$. Let $\delta_{g, h}$ be the Kronecker delta of the symbols $g$ and $h$ whose values are in $\F$. The addition, the multiplication, and the scalar multiplication of $\F$-matrices displayed in this paper are the usual matrix operations. Let $\F^{\times}$ be the set of all nonzero elements in $\F$. Let $\F_p$ be the prime subfield of $\F$. If $\mathbb{Z}$ is the ring of integers and $g\in\mathbb{Z}$, let $\overline{g}$ be the image of $g$ under the unital ring homomorphism from $\mathbb{Z}$ to $\F_p$. For any a subset $\mathbb{U}$ of an $\F$-linear space $\mathbb{V}$, let $\langle \mathbb{U}\rangle_\mathbb{V}$ be the $\F$-linear subspace of $\mathbb{V}$ spanned by $\mathbb{U}$.
All $\F$-algebras in this paper shall be finite-dimensional associative unital algebras. All modules in this paper shall be finite-dimensional left modules of a given $\F$-algebra.
\subsection{Schemes}
Let $\mathbb{X}$, $\mathbb{E}$ be nonempty finite sets. Let $\{R_a: a\!\in\!\mathbb{E}\}$ be a partition of $\mathbb{X}\times\mathbb{X}$. Call $\mathfrak{S}\!=\!(\mathbb{X}, \{R_a: a\!\in\!\mathbb{E}\})$ a {\em scheme} if the following conditions hold together:
\begin{enumerate}[(S1)]
\item There is a unique $0_{\mathfrak{S}}\in\mathbb{E}$ such that $R_{0_{\mathfrak{S}}}$ equals the diagonal set $\{(a,a): a\!\in\!\mathbb{X}\}$.
\item For any $g\!\in\!\mathbb{E}$, there is a unique $g^*\!\in\!\mathbb{E}$ such that $R_{g^*}$ equals $\{(a, b): (b,a)\in R_g\}$.
\item For any $g, h, i\in\mathbb{E}$ and $(x,y), (\widetilde{x},\widetilde{y})\!\in\! R_i$, there is a constant $p_{g,h}^i\in\mathbb{N}_0$ such that
$$p_{g,h}^i=|\{a: (x,a)\in R_g, (a,y)\in R_h\}|=|\{a: (\widetilde{x},a)\in R_g, (a,\widetilde{y})\in R_h\}|.$$
\end{enumerate}
From now on, let $\mathfrak{S}\!=\!(\mathbb{X}, \!\{R_a\!:\! a\!\in\!\mathbb{E}\})$ be a fixed scheme. Call $\mathfrak{S}$ a {\em symmetric scheme} if $g^*\!=\!g$ for any $g\!\in\!\mathbb{E}$. Call $\mathfrak{S}$ a {\em commutative scheme} if $p_{g,h}^i\!=\!p_{h,g}^i$ for any $g, h, i\!\in\!\mathbb{E}$. According to (S3), notice that the symmetric schemes are also commutative schemes.

Let $x R_g\!=\!\{a: (x,a)\!\in\! R_g\}$ and $k_g\!=\!p_{g,g^*}^{0_{\mathfrak{S}}}$ for any $x\!\in\!\mathbb{X}$ and $g\!\in\!\mathbb{E}$. If $x, y\!\in\!\mathbb{X}$ and $g\!\in\!\mathbb{E}$, (S2) and (S3) give $k_g\!=\!|x R_g|\!=\!|y R_g|\!\!>\!\!0$. If $g\!\in\!\mathbb{N}_0$ and $g\!\nmid\! k_h$ for any $h\in\mathbb{E}$, call $\mathfrak{S}$ a {\em $g'$-valenced scheme}. If $x, y\!\in\! \mathbb{X}$, $g, h, i\!\in\! \mathbb{E}$, and $y\!\in\! x R_i$, (S2) and (S3) give $p_{g,h}^i\!\!=\!\!|x R_g\cap y R_{h^*}|$. Call $\mathfrak{S}$ a {\em triply regular scheme} if, for any $x, y, z\!\in\!\mathbb{X}$, $g, h, i, j, k, \ell\!\in\!\mathbb{E}$, $y\!\in\! x R_j$, $z\!\in\! x R_k\cap y R_\ell$, $|x R_g\cap y R_h\cap z R_i|$ only depends on $g, h, i, j, k, \ell$ and is independent of the choices of $x, y, z$ that satisfy $y\in x R_j$ and $z\in x R_k\cap y R_\ell$. For the case that $\mathfrak{S}$ is a triply regular scheme, there is a constant $p_{g, h, i}^{j, k, \ell}\in\mathbb{N}_0$ such that $p_{g, h, i}^{j, k, \ell}=|x R_g\cap y R_h\cap z R_i|$ if $x, y, z\!\in\!\mathbb{X}$, $g, h, i, j, k, \ell\!\in\!\mathbb{E}$, $y\!\in\! x R_j$, $z\!\in\! x R_k\cap y R_\ell$.

Let $g\in\mathbb{N}$ and $\mathbb{X}_h, \mathbb{E}_h$ be nonempty finite sets for any $h\in[1, g]$. Let $\{R_a^h: a\in\mathbb{E}_h\}$ be a partition of $\mathbb{X}_h\times\mathbb{X}_h$ for any $h\!\in\![1, g]$. Let $\mathfrak{S}_h=(\mathbb{X}_h, \{R_a^h: a\in\mathbb{E}_h\})$ be a scheme for any $h\!\in\![1, g]$. If $\mathbf{x}, \mathbf{y}\!\in\!\prod_{h=1}^g\mathbb{X}_h$ and $\mathbf{i}\!\in\!\prod_{h=1}^g\mathbb{E}_h$, let $\mathbf{x}\!=_\mathbf{i}\!\mathbf{y}$ if $(\mathbf{x}_h, \mathbf{y}_h)\!\in\! R_{\mathbf{i}_h}^h$ for any $h\in[1, g]$. Let $R_\mathbf{i}\!=\! \{(\mathbf{a}, \mathbf{b}): \mathbf{a}, \mathbf{b}\in\prod_{h=1}^g\mathbb{X}_h, \mathbf{a}=_\mathbf{i}\mathbf{b}\}$ for any $\mathbf{i}\in\prod_{h=1}^g\mathbb{E}_h$. So $(\prod_{h=1}^g\mathbb{X}_h, \{R_\mathbf{a}: \mathbf{a}\!\in\! \prod_{h=1}^g\mathbb{E}_h\})$ is a scheme (see \cite{B, Z}).
Call $\mathfrak{S}$ the {\em direct product} of $\mathfrak{S}_1, \mathfrak{S}_2, \ldots, \mathfrak{S}_g$ if $\mathbb{X}=\prod_{h=1}^g\mathbb{X}_h$, $\mathbb{E}\!=\!\prod_{h=1}^g\mathbb{E}_h$, and $\{R_a: a\!\in\!\mathbb{E}\}\!=\!\{R_\mathbf{a}: \mathbf{a}\!\in\!\prod_{h=1}^g\mathbb{E}_h\}$.
If $\mathfrak{S}$ is the direct product of $\mathfrak{S}_1, \mathfrak{S}_2, \ldots, \mathfrak{S}_g$, it is clear that $0_\mathfrak{S}=(0_{\mathfrak{S}_1}, 0_{\mathfrak{S}_2},\ldots, 0_{\mathfrak{S}_g})$.
If $\mathfrak{S}$ is the direct product of $\mathfrak{S}_1, \mathfrak{S}_2, \ldots, \mathfrak{S}_g$, notice that $\mathfrak{S}$ is a symmetric scheme if and only if $\mathfrak{S}_h$ is a symmetric scheme for any $h\in[1, g]$. The next lemma is required.
\begin{lem}\label{L;Lemma2.1}\cite[Theorem 2.6.1 (iv)]{Z}
Assume that $g\in\mathbb{N}$. Assume that $\mathfrak{S}$ is the direct product of the schemes $\mathfrak{S}_1, \mathfrak{S}_2, \ldots, \mathfrak{S}_g$ and $\mathbf{i}, \mathbf{j}, \mathbf{k}\in\mathbb{E}$. Then $p_{\mathbf{i},\mathbf{j}}^\mathbf{k}=\prod_{h=1}^g p_{\mathbf{i}_h,\mathbf{j}_h}^{\mathbf{k}_h}$.
\end{lem}
Let $\mathbb{U}$ be a set that satisfies the inequality $|\mathbb{U}|>1$. Let $R(\mathbb{U})_0=\{(a, a): a\in\mathbb{U}\}$ and
$R(\mathbb{U})_1=\{(a, b): a, b\in\mathbb{U}, a\neq b\}$. Then $\{R(\mathbb{U})_0, R(\mathbb{U})_1\}$ is a partition of $\mathbb{U}\times \mathbb{U}$. Notice that $(\mathbb{U}, \{R(\mathbb{U})_0, R(\mathbb{U})_1\})$ is a scheme (see \cite[Example 1.2]{B}). Call $\mathfrak{S}$ a {\em trivial scheme} induced from $\mathbb{U}$ if $\mathbb{X}\!=\!\mathbb{U}$, $\mathbb{E}\!=\![0,1]$, and $\{R_a: a\in\mathbb{E}\}\!=\!\{R(\mathbb{U})_a: a\in[0,1]\}$. If $\mathfrak{S}$ is a trivial scheme induced from $\mathbb{U}$, notice that $\mathfrak{S}$ is a symmetric scheme. If $\mathfrak{S}$ is a trivial scheme induced from $\mathbb{U}$, then $\mathfrak{S}$ is a triply regular scheme by the definition of a triply regular scheme and a related computation. The next lemma is required.
\begin{lem}\label{L;Lemma2.2}\cite[Example 1.2]{B}
Assume that $\mathbb{U}$ is a set and $|\mathbb{U}|>1$. Assume that $\mathfrak{S}$ is a trivial scheme induced from $\mathbb{U}$ and
$\mathbb{V}\!=\!\{(0, 0, 0), (0,1, 0), (1,0, 1), (1,1, 0), (1,1,1)\}$. Assume that $i, j, k\in\mathbb{E}$. Then $p_{0,0}^0=p_{0,1}^1=p_{1,0}^1=1, p_{1,1}^0=|\mathbb{U}|-1, p_{1,1}^1=|\mathbb{U}|-2$, and $p_{i, j}^k\!\neq\!0$ only if $(i, j, k)\!\in\!\mathbb{V}$. In particular, $\{(a, b, c): a, b, c\!\in\! [0,1], \ p_{a, b}^c\!\neq\! 0\}$ equals
\[\begin{cases}
\mathbb{U}\setminus\{(1,1,1)\}, &\text{if $|\mathbb{U}|=2$},\\
\mathbb{U}, &\text{if $|\mathbb{U}|>2$}.
\end{cases}\]
\end{lem}
Fix $n\in\mathbb{N}$ and a set sequence $\mathbb{U}_1, \mathbb{U}_2, \ldots, \mathbb{U}_n$ satisfying $|\mathbb{U}_g|>1$ for any $g\in [1, n]$. Use $n_1$ and $n_2$ to denote $|\{a: a\in [1, n], |\mathbb{U}_a|=2\}|$ and $|\{a: a\in [1, n], |\mathbb{U}_a|>2\}|$, respectively. If $\mathbb{V}\subseteq [1, n]$, set $\mathbb{V}^\circ=\{a: a\in\mathbb{V}, |\mathbb{U}_a|>2\}$. For any sets $\mathbb{V}$ and $\mathbb{W}$, set $\mathbb{V}\triangle\mathbb{W}=(\mathbb{V}\setminus\mathbb{W})\cup(\mathbb{W}\setminus\mathbb{V})$. If $\mathbf{g}$ and $\mathbf{h}$ are $n$-tuples whose entries are in $[0,1]$, set $\mathbbm{g}=\{a: a\in [1, n],\mathbf{g}_a=1\}$ and  $\mathbf{g}\preceq\mathbf{h}$ if $\mathbbm{g}\subseteq\mathbbm{h}$. Let $\mathbf{0}$ be the $n$-tuple with all-zero entries. Let $\mathbf{1}$ be the $n$-tuple with all-one entries. The next lemma's proof is trivial.
\begin{lem}\label{L;Lemma2.3}
Assume that $\mathbf{g}$ and $\mathbf{h}$ are $n$-tuples whose entries are in $[0,1]$. Then $\mathbf{g}=\mathbf{h}$ if and only if $\mathbbm{g}=\mathbbm{h}$. In particular, $\preceq$ is a partial order on the set of all $n$-tuples whose entries are in $[0,1]$. Moreover, if $\mathbb{U}\subseteq [1,n]$, then there must be a unique $n$-tuple $\mathbf{i}$ such that all entries of $\mathbf{i}$ are in $[0,1]$ and the equation $\mathbbm{i}=\mathbb{U}$ holds.
\end{lem}
The trivial scheme induced from $\mathbb{U}_g$ is denoted by $\mathfrak{Tri}(\mathbb{U}_g)$ for any $g\in [1, n]$. Call $\mathfrak{S}$ the {\em factorial scheme} of $\mathfrak{Tri}(\mathbb{U}_1), \mathfrak{Tri}(\mathbb{U}_2), \ldots, \mathfrak{Tri}(\mathbb{U}_n)$ if $\mathfrak{S}$ is the direct product of $\mathfrak{Tri}(\mathbb{U}_1), \mathfrak{Tri}(\mathbb{U}_2), \ldots, \mathfrak{Tri}(\mathbb{U}_n)$ (see \cite[Pages 344, 345, 347]{B}). If $\mathfrak{S}$ is the factorial scheme of $\mathfrak{Tri}(\mathbb{U}_1), \mathfrak{Tri}(\mathbb{U}_2), \ldots, \mathfrak{Tri}(\mathbb{U}_n)$, then each element in $\mathbb{E}$ is a $n$-tuple whose entries are in $[0,1]$. If $\mathfrak{S}$ is the factorial scheme of $\mathfrak{Tri}(\mathbb{U}_1), \mathfrak{Tri}(\mathbb{U}_2), \ldots, \mathfrak{Tri}(\mathbb{U}_n)$, then $0_\mathfrak{S}=\mathbf{0}$ and $k_\mathbf{0}=1$. If $\mathfrak{S}$ is the factorial scheme of $\mathfrak{Tri}(\mathbb{U}_1), \mathfrak{Tri}(\mathbb{U}_2), \ldots, \!\mathfrak{Tri}(\mathbb{U}_n)$ and $\mathbf{g}\!\in\!\mathbb{E}\!\setminus\!\{\mathbf{0}\}$, then $k_\mathbf{g}=\prod_{h\in\mathbbm{g}}(|\mathbb{U}_h|-1)$ by Lemmas \ref{L;Lemma2.1} and \ref{L;Lemma2.2}. If $\mathfrak{S}$ is the factorial scheme of $\mathfrak{Tri}(\mathbb{U}_1), \mathfrak{Tri}(\mathbb{U}_2), \ldots, \mathfrak{Tri}(\mathbb{U}_n)$, notice that $\mathfrak{S}$ is also a symmetric scheme since $\mathfrak{Tri}(\mathbb{U}_g)$ is a symmetric scheme for any $g\in [1, n]$. The next lemmas are required.
\begin{lem}\label{L;Lemma2.4}\cite[Theorem 10]{W}
Assume that $\mathfrak{S}$ is a direct product of some triply regular schemes. Then $\mathfrak{S}$ is also a triply regular scheme. In particular, if $\mathfrak{S}$ is the factorial scheme of $\mathfrak{Tri}(\mathbb{U}_1), \mathfrak{Tri}(\mathbb{U}_2), \ldots, \!\mathfrak{Tri}(\mathbb{U}_n)$, then $\mathfrak{S}$ is a triply regular scheme.
\end{lem}
\begin{lem}\label{L;Lemma2.5}
Assume that $\mathfrak{S}$ is the factorial scheme of $\mathfrak{Tri}(\mathbb{U}_1), \mathfrak{Tri}(\mathbb{U}_2), \ldots,\! \mathfrak{Tri}(\mathbb{U}_n)$. Assume that $\mathbf{g}, \mathbf{h}, \mathrm{i}\in\mathbb{E}$. Then $p_{\mathbf{g}, \mathbf{h}}^\mathbf{i}\neq 0$ if and only if $\mathbbm{g}\triangle\mathbbm{h}\subseteq\mathbbm{i}\subseteq(\mathbbm{g}\triangle\mathbbm{h})\cup(\mathbbm{g}\cap\mathbbm{h})^\circ$.
\end{lem}
\begin{proof}
The desired lemma follows from the combination of Lemmas \ref{L;Lemma2.1}, \ref{L;Lemma2.2}, \ref{L;Lemma2.3}.
\end{proof}
\subsection{Algebras}
Let $\mathbb{A}$ be an $\F$-algebra with the zero element $0_\mathbb{A}$ and the identity element $1_\mathbb{A}$. Let $e\in\mathbb{A}$ and $\mathbb{B}$ be an $\F$-basis of $\mathbb{A}$. Then $e$ can be uniquely written as an $\F$-linear combination of the elements in $\mathbb{B}$. If $f\in \mathbb{B}$, let $c_\mathbb{B}(e, f)$ be the coefficient of $f$ in this $\F$-linear combination that represents $e$. Then $\{a: a\in \mathbb{B}, c_{\mathbb{B}}(e, a)\in\F^\times\}$ is denoted by $\mathrm{Supp}_{\mathbb{B}}(e)$. So $e=0_\mathbb{A}$ if and only if $\mathrm{Supp}_{\mathbb{B}}(e)=\varnothing$. Call $e$ an {\em idempotent} of $\mathbb{A}$ if $e^2=e$. The {\em center} $\langle\{a: a\!\in\! \mathbb{A}, ab=ba\ \forall\ b\in\mathbb{A}\}\rangle_\mathbb{A}$ of $\mathbb{A}$ is written as $\mathrm{Z}(\mathbb{A})$.
So $\mathrm{Z}(\mathbb{A})$ is an $\F$-subalgebra of $\mathbb{A}$ with the identity element $1_\mathbb{A}$.
For any $g\in\mathbb{N}$ and a division $\F$-algebra $\mathbb{K}$, let $\mathrm{M}_g(\mathbb{K})$ be the full matrix $\F$-algebra of $(g\times g)$-matrices whose entries are in $\mathbb{K}$. If $g, h\in\mathbb{N}$ and $\mathbb{K}$ is a division $\F$-algebra, let
$g\mathrm{M}_h(\mathbb{K})$ be the direct sum of $g$ copies of $\mathrm{M}_h(\mathbb{K})$. Then $\mathbb{A}$ is called a {\em semisimple $\F$-algebra} if there are $g, h_1, h_2, \ldots, h_g\in\mathbb{N}$ and division $\F$-algebras $\mathbb{K}_1, \mathbb{K}_2, \ldots, \mathbb{K}_g$ satisfying the condition
$$\mathbb{A}\cong \bigoplus_{i=1}^g\mathrm{M}_{h_i}(\mathbb{K}_i)\ \text{as $\F$-algebras}.$$

For any an $\F$-linear subspace $\mathbb{U}$ of an $\F$-linear space $\mathbb{V}$, let $\mathbb{V}/\mathbb{U}$ be the quotient $\F$-linear space of $\mathbb{V}$ with respect to $\mathbb{U}$. Let $\mathbb{I}$ be a two-sided ideal of $\mathbb{A}$ and $\mathbb{A}/\mathbb{I}$ be the quotient $\F$-algebra of $\mathbb{A}$ with respect to $\mathbb{I}$. Call $\mathbb{A}$ a {\em nilpotent two-sided ideal} of $\mathbb{A}$ if there is $g\in\mathbb{N}$ such that $e_1e_2\cdots e_g=0_\mathbb{A}$ for any $h\in [1, g]$ and $e_h\in\mathbb{I}$. If $\mathbb{I}$ is a nilpotent two-sided ideal of $\mathbb{A}$ and $g\in\mathbb{N}$, call $g$ the {\em nilpotent index} of $\mathbb{I}$ if $g$ is the smallest choice in $\mathbb{N}$ that satisfies the condition $e_1e_2\cdots e_g=0_\mathbb{A}$ for any $h\in [1, g]$ and $e_h\in\mathbb{I}$. The {\em Jacobson radical} of $\mathbb{A}$ is the sum of all nilpotent two-sided ideals of $\mathbb{A}$. Let $\mathrm{Rad}(\mathbb{A})$ be the Jacobson radical of $\mathbb{A}$. So $\mathrm{Rad}(\mathbb{A})$ is also a nilpotent two-sided ideal of $\mathbb{A}$. If $\mathbb{U}\subseteq\mathbb{A}$, let $e\mathbb{U}e=\{eae: a\in\mathbb{U}\}$. If $e$ is an idempotent of $\mathbb{A}$, then $e\mathbb{I}e$ is an $\F$-subalgebra of $\mathbb{A}$ with the identity element $e$. The next lemmas are required.
\begin{lem}\label{L;Lemma2.6}\cite[Theorem 3.1.1, Proposition 3.2.4]{DK}
Assume that $\mathbb{A}$ is an $\F$-algebra. Then $\mathbb{A}$ is a semisimple $\F$-algebra if and only if
$\mathrm{Rad}(\mathbb{A})=\{0_\mathbb{A}\}$. Moreover, assume that $e$ is an idempotent of $\mathbb{A}$.
Then $\mathrm{Rad}(e\mathbb{A}e)=e\mathrm{Rad}(\mathbb{A})e\subseteq\mathrm{Rad}(\mathbb{A})$.
In particular, both $\mathrm{Z}(\mathbb{A})$ and $e\mathbb{A}e$ are semisimple $\mathbb{F}$-subalgebras of $\mathbb{A}$
if $\mathbb{A}$ is a semisimple $\F$-algebra.
\end{lem}
\begin{lem}\label{L;Lemma2.7}\cite[Corollary 3.1.14]{DK}
Assume that $\mathbb{A}$ is an $\F$-algebra. Assume that $\mathbb{I}$ is a nilpotent two-sided ideal of $\mathbb{A}$. \!Then $\mathrm{Rad}(\mathbb{A}/\mathbb{I})\!=\!\mathrm{Rad}(\mathbb{A})/\mathbb{I}$. In particular, there are $g, h_1, h_2, \ldots, h_g\in\mathbb{N}$ and division $\F$-algebras $\mathbb{K}_1, \mathbb{K}_2, \ldots, \mathbb{K}_g$ satisfying the condition
$$
\mathbb{A}/\mathrm{Rad}(\mathbb{A})\cong\bigoplus_{i=1}^g\mathrm{M}_{h_i}(\mathbb{K}_i)\text{ as $\F$-algebras}.
$$
\end{lem}
Let $\mathbb{U}$ be an $\mathbb{A}$-module. Call $\mathbb{U}$ an {\em irreducible $\mathbb{A}$-module} if $\mathbb{U}$ has no nonzero proper $\mathbb{A}$-submodule. Let $\mathbb{K}$ be an extension field of $\F$. Notice that $\mathbb{K}$ is an $\F$-algebra via the addition and the multiplication of the elements in $\mathbb{K}$. Hence $\mathbb{K}$ is an $\F$-linear space via the addition and the multiplication of the elements in $\mathbb{K}$. Let $\mathbb{A}_\mathbb{K}$ be the $\F$-tensor product $\mathbb{K}\otimes\mathbb{A}$ of $\F$-algebras. Let $\mathbb{U}_\mathbb{K}$ be the $\F$-tensor product $\mathbb{K}\otimes \mathbb{U}$ of $\F$-linear spaces. Then $\mathbb{U}_\mathbb{K}$ is an $\mathbb{A}_\mathbb{K}$-module via the diagonal action of $\mathbb{A}_\mathbb{K}$ on $\mathbb{U}_\mathbb{K}$. Call $\mathbb{U}$ an {\em absolutely irreducible $\mathbb{A}$-module} if $\mathbb{U}_\mathbb{L}$ is an irreducible $\mathbb{A}_\mathbb{L}$-module for any an extension field $\mathbb{L}$ of $\mathbb{F}$. So an absolutely irreducible $\mathbb{A}$-module is an irreducible $\mathbb{A}$-module. However, the converse may not be true. The next lemmas are required.
\begin{lem}\label{L;Lemma2.8}\cite[Corollary 3.1.7]{DK}
Assume that $\mathbb{A}$ is an $\F$-algebra and there are $g, h_1, h_2, \ldots, h_g\in\mathbb{N}$ and division $\F$-algebras $\mathbb{K}_1, \mathbb{K}_2, \ldots, \mathbb{K}_g$ satisfying the condition
$$
\mathbb{A}/\mathrm{Rad}(\mathbb{A})\cong\bigoplus_{i=1}^g\mathrm{M}_{h_i}(\mathbb{K}_i)\text{ as $\F$-algebras}.
$$
Then $g$ is equal to the number of all pairwise nonisomorphic irreducible $\mathbb{A}$-modules.
\end{lem}
\begin{lem}\label{L;Lemma2.9}\cite[Corollary 1.10.4]{K}
Assume that $\mathbb{A}$ is an $\F$-algebra. Then there are $g, h_1, h_2, \ldots, h_g\in\mathbb{N}$ such that the quotient $\F$-algebra $\mathbb{A}/\mathrm{Rad}(\mathbb{A})$ satisfies the formula
$$
\mathbb{A}/\mathrm{Rad}(\mathbb{A})\cong\bigoplus_{i=1}^g\mathrm{M}_{h_i}(\F)\text{ as $\F$-algebras}
$$
if and only if every irreducible $\mathbb{A}$-module is also an absolutely irreducible $\mathbb{A}$-module.
\end{lem}
Let $\mathrm{Hom}_\F(\mathbb{V}, \mathbb{F})$ be the dual $\F$-linear space of an $\F$-linear space $\mathbb{V}$. So $\mathrm{Hom}_\F(\mathbb{U}, \mathbb{F})$ is defined as
$\mathbb{U}$ is a defined $\mathbb{A}$-module. If there is an algebra anti-homomorphism $\alpha$ from $\mathbb{A}$ to $\mathbb{A}$, let $\mathbb{A}$ act on $\mathrm{Hom}_\F(\mathbb{U}, \mathbb{F})$ by setting
$(f\beta)(g)=\beta(\alpha(f) g)$ for any $f\in\mathbb{A}$, $\beta\in\mathrm{Hom}_\F(\mathbb{U}, \mathbb{F})$, and $g\in\mathbb{U}$. If there is an algebra anti-homomorphism $\alpha$ from $\mathbb{A}$ to $\mathbb{A}$,
then the defined action of $\mathbb{A}$ on $\mathrm{Hom}_\F(\mathbb{U}, \mathbb{F})$ is briefly called the {\em $\alpha$-action} of $\mathbb{A}$. If there is an algebra anti-homomorphism $\alpha$ from $\mathbb{A}$ to $\mathbb{A}$, it is obvious to see that $\mathrm{Hom}_\F(\mathbb{U}, \mathbb{F})$ is an $\mathbb{A}$-module under the $\alpha$-action of $\mathbb{A}$. If there is an algebra anti-homomorphism $\alpha$ from $\mathbb{A}$ to $\mathbb{A}$, use $\mathbb{U}^\alpha$ to denote the $\mathbb{A}$-module $\mathrm{Hom}_\F(\mathbb{U}, \mathbb{F})$ under the $\alpha$-action of $\mathbb{A}$. If there is an algebra anti-homomorphism $\alpha$ from $\mathbb{A}$ to $\mathbb{A}$ and $\mathbb{U}^\alpha\cong\mathbb{U}$ as $\mathbb{A}$-modules, call $\mathbb{U}$ a {\em self-contragredient $\mathbb{A}$-module} with respect to $\alpha$.

Let $\mathbb{J}f=\{af: a\in\mathbb{J}\}$ and $f\mathbb{J}=\{fa: a\in\mathbb{J}\}$ for any $f\in\mathbb{A}$ and $\mathbb{J}\subseteq\mathbb{A}$. If $\mathbb{J}$ is a left ideal of $\mathbb{A}$, let
$\mathrm{r}(\mathbb{J})=\{a: a\in\mathbb{A}, \mathbb{J}a=\{0_\mathbb{A}\}\}$ and notice that $\mathrm{r}(\mathbb{J})$ is a right ideal of $\mathbb{A}$. If $f\in\mathbb{A}$ and $\alpha$ is a map whose domain is $\mathbb{A}$, let $\alpha(\mathbb{J}f)=\{\alpha(af): a\in\mathbb{J}\}$ and $\alpha(f\mathbb{J})=\{\alpha(fa): a\in\mathbb{J}\}$ for any $\mathbb{J}\subseteq\mathbb{A}$. For any $\alpha\in\mathrm{Hom}_\F(\mathbb{A}, \mathbb{F})$, call $\alpha$ a {\em nonsingular $\F$-linear functional} of $\mathbb{A}$ if the equations $\alpha(\mathbb{A}f)=\{\overline{0}\}$, $\alpha(f\mathbb{A})=\{\overline{0}\}$, and $f=0_\mathbb{A}$ are pairwise equivalent for any $f\in\mathbb{A}$. For any $\alpha\in\mathrm{Hom}_\F(\mathbb{A}, \mathbb{F})$, call $\alpha$ a {\em symmetric $\F$-linear functional} of $\mathbb{A}$ if $\alpha(fg)=\alpha(gf)$ for any $g, f\in\mathbb{A}$.
Call $\mathbb{A}$ a {\em Frobenius $\F$-algebra} if there is a nonsingular $\F$-linear functional of $\mathbb{A}$. Call $\mathbb{A}$ a {\em symmetric $\F$-algebra} if $\mathbb{A}$ is a Frobenius $\F$-algebra and a nonsingular $\F$-linear functional of $\mathbb{A}$ is also a symmetric $\F$-linear functional of $\mathbb{A}$. It is obvious to see that a symmetric $\F$-algebra is also a Frobenius $\F$-algebra. The next lemmas are required.
\begin{lem}\label{L;Lemma2.10}\cite[Theorem 3.6 (iii)]{K}
Assume that $\mathbb{A}$ is an $\F$-algebra. Then $\mathbb{A}$ is a Frobenius $\F$-algebra only if, for any a left ideal $\mathbb{I}$ of $\mathbb{A}$, the $\F$-dimension of $\mathbb{A}$ is equal to the sum of the $\F$-dimension of $\mathbb{I}$ and the $\F$-dimension of the right ideal $\mathrm{r}(\mathbb{I})$ of $\mathbb{A}$.
\end{lem}
\begin{lem}\label{L;Lemma2.11}\cite[Lemma 3.4]{K}
Assume that $g\!\in\!\mathbb{N}$ and $\mathbb{A}_1, \mathbb{A}_2,\ldots, \mathbb{A}_g$ are $\F$-algebras. Then the direct sum of $\mathbb{A}_1, \mathbb{A}_2,\ldots, \mathbb{A}_g$ is a Frobenius $\F$-algebra if and only if $\mathbb{A}_h$ is a Frobenius $\F$-algebra for any $h\in [1, g]$. Moreover, the direct sum of $\mathbb{A}_1, \mathbb{A}_2,\ldots, \mathbb{A}_g$ is a symmetric $\F$-algebra if and only if $\mathbb{A}_h$ is a symmetric $\F$-algebra for any $h\in [1, g]$. In particular, if $\mathbb{A}$ is a semisimple $\F$-algebra, then $\mathbb{A}$ is also a symmetric $\F$-algebra.
\end{lem}
\subsection{Terwilliger $\F$-algebras of schemes}
Let $\mathbb{U}$ be a nonempty finite set. Let $\mathrm{M}_\mathbb{U}(\F)$ be the full matrix $\F$-algebra of square $\F$-matrices whose rows and columns are labeled by the elements in $\mathbb{U}$. Notice that $\mathrm{M}_\mathbb{U}(\F)\cong \mathrm{M}_{|\mathbb{U}|}(\F)$ as $\F$-algebras. Let $I, O$ be the identity matrix and the zero matrix in $\mathrm{M}_\mathbb{X}(\F)$, respectively. Let $x, y\in\mathbb{X}$. If $M\in\mathrm{M}_\mathbb{X}(\F)$, let $M(x,y)$ be the $(x, y)$-entry of $M$ and $M^T$ be the transpose of $M$.

Let $g, h\in\mathbb{E}$ and $A_g$ be the adjacency $\F$-matrix with respect to $R_g$. Recall that $A_g$ is a $\{\overline{0}, \overline{1}\}$-matrix in $\mathrm{M}_\mathbb{X}(\F)$. Moreover, $A_g(x, y)=\overline{1}$ if and only if $y\in xR_g$. Let
$E_g^*(x)$ be the dual $\F$-idempotent with respect to $x$ and $R_g$. Recall that $E_g^*(x)$ is a diagonal $\{\overline{0}, \overline{1}\}$-matrix in $\mathrm{M}_\mathbb{X}(\F)$. Moreover, $E_g^*(x)(y,y)=\overline{1}$ if and only if $y\in xR_g$.
It is clear that $A_g^T=A_g$ if $\mathfrak{S}$ is a symmetric scheme. In general, recall the equations
\begin{align}\label{Eq;1}
A_g^T=A_{g^*}\ \text{and}\ E_g^*(x)^T=E_g^*(x),
\end{align}
\begin{align}\label{Eq;2}
E_g^*(x)E_h^*(x)=\delta_{g,h}E_g^*(x),
\end{align}
\begin{align}\label{Eq;3}
A_{0_\mathfrak{S}}=I=\sum_{i\in\mathbb{E}}E_i^*(x).
\end{align}

Call an $\F$-subalgebra of $\mathrm{M}_\mathbb{X}(\F)$ the {\em Terwilliger $\F$-algebra} of $\mathfrak{S}$ with respect to $x$
if it is generated by $\{A_a: a\in\mathbb{E}\}\cup\{E_a^*(x): a\in\mathbb{E}\}$. Let $\mathbb{T}(x)$ be the Terwilliger $\F$-algebra of $\mathfrak{S}$ with respect to $x$. So $\mathbb{T}(x)$ is a unital $\F$-subalgebra of $\mathrm{M}_\mathbb{X}(\F)$ by Equation \eqref{Eq;3}. According to Equation \eqref{Eq;1}, the $\F$-linear map $\alpha_T$ that sends $M$ to $M^T$ is an algebra anti-automorphism from $\mathbb{T}(x)$ to $\mathbb{T}(x)$. Let $i\in\mathbb{E}$. Then $E_g^*(x)A_hE_i^*(x)\neq O$ if and only if $p_{g^*, i}^h\!\neq \!0$ (see \cite[Lemma 3.2]{Han}). So $\{E_a^*(x)A_bE_c^*(x): a, b, c\in \mathbb{E}, p_{a^*, c}^b\!\neq\! 0\}$ is an $\F$-linearly independent subset of $\mathbb{T}(x)$ by the definition of $\mathbb{T}(x)$ and Equation \eqref{Eq;2}. In general, the algebraic structures of $\mathbb{T}(x)$ and $\mathrm{Rad}(\mathbb{T}(x))$ may also depend on both $\F$ and $x$ (see \cite[5.1]{Han}). The reader may refer to \cite{CX,Han, J1, J2} for some recent progress on the algebraic structures of $\mathbb{T}(x)$ and $\mathrm{Rad}(\mathbb{T}(x))$. The next lemmas are required.
\begin{lem}\label{L;Lemma2.12}\cite[Theorem 3.4]{Han}
Assume that $x\in \mathbb{X}$. Then $\mathbb{T}(x)$ is a semisimple $\F$-algebra only if $\mathfrak{S}$ is a $p'$-valenced scheme.
\end{lem}
\begin{lem}\label{L;Lemma2.13}\cite[Lemma 4]{MA}
Assume that $x\in \mathbb{X}$ and $\mathfrak{S}$ is a triply regular scheme. Then $\{E_a^*(x)A_bE_c^*(x): a, b, c\in \mathbb{E}, p_{a^*,c}^b\neq0\}$ is an $\F$-basis of $\mathbb{T}(x)$. In particular, the $\F$-dimension of $\mathbb{T}(x)$ only depends on $\mathfrak{S}$ and is independent of the choice of $x$ in $\mathbb{X}$.
\end{lem}
\begin{lem}\label{L;Lemma2.14}
Assume that $x\in \mathbb{X}$ and $\mathfrak{S}$ is a symmetric scheme. Assume that $g\in \mathbb{E}$ and $\mathfrak{S}$
is also a triply regular scheme. Then the $\F$-subalgebra $E_g^*(x)\mathbb{T}(x)E_g^*(x)$ of $\mathbb{T}(x)$ is also a commutative $\F$-algebra with an $\F$-basis $\{E_g^*A_aE_g^*:\ a\in\mathbb{E}, \ p_{g, g}^a\neq 0\}$.
\end{lem}
\begin{proof}
The desired lemma is from combining Lemma \ref{L;Lemma2.13}, Equations \eqref{Eq;2}, \eqref{Eq;1}.
\end{proof}
We end this section by simplifying the notation. From now on, assume that $\mathfrak{S}$ is the factorial scheme of $\mathfrak{Tri}(\mathbb{U}_1), \mathfrak{Tri}(\mathbb{U}_2), \ldots, \mathfrak{Tri}(\mathbb{U}_n)$. Recall that all entries of each $n$-tuple in $\mathbb{E}$ are arbitrarily chosen from $[0,1]$. This fact implies that $|\mathbb{E}|=2^n$. We shall quote the fact that $\mathfrak{S}$ is a symmetric scheme without citation. From now on, fix $\mathbf{x}\in\mathbb{X}$. For convenience, we abbreviate $\mathbb{T}=\mathbb{T}(\mathbf{x})$ and $E_\mathbf{g}^*=E_\mathbf{g}^*(\mathbf{x})$ for any $\mathbf{g}\in \mathbb{E}$.
\section{Algebraic structure of $\mathbb{T}$: $\F$-Basis}
In this section, we present two $\F$-bases of $\mathbb{T}$. Moreover, we present the structure constants of these $\F$-bases in $\mathbb{T}$. We first display a notation and a preliminary lemma.
\begin{nota}\label{N;Notation3.1}
Let $\mathbb{P}=\{(\mathbf{a}, \mathbf{b}, \mathbf{c}): \mathbf{a}, \mathbf{b}, \mathbf{c}\in\mathbb{E}, \mathbbm{a}\triangle\mathbbm{c}\subseteq\mathbbm{b}\subseteq(\mathbbm{a}\triangle\mathbbm{c})\cup(\mathbbm{a}\cap\mathbbm{c})^\circ\}$.
For any $\mathbf{g}, \mathbf{h}\in\mathbb{E}$, let $\mathbb{P}_{\mathbf{g}, \mathbf{h}}=\{\mathbf{a}: (\mathbf{g}, \mathbf{a}, \mathbf{h})\in\mathbb{P}\}$. For any $\mathbf{g}, \mathbf{h}, \mathbf{i}\in\mathbb{E}$, Lemma \ref{L;Lemma2.5}  thus implies that the conditions $(\mathbf{g}, \mathbf{h}, \mathbf{i})\!\in\!\mathbb{P}$, $\mathbf{h}\!\in\!\mathbb{P}_{\mathbf{g}, \mathbf{i}}$, and $p_{\mathbf{g}, \mathbf{i}}^\mathbf{h}\!\neq\! 0$ are pairwise equivalent.
\end{nota}
\begin{lem}\label{L;Lemma3.2}
$\mathbb{T}$ has an $\F$-basis $\{E_\mathbf{a}^*A_\mathbf{b}E_\mathbf{c}^*\!: (\mathbf{a}, \mathbf{b}, \mathbf{c})\!\in\!\mathbb{P}\}$ whose cardinality equals $|\mathbb{P}|$.
\end{lem}
\begin{proof}
The desired lemma follows from an application of Lemmas \ref{L;Lemma2.4} and \ref{L;Lemma2.13}.
\end{proof}
The $\F$-basis of $\mathbb{T}$ in Lemma \ref{L;Lemma3.2} motivates us to present two additional lemmas.
\begin{lem}\label{L;Lemma3.3}
The $\F$-dimension of a Terwilliger $\F$-algebra of $\mathfrak{Tri}(\mathbb{U}_1)$ is equal to
\[\begin{cases}
4, & \text{if $|\mathbb{U}_1|=2$},\\
5, & \text{if $|\mathbb{U}_1|>2$}.
\end{cases}\]
\end{lem}
\begin{proof}
The desired lemma follows from an application of Lemmas \ref{L;Lemma3.2} and \ref{L;Lemma2.2}.
\end{proof}
\begin{lem}\label{L;Lemma3.4}
The $\F$-dimension of $\mathbb{T}$ is equal to $2^{2n_1}5^{n_2}$.
\end{lem}
\begin{proof}
Recall that $n_1\!=\!|\{a: a\in [1, n], |\mathbb{U}_a|=2\}|$ and $n_2\!=\!|\{a: a\in [1, n], |\mathbb{U}_a|\!>\!2\}|$. Notice that
$n=n_1+n_2$. Let $g_h$ be the $\F$-dimension of a Terwilliger $\F$-algebra of $\mathfrak{Tri}(\mathbb{U}_h)$ for any $h\in [1, n]$.
Notice that the $\F$-dimension of $\mathbb{T}$ is equal to $\prod_{i=1}^n g_i$ by Lemmas \ref{L;Lemma3.2} and \ref{L;Lemma2.1}. Therefore the desired lemma follows from Lemma \ref{L;Lemma3.3}.
\end{proof}
We are now ready to introduce the first $\F$-basis of $\mathbb{T}$ and an additional notation.
\begin{thm}\label{T;Theorem3.5}
$\mathbb{T}$ has an $\F$-basis $\{E_\mathbf{a}^*A_\mathbf{b}E_\mathbf{c}^*: (\mathbf{a}, \mathbf{b}, \mathbf{c})\in\mathbb{P}\}$ with cardinality $2^{2n_1}5^{n_2}$.
\end{thm}
\begin{proof}
The desired theorem follows from an application of Lemmas \ref{L;Lemma3.2} and \ref{L;Lemma3.4}.
\end{proof}
\begin{nota}\label{N;Notation3.6}
Let $\mathbb{B}_1=\{E_\mathbf{a}^*A_\mathbf{b}E_\mathbf{c}^*: (\mathbf{a}, \mathbf{b}, \mathbf{c})\in\mathbb{P}\}$. If $M\!\in\!\mathbb{T}$ and $(\mathbf{g}, \mathbf{h}, \mathbf{i})\!\in\!\mathbb{P}$, notice that $\mathrm{Supp}_{\mathbb{B}_1}(M)$ is defined and write $c^{\mathbf{g},\mathbf{h},\mathbf{i}}(M)$ for $c_{\mathbb{B}_1}(M, E_\mathbf{g}^*A_\mathbf{h}E_\mathbf{i}^*)$
by Theorem \ref{T;Theorem3.5}.
\end{nota}
We next investigate the structure constants of $\mathbb{B}_1$ in $\mathbb{T}$. We first present a lemma.
\begin{lem}\label{L;Lemma3.7}
Assume that $\mathbf{g}, \mathbf{h}, \mathbf{i}, \mathbf{j}, \mathbf{k}\in\mathbb{E}$. Then
$$ E_\mathbf{g}^*A_\mathbf{h}E_\mathbf{i}^*A_\mathbf{j}E_\mathbf{k}^*=\sum_{\mathbf{l}\in\mathbb{P}_{\mathbf{g}, \mathbf{k}}}\overline{p_{\mathbf{h},\mathbf{i}, \mathbf{j}}^{\mathbf{g}, \mathbf{l}, \mathbf{k}}}E_\mathbf{g}^*A_\mathbf{l}E_\mathbf{k}^*.$$
Moreover, if there is $\mathbf{l}\!\in\!\mathbb{P}_{\mathbf{g}, \mathbf{k}}$ such that the condition $p_{\mathbf{h},\mathbf{i}, \mathbf{j}}^{\mathbf{g}, \mathbf{l}, \mathbf{k}}\!\neq\!0$ holds, then $\mathbf{l}\!\in\!\mathbb{P}_{\mathbf{g}, \mathbf{k}}\cap\mathbb{P}_{\mathbf{h}, \mathbf{j}}$.
\end{lem}
\begin{proof}
The desired lemma follows from Theorem \ref{T;Theorem3.5} and a direct computation.
\end{proof}
\begin{thm}\label{T;Theorem3.8}
Assume that $\mathbf{g}, \mathbf{h}, \mathbf{i}, \mathbf{j}, \mathbf{k}, \mathbf{l}, \mathbf{m}\in\mathbb{E}$, $\mathbf{h}\in\mathbb{P}_{\mathbf{g}, \mathbf{i}}$, $\mathbf{k}\in\mathbb{P}_{\mathbf{j}, \mathbf{l}}$, $\mathbf{m}\in\mathbb{P}_{\mathbf{g}, \mathbf{l}}$. Then
$$ c^{\mathbf{g}, \mathbf{m},\mathbf{l}}(E_\mathbf{g}^*A_\mathbf{h}E_\mathbf{i}^*E_\mathbf{j}^*A_\mathbf{k}E_\mathbf{l}^*)=
\delta_{\mathbf{i},\mathbf{j}}\overline{p_{\mathbf{h}, \mathbf{i}, \mathbf{k}}^{\mathbf{g}, \mathbf{m}, \mathbf{l}}}.$$
\end{thm}
\begin{proof}
The desired theorem follows from using Equation \eqref{Eq;2} and Lemma \ref{L;Lemma3.7}.
\end{proof}
The next goal is to find another $\F$-basis of $\mathbb{T}$. We list a notation and three lemmas.
\begin{nota}\label{N;Notation3.9}
Assume that $\mathbf{g}, \mathbf{h}, \mathbf{i}\in\mathbb{E}$ and $\mathbf{h}\in\mathbb{P}_{\mathbf{g}, \mathbf{i}}$. Then $\{\mathbf{a}: \mathbf{a}\!\in\!\mathbb{E}, \mathbbm{g}\triangle\mathbbm{i}\!\subseteq\! \mathbbm{a}\subseteq\mathbbm{h}\}$ is denoted by $\mathbb{P}_{\mathbf{g}, \mathbf{h}, \mathbf{i}}$. Notice that $\mathbb{P}_{\mathbf{g}, \mathbf{h},\mathbf{i}}\!\neq\!\varnothing$ as $\{\mathbf{h}\}\subseteq\mathbb{P}_{\mathbf{g}, \mathbf{h},\mathbf{i}}\subseteq \mathbb{P}_{\mathbf{g},\mathbf{i}}$ by Lemma \ref{L;Lemma2.5}. Set
\begin{align}\label{Eq;4}
B_{\mathbf{g}, \mathbf{h}, \mathbf{i}}=\sum_{\mathbf{j}\in\mathbb{P}_{\mathbf{g}, \mathbf{h},\mathbf{i}}}E_\mathbf{g}^*A_\mathbf{j}E_\mathbf{i}^*.
\end{align}
As $\{\mathbf{h}\}\subseteq\mathbb{P}_{\mathbf{i}, \mathbf{h}, \mathbf{g}}=\mathbb{P}_{\mathbf{g}, \mathbf{h}, \mathbf{i}}\subseteq\mathbb{P}_{\mathbf{g}, \mathbf{i}}=\mathbb{P}_{\mathbf{i}, \mathbf{g}}$, notice that $B_{\mathbf{i}, \mathbf{h}, \mathbf{g}}$ is defined by Equation \eqref{Eq;4}.
\end{nota}
\begin{lem}\label{L;Lemma3.10}
Assume that $\mathbf{g}, \mathbf{h}, \mathbf{i}\in\mathbb{E}$ and $\mathbf{h}\in\mathbb{P}_{\mathbf{g}, \mathbf{i}}$. Then $B_{\mathbf{g}, \mathbf{h}, \mathbf{i}}^T\!=\!B_{\mathbf{i}, \mathbf{h}, \mathbf{g}}\neq O$.
\end{lem}
\begin{proof}
As $\{\mathbf{h}\}\subseteq\mathbb{P}_{\mathbf{g}, \mathbf{h},\mathbf{i}}\subseteq \mathbb{P}_{\mathbf{g},\mathbf{i}}$, Equation \eqref{Eq;4} and Theorem \ref{T;Theorem3.5} yield the inequality $B_{\mathbf{g}, \mathbf{h}, \mathbf{i}}\neq O$. Therefore the desired lemma follows from Equations \eqref{Eq;4} and \eqref{Eq;1}.
\end{proof}
\begin{lem}\label{L;Lemma3.11}
Assume that $\mathbf{g}, \mathbf{h}, \mathbf{i}, \mathbf{j}, \mathbf{k}, \mathbf{l}\in\mathbb{E}$, $\mathbf{h}\in\mathbb{P}_{\mathbf{g}, \mathbf{i}}$, $\mathbf{k}\in\mathbb{P}_{\mathbf{j},\mathbf{l}}$. Then $B_{\mathbf{g}, \mathbf{h}, \mathbf{i}}=B_{\mathbf{j}, \mathbf{k},\mathbf{l}}$ if and only if $\mathbf{g}=\mathbf{j}$, $\mathbf{h}=\mathbf{k}$, and $\mathbf{i}=\mathbf{l}$.
\end{lem}
\begin{proof}
It suffices to check that $B_{\mathbf{g}, \mathbf{h}, \mathbf{i}}=B_{\mathbf{j}, \mathbf{k},\mathbf{l}}$ only if $\mathbf{g}=\mathbf{j}$, $\mathbf{h}=\mathbf{k}$, $\mathbf{i}=\mathbf{l}$. The desired lemma follows from combining Equations \eqref{Eq;4}, \eqref{Eq;2}, Theorem \ref{T;Theorem3.5}, Lemma \ref{L;Lemma2.3}.
\end{proof}
\begin{lem}\label{L;Lemma3.12}
$\mathbb{T}$ has an $\F$-linearly independent subset $\{B_{\mathbf{a},\mathbf{b},\mathbf{c}}: (\mathbf{a}, \mathbf{b}, \mathbf{c})\in\mathbb{P}\}$.
\end{lem}
\begin{proof}
Set $\mathbb{U}=\{B_{\mathbf{a},\mathbf{b},\mathbf{c}}: (\mathbf{a}, \mathbf{b}, \mathbf{c})\in\mathbb{P}\}$. Lemma \ref{L;Lemma3.10} implies that $M\neq O$ for any $M\in\mathbb{U}$. Let $L$ be a nonzero $\F$-linear combination of the matrices in $\mathbb{U}$. Assume that $L=O$. If $M\in\mathbb{U}$, let $c_M$ be the coefficient of $M$ in $L$. Then there is $N\in\mathbb{U}$ such that $c_N\in\F^\times$.
Equations \eqref{Eq;4} and \eqref{Eq;2} show that $N=E_\mathbf{g}^*NE_\mathbf{h}^*$ for some $\mathbf{g}, \mathbf{h}\in\mathbb{E}$. So $\mathbb{V}=\{A: A\!\in\!\mathbb{U}, c_A\!\in\!\F^\times, A\!=\!E_\mathbf{g}^*AE_\mathbf{h}^*\}\neq\varnothing$. By Lemma \ref{L;Lemma3.11}, there are $i\!\in\!\mathbb{N}$ and pairwise distinct $\mathbf{j}_{(1)}, \mathbf{j}_{(2)}, \ldots, \mathbf{j}_{(i)}\!\in\!\mathbb{P}_{\mathbf{g}, \mathbf{h}}$ such that $\mathbb{V}\!=\!\{B_{\mathbf{g}, \mathbf{j}_{(1)}, \mathbf{h}}, B_{\mathbf{g}, \mathbf{j}_{(2)}, \mathbf{h}}, \ldots, B_{\mathbf{g}, \mathbf{j}_{(i)}, \mathbf{h}}\}$.

Assume that $i=1$. Equations \eqref{Eq;4} and \eqref{Eq;2} give $O=E_\mathbf{g}^*LE_\mathbf{h}^*=cB_{\mathbf{g}, \mathbf{j}_{(1)}, \mathbf{h}}$ for some $c\in\F^\times$. This is a contradiction by Lemma \ref{L;Lemma3.10}. Assume further that $i>1$. By Lemma \ref{L;Lemma2.3}, there is no loss to assume that $\mathbf{j}_{(1)}$ is maximal in $\{\mathbf{j}_{(2)}, \mathbf{j}_{(3)}, \ldots, \mathbf{j}_{(i)}\}$ with respect to the partial order $\preceq$. As $E_\mathbf{g}^*LE_\mathbf{h}^*=O$, $B_{\mathbf{g}, \mathbf{j}_{(1)}, \mathbf{h}}$ is an $\F$-linear combination of the matrices in $\{B_{\mathbf{g}, \mathbf{j}_{(2)},\mathbf{h}}, B_{\mathbf{g}, \mathbf{j}_{(3)},\mathbf{h}}, \ldots, B_{\mathbf{g}, \mathbf{j}_{(i)},\mathbf{h}}\}$. So $\mathbf{j}_{(1)}\!\preceq\!\mathbf{k}\!\in\!\{\mathbf{j}_{(2)}, \mathbf{j}_{(3)}, \ldots, \mathbf{j}_{(i)}\}$ by Equation \eqref{Eq;4} and Theorem \ref{T;Theorem3.5}. So $\mathbf{j}_{(1)}\!\!\in\!\!\{\mathbf{j}_{(2)}, \mathbf{j}_{(3)}, \ldots, \mathbf{j}_{(i)}\}$ by the choice of $\mathbf{j}_{(1)}$. This is a contradiction. These contradictions give $L\!\neq\!O$. The desired lemma follows.
\end{proof}
We are now ready to present the second $\F$-basis of $\mathbb{T}$ and an additional notation.
\begin{thm}\label{T;Theorem3.13}
$\mathbb{T}$ has an $\F$-basis $\{B_{\mathbf{a},\mathbf{b},\mathbf{c}}: (\mathbf{a}, \mathbf{b}, \mathbf{c})\in\mathbb{P}\}$ with cardinality $2^{2n_1}5^{n_2}$.
\end{thm}
\begin{proof}
The desired theorem follows from applying Theorem \ref{T;Theorem3.5} and Lemma \ref{L;Lemma3.12}.
\end{proof}
\begin{nota}\label{N;Notation3.14}
Let $\mathbb{B}_2=\{B_{\mathbf{a},\mathbf{b},\mathbf{c}}: (\mathbf{a}, \mathbf{b}, \mathbf{c})\in\mathbb{P}\}$. If $M\in\mathbb{T}$ and $(\mathbf{g}, \mathbf{h}, \mathbf{i})\in\mathbb{P}$, notice that $\mathrm{Supp}_{\mathbb{B}_2}(M)$ is defined and write $c_{\mathbf{g},\mathbf{h},\mathbf{i}}(M)$ for $c_{\mathbb{B}_2}(M, \ B_{\mathbf{g}, \mathbf{h}, \mathbf{i}})$ by Theorem \ref{T;Theorem3.13}.
\end{nota}
We next investigate the structure constants of $\mathbb{B}_2$ in $\mathbb{T}$. We begin with a notation.
\begin{nota}\label{N;Notation3.15}
Assume that $\mathbf{g}, \mathbf{h}, \mathbf{i}, \mathbf{j}, \mathbf{k}, \mathbf{l}\in\mathbb{E}$. If $\mathbbm{l}=\mathbbm{g}\cup\mathbbm{h}$, set $\mathbf{g}\cup\mathbf{h}=\mathbf{l}$ by Lemma \ref{L;Lemma2.3}. If $\mathbbm{l}=\mathbbm{g}\cap\mathbbm{h}$, set $\mathbf{g}\cap\mathbf{h}=\mathbf{l}$ by Lemma \ref{L;Lemma2.3}. So $\mathbf{g}\cup\mathbf{h}=\mathbf{h}\cup\mathbf{g}$ and $\mathbf{g}\cap\mathbf{h}=\mathbf{h}\cap\mathbf{g}$ by Lemma \ref{L;Lemma2.3}. According to Lemma \ref{L;Lemma2.3}, $\mathbf{g}\cap\mathbf{h}\cap\mathbf{i}=(\mathbf{g}\cap\mathbf{h})\cap\mathbf{i}=\mathbf{g}\cap(\mathbf{h}\cap\mathbf{i})$ and
$[\mathbf{g}, \mathbf{h}, \mathbf{i}, \mathbf{j}, \mathbf{k}]=\mathbf{l}$ if $\mathbbm{l}=(\mathbbm{g}\triangle\mathbbm{k})\cup((\mathbbm{g}\cap\mathbbm{k})^\circ\setminus\mathbbm{i})\cup
((\mathbbm{h}\cup\mathbbm{j})\cap(\mathbbm{g}\cap\mathbbm{i}\cap\mathbbm{k})^\circ)$. For example, if $n=|\mathbb{U}_1|=2$, $|\mathbb{U}_2|=3$, and $\mathbf{g}=(0,1)$, then $\mathbf{g}\cup\mathbf{1}=[\mathbf{g},\mathbf{1},\mathbf{1},\mathbf{g},\mathbf{1}]=\mathbf{1}$ and $\mathbf{g}\cap\mathbf{1}=\mathbf{g}$.
\end{nota}
\begin{lem}\label{L;Lemma3.16}
Assume that $\mathbf{g}, \mathbf{h}, \mathbf{i}, \mathbf{j}, \mathbf{k}\!\in\!\mathbb{E}$. Then $[\mathbf{g}, \mathbf{h}, \mathbf{i}, \mathbf{j}, \mathbf{k}]\!\in\!\mathbb{P}_{\mathbf{g},\mathbf{k}}$. Moreover, if $\mathbf{l}, \mathbf{m}, \mathbf{q}\!\in\!\mathbb{E}$,
$\mathbf{h}\!\in\!\mathbb{P}_{\mathbf{g},\mathbf{i}}$, $\mathbf{j}\!\in\!\mathbb{P}_{\mathbf{i}, \mathbf{k}}$, $\mathbf{l}\!\in\!\mathbb{P}_{\mathbf{g}, \mathbf{h}, \mathbf{i}}$, $\mathbf{m}\!\in\!\mathbb{P}_{\mathbf{i},\mathbf{j},\mathbf{k}}$, $\mathbf{q}\!\in\!\mathbb{P}_{\mathbf{g},\mathbf{k}}\cap\mathbb{P}_{\mathbf{l},\mathbf{m}}$, then $\mathbf{q}\!\in\!\mathbb{P}_{\mathbf{g}, [\mathbf{g}, \mathbf{h}, \mathbf{i}, \mathbf{j}, \mathbf{k}], \mathbf{k}}$.
\end{lem}
\begin{proof}
As $(\mathbbm{g}\triangle\mathbbm{k})\cup((\mathbbm{g}\cap\mathbbm{k})^\circ\setminus\mathbbm{i})\cup
((\mathbbm{h}\cup\mathbbm{j})\cap(\mathbbm{g}\cap\mathbbm{i}\cap\mathbbm{k})^\circ)\subseteq(\mathbbm{g}\triangle\mathbbm{k})
\cup(\mathbbm{g}\cap\mathbbm{k})^\circ$, the first statement follows. Hence $\mathbb{P}_{\mathbf{g}, [\mathbf{g}, \mathbf{h}, \mathbf{i}, \mathbf{j}, \mathbf{k}], \mathbf{k}}$ is defined. As $\mathbf{l}\in\mathbb{P}_{\mathbf{g}, \mathbf{h}, \mathbf{i}}$, $\mathbf{m}\!\in\!\mathbb{P}_{\mathbf{i},\mathbf{j},\mathbf{k}}$, $\mathbf{q}\in\mathbb{P}_{\mathbf{l},\mathbf{m}}$, notice that
$\mathbbm{q}\cap(\mathbbm{g}\cap\mathbbm{i}\cap\mathbbm{k})^\circ\subseteq(\mathbbm{l}\cup\mathbbm{m})\cap(\mathbbm{g}
\cap\mathbbm{i}\cap\mathbbm{k})^\circ\subseteq(\mathbbm{h}\cup\mathbbm{j})\cap(\mathbbm{g}\cap\mathbbm{i}\cap\mathbbm{k})^\circ$. Therefore $\mathbbm{q}\cap(\mathbbm{g}\cap\mathbbm{k})^\circ\subseteq((\mathbbm{g}\cap\mathbbm{k})^\circ\setminus\mathbbm{i})
\cup((\mathbbm{h}\cup\mathbbm{j})\cap(\mathbbm{g}\cap\mathbbm{i}\cap\mathbbm{k})^\circ)$. As $\mathbf{q}\in \mathbb{P}_{\mathbf{g}, \mathbf{k}}$, $\mathbbm{g}\triangle\mathbbm{k}\subseteq\mathbbm{q}$ and
$\mathbbm{q}\subseteq(\mathbbm{g}\triangle\mathbbm{k})\cup((\mathbbm{g}\cap\mathbbm{k})^\circ\setminus\mathbbm{i})\cup
((\mathbbm{h}\cup\mathbbm{j})\cap(\mathbbm{g}\cap\mathbbm{i}\cap\mathbbm{k})^\circ)$. The desired lemma follows.
\end{proof}
\begin{lem}\label{L;Lemma3.17}
Assume that $\mathbf{g}, \mathbf{h},\mathbf{i}, \mathbf{j}, \mathbf{k}\in\mathbb{E}$, $\mathbf{h}\!\in\!\mathbb{P}_{\mathbf{g}, \mathbf{i}}$, $\mathbf{j}\!\in\!\mathbb{P}_{\mathbf{i}, \mathbf{k}}$. Then $[\mathbf{g}, \mathbf{h}, \mathbf{i}, \mathbf{j}, \mathbf{k}]\!\in\!\mathbb{P}_{\mathbf{g}, \mathbf{k}}$ and
$$\mathrm{Supp}_{\mathbb{B}_1}(B_{\mathbf{g}, \mathbf{h}, \mathbf{i}}B_{\mathbf{i}, \mathbf{j}, \mathbf{k}})\subseteq\{E_\mathbf{g}^*A_\mathbf{a}E_\mathbf{k}^*: \mathbf{a}\in\mathbb{P}_{\mathbf{g}, [\mathbf{g}, \mathbf{h}, \mathbf{i}, \mathbf{j}, \mathbf{k}], \mathbf{k}}\}=\mathrm{Supp}_{\mathbb{B}_1}(B_{\mathbf{g}, [\mathbf{g}, \mathbf{h}, \mathbf{i}, \mathbf{j}, \mathbf{k}], \mathbf{k}}).$$
\end{lem}
\begin{proof}
The leftmost containment follows from combining Equation \eqref{Eq;4}, Lemmas \ref{L;Lemma3.7}, and \ref{L;Lemma3.16}.
The desired lemma follows from Equation \eqref{Eq;4} and Theorem \ref{T;Theorem3.5}.
\end{proof}
\begin{nota}\label{N;Notation3.18}
Assume that $\mathbf{y}, \mathbf{z}\in\mathbb{X}$ and $\mathbf{g}, \mathbf{h}\in\mathbb{E}$. Notice that $[1, n]$ is a union of the pairwise disjoint subsets $\mathbbm{g}\cap\mathbbm{h}$, $\mathbbm{g}\setminus \mathbbm{h}$, and $[1,n]\setminus\mathbbm{g}$. According to Lemma \ref{L;Lemma2.3}, there is a unique $\mathbf{w}\in\mathbb{X}$ such that $\mathbf{w}_i=\mathbf{x}_i$, $\mathbf{w}_j=\mathbf{z}_j$, and $\mathbf{w}_k=\mathbf{y}_k$ for any $i\in\mathbbm{g}\cap\mathbbm{h}$, $j\in\mathbbm{g}\setminus\mathbbm{h}$, and $k\in[1,n]\setminus\mathbbm{g}$. This unique element $\mathbf{w}$ in $\mathbb{X}$
is denoted by $[\mathbf{g}, \mathbf{h}; \mathbf{y}, \mathbf{z}]$.
\end{nota}
\begin{lem}\label{L;Lemma3.19}
Assume that $\mathbf{y}, \mathbf{z}\in\mathbb{X}$ and $\mathbf{g}, \mathbf{h}, \mathbf{i}, \mathbf{j}, \mathbf{k}, \mathbf{l}, \mathbf{m}\in\mathbb{E}$.
Assume that $\mathbf{h}\in\mathbb{P}_{\mathbf{g}, \mathbf{i}}$, $\mathbf{j}\in\mathbb{P}_{\mathbf{i}, \mathbf{k}}$, $\mathbf{l}\in\mathbb{P}_{\mathbf{g}, \mathbf{h}, \mathbf{i}}$, $\mathbf{m}\in\mathbb{P}_{\mathbf{i},\mathbf{j},\mathbf{k}}$.
Then $\mathbf{y}R_\mathbf{l}\cap\mathbf{x}R_\mathbf{i}\cap\mathbf{z}R_\mathbf{m}\subseteq[\mathbf{h}, \mathbf{j}; \mathbf{y}, \mathbf{z}]R_{\mathbf{h}\cap\mathbf{i}\cap\mathbf{j}}$.
\end{lem}
\begin{proof}
Assume that $\mathbf{y}R_\mathbf{l}\!\cap\!\mathbf{x}R_\mathbf{i}\cap\mathbf{z}R_\mathbf{m}\neq\varnothing$. Pick $\mathbf{w}\in\mathbf{y}R_\mathbf{l}\!\cap\!\mathbf{x}R_\mathbf{i}\cap\mathbf{z}R_\mathbf{m}$. If $q\in(\mathbbm{h}\cap\mathbbm{j})\setminus\mathbbm{i}$, then $\mathbf{w}_q=\mathbf{x}_q=[\mathbf{h},\mathbf{j}; \mathbf{y}, \mathbf{z}]_q$. If $q\in\mathbbm{h}\cap\mathbbm{i}\cap\mathbbm{j}$, then $\mathbf{w}_q\!\neq\! \mathbf{x}_q\!=\![\mathbf{h},\mathbf{j}; \mathbf{y}, \mathbf{z}]_q$. If $q\in\mathbbm{h}\setminus\mathbbm{j}$, then $\mathbf{w}_q\!=\!\mathbf{z}_q\!=\![\mathbf{h},\mathbf{j}; \mathbf{y}, \mathbf{z}]_q$ as $\mathbf{m}\!\in\!\mathbb{P}_{\mathbf{i},\mathbf{j},\mathbf{k}}$. If $q\in[1,n]\setminus\mathbbm{h}$, then $\mathbf{w}_q=\mathbf{y}_q=[\mathbf{h},\mathbf{j}; \mathbf{y}, \mathbf{z}]_q$ as $\mathbf{l}\in\mathbb{P}_{\mathbf{g}, \mathbf{h}, \mathbf{i}}$. Therefore $\mathbf{w}\in[\mathbf{h}, \mathbf{j}; \mathbf{y}, \mathbf{z}]R_{\mathbf{h}\cap\mathbf{i}\cap\mathbf{j}}$. The desired lemma follows.
\end{proof}
\begin{lem}\label{L;Lemma3.20}
Assume that $\mathbf{y}, \mathbf{z}\in\mathbb{X}$, $\mathbf{g}, \mathbf{h}, \mathbf{i}, \mathbf{j}, \mathbf{k}\in\mathbb{E}$,
$\mathbf{x}\in\mathbf{y}R_\mathbf{g}\cap\mathbf{z}R_\mathbf{k}$. Assume that $\mathbf{h}\in\mathbb{P}_{\mathbf{g}, \mathbf{i}}$ and $\mathbf{j}\in\mathbb{P}_{\mathbf{i}, \mathbf{k}}$. Then $[\mathbf{h}, \mathbf{j};\mathbf{y},\mathbf{z}]R_{\mathbf{h}\cap\mathbf{i}\cap\mathbf{j}}\subseteq\mathbf{x}R_\mathbf{i}$.
\end{lem}
\begin{proof}
Pick $\mathbf{w}\in[\mathbf{h}, \mathbf{j};\mathbf{y},\mathbf{z}]R_{\mathbf{h}\cap\mathbf{i}\cap\mathbf{j}}$.  If $q\in(\mathbbm{h}\cap\mathbbm{j})\setminus\mathbbm{i}$, then $\mathbf{w}_q=[\mathbf{h}, \mathbf{j}; \mathbf{y}, \mathbf{z}]_q=\mathbf{x}_q$. If $q\in\mathbbm{h}\cap\mathbbm{i}\cap\mathbbm{j}$, then $\mathbf{w}_q\neq[\mathbf{h}, \mathbf{j}; \mathbf{y}, \mathbf{z}]_q=\mathbf{x}_q$. As $\mathbf{j}\in\mathbb{P}_{\mathbf{i}, \mathbf{k}}$, it is obvious to see that $\mathbbm{i}\triangle\mathbbm{k}\subseteq\mathbbm{j}$.
If $q\in(\mathbbm{h}\cap\mathbbm{k})\setminus\mathbbm{j}$, then $\mathbf{w}_q=[\mathbf{h}, \mathbf{j}; \mathbf{y},\mathbf{z}]_q=\mathbf{z}_q\neq \mathbf{x}_q$. If $q\in\mathbbm{h}\setminus(\mathbbm{j}\cup\mathbbm{k})$, then $\mathbf{w}_q=[\mathbf{h}, \mathbf{j}; \mathbf{y},\mathbf{z}]_q=\mathbf{z}_q=\mathbf{x}_q$. As $\mathbf{h}\in\mathbb{P}_{\mathbf{g}, \mathbf{i}}$, it is obvious to see that $\mathbbm{g}\triangle\mathbbm{i}\subseteq\mathbbm{h}$.
If $q\in\mathbbm{g}\setminus\mathbbm{h}$, then $\mathbf{w}_q\!=\![\mathbf{h}, \mathbf{j}; \mathbf{y},\mathbf{z}]_q\!=\!\mathbf{y}_q\!\neq\! \mathbf{x}_q$. For the remaining case $q\in[1,n]\setminus(\mathbbm{g}\cup\mathbbm{h})$, $\mathbf{w}_q=[\mathbf{h}, \mathbf{j}; \mathbf{y},\mathbf{z}]_q=\mathbf{y}_q=\mathbf{x}_q$. Therefore $\mathbf{w}\in\mathbf{x}R_{\mathbf{i}}$. The desired lemma follows.
\end{proof}
\begin{lem}\label{L;Lemma3.21}
Assume that $\mathbf{w}, \mathbf{y}, \mathbf{z}\in\mathbb{X}$, $\mathbf{g}, \mathbf{h}, \mathbf{i}, \mathbf{j}, \mathbf{k}\in\mathbb{E}$,
$\mathbf{x}\in\mathbf{y}R_\mathbf{g}\cap\mathbf{z}R_\mathbf{k}$. Assume that $\mathbf{w}\!\in\![\mathbf{h}, \mathbf{j};\mathbf{y},\mathbf{z}]R_{\mathbf{h}\cap\mathbf{i}\cap\mathbf{j}}$, $\mathbf{h}\!\in\!\mathbb{P}_{\mathbf{g}, \mathbf{i}}$, $\mathbf{j}\!\in\!\mathbb{P}_{\mathbf{i}, \mathbf{k}}$. Then there is $\mathbf{l}\!\in\!\mathbb{P}_{\mathbf{g}, \mathbf{h}, \mathbf{i}}$ such that $\mathbf{w}\in\mathbf{y}R_\mathbf{l}$.
\end{lem}
\begin{proof}
By Lemma \ref{L;Lemma2.3}, let $\mathbf{l}\!\in\!\mathbb{E}$ and $\mathbbm{l}\!=\!(\mathbbm{g}\triangle\mathbbm{i})\cup(\mathbbm{g}\cap\mathbbm{h}\cap\mathbbm{i}\cap\{a: a\!\in\! [1, n], \mathbf{w}_a\!\neq\!\mathbf{y}_a\})$. So $\mathbf{l}\in\mathbb{P}_{\mathbf{g}, \mathbf{h}, \mathbf{i}}$ as $\mathbf{h}\!\in\!\mathbb{P}_{\mathbf{g}, \mathbf{i}}$. As $\mathbf{h}\in\mathbb{P}_{\mathbf{g}, \mathbf{i}}$,
notice that $\mathbbm{g}\triangle\mathbbm{i}\subseteq\mathbbm{h}\subseteq\mathbbm{g}\cup\mathbbm{i}$. If
$q\in(\mathbbm{h}\cap\mathbbm{j})\setminus\mathbbm{g}$, then $\mathbf{w}_q\neq[\mathbf{h},\mathbf{j};\mathbf{y}, \mathbf{z}]_q=\mathbf{x}_q=\mathbf{y}_q$. If $q\in(\mathbbm{g}\cap\mathbbm{h}\cap\mathbbm{j})\setminus\mathbbm{i}$,
then $\mathbf{w}_q=[\mathbf{h},\mathbf{j};\mathbf{y}, \mathbf{z}]_q=\mathbf{x}_q\neq \mathbf{y}_q$. If $q\in\mathbbm{g}\cap\mathbbm{h}\cap\mathbbm{i}\cap\mathbbm{j}$, then $\mathbf{w}_q\neq[\mathbf{h},\mathbf{j};\mathbf{y}, \mathbf{z}]_q=\mathbf{x}_q\neq \mathbf{y}_q$. As $\mathbf{j}\in\mathbb{P}_{\mathbf{i}, \mathbf{k}}$, notice that
$\mathbbm{i}\triangle\mathbbm{k}\subseteq\mathbbm{j}$. If $q\!\in\!\mathbbm{h}\setminus(\mathbbm{j}\cup\mathbbm{k})$, then
$\mathbf{w}_q\!=\![\mathbf{h}, \mathbf{j}; \mathbf{y}, \mathbf{z}]_q=\mathbf{z}_q=\mathbf{x}_q\neq\mathbf{y}_q$. If $q\in(\mathbbm{h}\cap\mathbbm{k})\setminus(\mathbbm{g}\cup\mathbbm{j})$, then $\mathbf{w}_q\!\!=\!\![\mathbf{h}, \mathbf{j}; \mathbf{y}, \mathbf{z}]_q\!\!=\!\!\mathbf{z}_q\!\!\neq\!\!\mathbf{x}_q\!\!=\!\!\mathbf{y}_q$. If $q\!\in\!(\mathbbm{g}\!\cap\!\mathbbm{h}\!\cap\!\mathbbm{k})\!\setminus\!\mathbbm{j}$,
then $\mathbf{w}_q\!\!=\!\![\mathbf{h}, \mathbf{j}; \mathbf{y}, \mathbf{z}]_q\!\!=\!\!\mathbf{z}_q\!\!\neq\!\!\mathbf{x}_q\!\!\neq\!\!\mathbf{y}_q$. If
$q\!\in\![1,n]\!\setminus\!\mathbbm{h}$, then $\mathbf{w}_q\!=\![\mathbf{h},\mathbf{j};\mathbf{y},\mathbf{z}]_q\!=\!\mathbf{y}_q$.
So $\mathbf{w}\!\in\!\mathbf{y}R_\mathbf{l}$. The desired lemma follows.
\end{proof}
\begin{lem}\label{L;Lemma3.22}
Assume that $\mathbf{g}, \mathbf{h}, \mathbf{i}, \mathbf{j}, \mathbf{k}\in\mathbb{E}$. Then $[\mathbf{g}, \mathbf{h}, \mathbf{i}, \mathbf{j}, \mathbf{k}]\in\mathbb{P}_{\mathbf{g}, \mathbf{k}}$. Moreover, if $\mathbf{w},\mathbf{y}, \mathbf{z}\in\mathbb{X}$,
$\mathbf{x}\in\mathbf{y}R_\mathbf{g}\cap\mathbf{z}R_\mathbf{k}$, $\mathbf{w}\in [\mathbf{h}, \mathbf{j};\mathbf{y}, \mathbf{z}]R_{\mathbf{h}\cap\mathbf{i}\cap\mathbf{j}}$, $\mathbf{h}\in\mathbb{P}_{\mathbf{g}, \mathbf{i}}$, $\mathbf{j}\in\mathbb{P}_{\mathbf{i}, \mathbf{k}}$, and there is $\mathbf{l}\in\mathbb{P}_{\mathbf{g}, [\mathbf{g}, \mathbf{h}, \mathbf{i}, \mathbf{j}, \mathbf{k}], \mathbf{k}}$ such that $\mathbf{z}\in\mathbf{y}R_\mathbf{l}$, then there is also $\mathbf{m}\in\mathbb{P}_{\mathbf{i}, \mathbf{j}, \mathbf{k}}$ such that $\mathbf{w}\in\mathbf{z}R_\mathbf{m}$.
\end{lem}
\begin{proof}
By Lemma \ref{L;Lemma2.3}, let $\mathbf{m}\!\in\!\mathbb{E}$ and $\mathbbm{m}\!=\!(\mathbbm{i}\triangle\mathbbm{k})\!\cup\!(\mathbbm{i}\cap\mathbbm{j}\cap\mathbbm{k}\cap\{a: a\!\in\![1, n], \mathbf{w}_a\!\neq\!\mathbf{z}_a\})$. So $\mathbf{m}\in\mathbb{P}_{\mathbf{i}, \mathbf{j}, \mathbf{k}}$ as $\mathbf{j}\in\mathbb{P}_{\mathbf{i}, \mathbf{k}}$. If $q\in(\mathbbm{h}\cap\mathbbm{i}\cap\mathbbm{j})\setminus\mathbbm{k}$, then
$\mathbf{w}_q\neq[\mathbf{h}, \mathbf{j};\mathbf{y}, \mathbf{z}]_q=\mathbf{x}_q=\mathbf{z}_q$. If $q\in\mathbbm{h}\cap\mathbbm{i}\cap\mathbbm{j}\cap\mathbbm{k}$, then $\mathbf{w}_q\neq[\mathbf{h},\mathbf{j};\mathbf{y},\mathbf{z}]_q=\mathbf{x}_q\neq \mathbf{z}_q$. As $\mathbf{j}\in\mathbb{P}_{\mathbf{i}, \mathbf{k}}$, it is obvious to see that $\mathbbm{i}\triangle\mathbbm{k}\subseteq\mathbbm{j}\subseteq\mathbbm{i}\cup\mathbbm{k}$. If $q\in(\mathbbm{h}\cap\mathbbm{j})\!\setminus\!\mathbbm{i}$, then $\mathbf{w}_q\!=\![\mathbf{h},\mathbf{j};\mathbf{y}, \mathbf{z}]_q=\mathbf{x}_q\!\neq\!\mathbf{z}_q$. If $q\in\mathbbm{h}\setminus\mathbbm{j}$, then $\mathbf{w}_q=[\mathbf{h}, \mathbf{j}; \mathbf{y}, \mathbf{z}]_q=\mathbf{z}_q$. As $\mathbf{h}\in\mathbb{P}_{\mathbf{g}, \mathbf{i}}$, it is obvious to see that $\mathbbm{g}\triangle\mathbbm{i}\subseteq\mathbbm{h}\subseteq\mathbbm{g}\cup\mathbbm{i}$. As $[\mathbf{g}, \mathbf{h}, \mathbf{i}, \mathbf{j}, \mathbf{k}]\in\mathbb{P}_{\mathbf{g}, \mathbf{k}}$ and $\mathbf{l}\in\mathbb{P}_{\mathbf{g},[\mathbf{g}, \mathbf{h}, \mathbf{i}, \mathbf{j}, \mathbf{k}], \mathbf{k}}$ by Lemma \ref{L;Lemma3.16}, it is obvious to see that $\mathbbm{g}\triangle\mathbbm{k}\subseteq\mathbbm{l}\subseteq(\mathbbm{g}\triangle
\mathbbm{k})\cup((\mathbbm{g}\cap\mathbbm{k})^\circ\setminus\mathbbm{i})\cup((\mathbbm{h}\cup\mathbbm{j})
\cap(\mathbbm{g}\cap\mathbbm{i}\cap\mathbbm{k})^\circ)$. If $q\in\mathbbm{g}\setminus(\mathbbm{h}\cup\mathbbm{k})$, then $\mathbf{w}_q=[\mathbf{h},\mathbf{j};\mathbf{y},\mathbf{z}]_q=\mathbf{y}_q\neq\mathbf{z}_q$. If $q\in(\mathbbm{g}\cap\mathbbm{k})\setminus(\mathbbm{h}\cup\mathbbm{j})$, then $\mathbf{w}_q=[\mathbf{h}, \mathbf{j};\mathbf{y}, \mathbf{z}]_q=\mathbf{y}_q=\mathbf{z}_q$. If $q\in(\mathbbm{g}\cap\mathbbm{j}\cap\mathbbm{k})\setminus\mathbbm{h}$, then
$\mathbf{w}_q=[\mathbf{h},\mathbf{j};\mathbf{y},\mathbf{z}]_q=\mathbf{y}_q\neq\mathbf{x}_q\neq\mathbf{z}_q$. If $q\in\mathbbm{k}\setminus(\mathbbm{g}\cup\mathbbm{h})$, then $\mathbf{w}_q=[\mathbf{h},\mathbf{j};\mathbf{y}, \mathbf{z}]_q=\mathbf{y}_q\neq\mathbf{z}_q$. For the remaining case $q\in [1,n]\setminus(\mathbbm{g}\cup\mathbbm{h}\cup\mathbbm{k})$,
$\mathbf{w}_q=[\mathbf{h}, \mathbf{j};\mathbf{y}, \mathbf{z}]_q=\mathbf{y}_q=\mathbf{z}_q$. Therefore $\mathbf{w}\in\mathbf{z}R_\mathbf{m}$. The desired lemma follows.
\end{proof}
We are now ready to list the structure constants of $\mathbb{B}_2$ in $\mathbb{T}$ and another corollary.
\begin{thm}\label{T;Theorem3.23}
Assume that $\mathbf{g}, \mathbf{h}, \mathbf{i}, \mathbf{j}, \mathbf{k}, \mathbf{l}, \mathbf{m},\mathbf{q}, \mathbf{r}\!\in\!\mathbb{E}$, $\mathbf{h}\in\mathbb{P}_{\mathbf{g}, \mathbf{i}}$, $\mathbf{k}\!\in\!\mathbb{P}_{\mathbf{j}, \mathbf{l}}$, $\mathbf{q}\in\mathbb{P}_{\mathbf{m}, \mathbf{r}}$. Then
$$[\mathbf{g}, \mathbf{h}, \mathbf{i}, \mathbf{k}, \mathbf{l}]\in\mathbb{P}_{\mathbf{g}, \mathbf{l}}\ \text{and}\ c_{\mathbf{m}, \mathbf{q}, \mathbf{r}}(B_{\mathbf{g}, \mathbf{h}, \mathbf{i}}B_{\mathbf{j},\mathbf{k},\mathbf{l}})=\delta_{\mathbf{i}, \mathbf{j}}\delta_{\mathbf{m},\mathbf{g}}\delta_{\mathbf{q}, [\mathbf{g}, \mathbf{h}, \mathbf{i}, \mathbf{k}, \mathbf{l}]}\delta_{\mathbf{r},\mathbf{l}}\overline{k_{\mathbf{h}\cap\mathbf{i}\cap\mathbf{k}}}.$$
\end{thm}
\begin{proof}
Notice that $[\mathbf{g}, \mathbf{h}, \mathbf{i}, \mathbf{k}, \mathbf{l}]\in\mathbb{P}_{\mathbf{g}, \mathbf{l}}$ by Lemma \ref{L;Lemma3.16}. Assume that $\mathbf{i}=\mathbf{j}, \mathbf{m}\!=\!\mathbf{g}, \mathbf{r}\!=\!\mathbf{l}$ by combining Equations \eqref{Eq;4}, \eqref{Eq;2}, Theorem \ref{T;Theorem3.13}. Assume that $\mathbf{q}=[\mathbf{g}, \mathbf{h}, \mathbf{i},\mathbf{k}, \mathbf{l}]$ and $\mathbf{s}\in\mathbb{P}_{\mathbf{g}, \mathbf{q}, \mathbf{l}}$. Then there are $\mathbf{y}, \mathbf{z}\in\mathbb{X}$ such that $\mathbf{x}\in\mathbf{y}R_\mathbf{g}\cap\mathbf{z}R_\mathbf{l}$ and $\mathbf{y}\!\in\!\mathbf{z}R_\mathbf{s}$. Hence %By Equations \eqref{Eq;4} and \eqref{Eq;2}, there is no loss to assume that $\mathbf{i}=\mathbf{j}, \mathbf{m}=\mathbf{g}, \mathbf{r}\!=\!\mathbf{l}$. Assume that $\mathbf{q}=[\mathbf{g}, \mathbf{h}, \mathbf{i},\mathbf{k}, \mathbf{l}]$ and $\mathbf{s}\in\mathbb{P}_{\mathbf{g}, \mathbf{q}, \mathbf{l}}$. Then there are $\mathbf{y}, \mathbf{z}\in\mathbb{X}$ such that $\mathbf{x}\in\mathbf{y}R_\mathbf{g}\cap\mathbf{z}R_\mathbf{l}$ and $\mathbf{y}\in\mathbf{z}R_\mathbf{s}$. Therefore
$$
c^{\mathbf{g}, \mathbf{s}, \mathbf{l}}(B_{\mathbf{g}, \mathbf{h}, \mathbf{i}}B_{\mathbf{i}, \mathbf{k}, \mathbf{l}})=\sum_{\mathbf{t}\in\mathbb{P}_{\mathbf{g}, \mathbf{h}, \mathbf{i}}}\sum_{\mathbf{u}\in\mathbb{P}_{\mathbf{i}, \mathbf{k}, \mathbf{l}}}\overline{p_{\mathbf{t},\mathbf{i},\mathbf{u}}^{\mathbf{g}, \mathbf{s}, \mathbf{l}}}=\sum_{\mathbf{t}\in\mathbb{P}_{\mathbf{g}, \mathbf{h}, \mathbf{i}}}\sum_{\mathbf{u}\in\mathbb{P}_{\mathbf{i}, \mathbf{k}, \mathbf{l}}}\overline{|\mathbf{y}R_\mathbf{t}\cap\mathbf{x}R_\mathbf{i}\cap\mathbf{z}R_\mathbf{u}|}=
\overline{k_{\mathbf{h}\cap\mathbf{i}\cap\mathbf{k}}}
$$
by combining Lemmas \ref{L;Lemma3.7}, \ref{L;Lemma3.19}, \ref{L;Lemma3.20}, \ref{L;Lemma3.21}, \ref{L;Lemma3.22}. By Lemma \ref{L;Lemma3.17} and Equation \eqref{Eq;4}, $B_{\mathbf{g},\mathbf{h}, \mathbf{i}}B_{\mathbf{i}, \mathbf{k}, \mathbf{l}}\!\!=\!\!\overline{k_{\mathbf{h}\cap\mathbf{i}\cap\mathbf{k}}}B_{\mathbf{g}, [\mathbf{g}, \mathbf{h}, \mathbf{i}, \mathbf{k}, \mathbf{l}], \mathbf{l}}$. The desired lemma follows from Theorem \ref{T;Theorem3.13}.
\end{proof}
\begin{cor}\label{C;Corollary3.24}
Assume that $\mathbf{g}, \mathbf{h}, \mathbf{i}, \mathbf{j}\!\in\!\mathbb{E}$. Then the commutative $\F$-subalgebra $E_\mathbf{g}^*\mathbb{T}E_\mathbf{g}^*$ of $\mathbb{T}$ has an $\F$-basis $\{B_{\mathbf{g}, \mathbf{a}, \mathbf{g}}:\  (\mathbf{g}, \mathbf{a}, \mathbf{g})\in\mathbb{P}\}$. Moreover, if $\mathbf{h}, \mathbf{i}, \mathbf{j}\in\mathbb{P}_{\mathbf{g}, \mathbf{g}}$, then
$$\mathbf{h}\cup\mathbf{i}\in\mathbb{P}_{\mathbf{g}, \mathbf{g}}\ \text{and}\ c_{\mathbf{g}, \mathbf{j}, \mathbf{g}}(B_{\mathbf{g}, \mathbf{h}, \mathbf{g}}B_{\mathbf{g}, \mathbf{i}, \mathbf{g}})=\delta_{\mathbf{j}, \mathbf{h}\cup\mathbf{i}}\overline{k_{\mathbf{h}\cap\mathbf{i}}}.$$
\end{cor}
\begin{proof}
The first statement follows from combining Lemmas \ref{L;Lemma2.4}, \ref{L;Lemma2.14}, Theorem \ref{T;Theorem3.13}, Equations \eqref{Eq;4}, and \eqref{Eq;2}. Notice that $[\mathbf{g}, \mathbf{h}, \mathbf{g},\mathbf{i},\mathbf{g}]=\mathbf{h}\cup\mathbf{i}$ by Lemma \ref{L;Lemma2.3}. Therefore $\mathbf{h}\cup\mathbf{i}\in\mathbb{P}_{\mathbf{g},\mathbf{g}}$ by Lemma \ref{L;Lemma3.16}.
As $\mathbf{h}, \mathbf{i}\in\mathbb{P}_{\mathbf{g}, \mathbf{g}}$, notice that $\mathbf{h}\cap\mathbf{g}\cap\mathbf{i}=\mathbf{h}\cap\mathbf{i}$ by Lemma \ref{L;Lemma2.3}. The desired corollary
follows from the above discussion and Theorem \ref{T;Theorem3.23}.
\end{proof}
We conclude this section by presenting an example of Theorems \ref{T;Theorem3.13} and \ref{T;Theorem3.23}.
\begin{eg}\label{E;Example3.25}
Assume that $n=|\mathbb{U}_1|=2$ and $|\mathbb{U}_2|=3$. It is clear that $|\mathbb{E}|=4$. Assume that $\mathbf{g}=(0, 1)$ and $\mathbf{h}=(1,0)$. Hence $\mathbb{E}=\{\mathbf{0}, \mathbf{g}, \mathbf{h}, \mathbf{1}\}$
and $\mathbb{T}$ has an $\F$-basis containing exactly $B_{\mathbf{0},\mathbf{0},\mathbf{0}}$, $B_{\mathbf{0},\mathbf{g},\mathbf{g}}$, $B_{\mathbf{0},\mathbf{h},\mathbf{h}}$, $B_{\mathbf{0},\mathbf{1},\mathbf{1}}$, $B_{\mathbf{g},\mathbf{0},\mathbf{g}}$, $B_{\mathbf{g},\mathbf{g},\mathbf{0}}$, $B_{\mathbf{g},\mathbf{g},\mathbf{g}}$, $B_{\mathbf{g},\mathbf{h},\mathbf{1}}$, $B_{\mathbf{g},\mathbf{1},\mathbf{h}}$, $B_{\mathbf{g},\mathbf{1},\mathbf{1}}$, $B_{\mathbf{h},\mathbf{0},\mathbf{h}}$, $B_{\mathbf{h},\mathbf{g},\mathbf{1}}$, $B_{\mathbf{h},\mathbf{h},\mathbf{0}}$, $B_{\mathbf{h},\mathbf{1},\mathbf{g}}$, $B_{\mathbf{1},\mathbf{0},\mathbf{1}}$, $B_{\mathbf{1},\mathbf{g},\mathbf{h}}$, $B_{\mathbf{1},\mathbf{g},\mathbf{1}}$, $B_{\mathbf{1},\mathbf{h},\mathbf{g}}$, $B_{\mathbf{1},\mathbf{1},\mathbf{0}}$,
$B_{\mathbf{1},\mathbf{1},\mathbf{g}}$ by Theorem \ref{T;Theorem3.13}. According to Theorem \ref{T;Theorem3.23}, it is obvious that $B_{\mathbf{g},\mathbf{1}, \mathbf{1}}B_{\mathbf{1},\mathbf{g}, \mathbf{1}}=\overline{2}B_{\mathbf{g},\mathbf{1}, \mathbf{1}}$.
\end{eg}
\section{Algebraic structure of $\mathbb{T}$: Center}
In this section, we present an $\F$-basis of $\mathrm{Z}(\mathbb{T})$ and the structure constants of this $\F$-basis in $\mathrm{Z}(\mathbb{T})$. We first recall Notations \ref{N;Notation3.1}, \ref{N;Notation3.9}, \ref{N;Notation3.14}, \ref{N;Notation3.15} and give two notations.
\begin{nota}\label{N;Notation4.1}
Assume that $\mathbf{g}, \mathbf{h}, \mathbf{i}\in\mathbb{E}$. If $\mathbbm{i}=\mathbbm{g}\setminus\mathbbm{h}$, set $\mathbf{g}\setminus\mathbf{h}=\mathbf{i}$ by Lemma \ref{L;Lemma2.3}. If $\mathbbm{i}=\mathbbm{g}\triangle\mathbbm{h}$, set $\mathbf{g}\triangle\mathbf{h}=\mathbf{i}$ by Lemma \ref{L;Lemma2.3}. So $\mathbf{g}\triangle \mathbf{h}=\mathbf{h}\triangle \mathbf{g}\in\mathbb{P}_{\mathbf{g}, \mathbf{h}}=\mathbb{P}_{\mathbf{h}, \mathbf{g}}$. For example, if $n=|\mathbb{U}_1|=2$, $|\mathbb{U}_2|=3$, $\mathbf{g}=(0,1)$, and $\mathbf{h}=(1,0)$, then  $\mathbf{1}\setminus\mathbf{g}=\mathbf{g}\triangle\mathbf{1}=\mathbf{h}\setminus\mathbf{0}=\mathbf{h}$.
\end{nota}
\begin{nota}\label{N;Notation4.2}
Assume that $\mathbf{g}\in\mathbb{E}$ and $\mathbf{g}\in\mathbb{P}_{\mathbf{1}, \mathbf{1}}$. Then $\mathbf{g}\cap\mathbf{h}\in\mathbb{P}_{\mathbf{h}, \mathbf{h}}$ if $\mathbf{h}\in\mathbb{E}$. Write
\begin{align}\label{Eq;5}
C_\mathbf{g}=\sum_{\mathbf{h}\in\mathbb{E}}\overline{k_{\mathbf{g}\setminus\mathbf{h}}}B_{\mathbf{h},\mathbf{g}\cap\mathbf{h},\mathbf{h}}.
\end{align}
\end{nota}
The following lemmas investigate the objects contained in Notations \ref{N;Notation4.1} and \ref{N;Notation4.2}.
\begin{lem}\label{L;Lemma4.3}
Assume that $\mathbf{g}, \mathbf{h}, \mathbf{i}, \mathbf{j}\in\mathbb{E}$ and $\mathbf{g}\triangle\mathbf{i}\preceq\mathbf{h}$. Then $k_{\mathbf{j}\setminus\mathbf{i}}k_{\mathbf{h}\cap\mathbf{i}\cap\mathbf{j}}=k_{\mathbf{j}\setminus\mathbf{g}}k_{\mathbf{g}\cap\mathbf{h}\cap\mathbf{j}}$.
\end{lem}
\begin{proof}
Recall that $k_\mathbf{0}=1$ and $k_\mathbf{g}=\prod_{k\in\mathbbm{g}}(|\mathbb{U}_k|-1)$ if $\mathbf{g}\in\mathbb{E}\setminus\{\mathbf{0}\}$. It is obvious that  $(\mathbbm{j}\setminus\mathbbm{i})\cap(\mathbbm{h}\cap\mathbbm{i}\cap\mathbbm{j})=
(\mathbbm{j}\setminus\mathbbm{g})\cap(\mathbbm{g}\cap\mathbbm{h}\cap\mathbbm{j})=\varnothing$. Check
$(\mathbbm{j}\setminus\mathbbm{i})\cup(\mathbbm{h}\cap\mathbbm{i}\cap\mathbbm{j})=
(\mathbbm{j}\setminus\mathbbm{g})\cup(\mathbbm{g}\cap\mathbbm{h}\cap\mathbbm{j})$. As $\mathbbm{g}\triangle\mathbbm{i}\subseteq\mathbbm{h}$, $(\mathbbm{j}\setminus\mathbbm{i})\cup(\mathbbm{h}\cap\mathbbm{i}\cap\mathbbm{j})=(\mathbbm{j}\setminus(\mathbbm{g}\cup\mathbbm{i}))\cup
((\mathbbm{g}\cap\mathbbm{j})\setminus\mathbbm{i})\cup((\mathbbm{i}\cap\mathbbm{j})\setminus\mathbbm{g})\cup
(\mathbbm{g}\cap\mathbbm{h}\cap\mathbbm{i}\cap\mathbbm{j})$. Hence $(\mathbbm{j}\setminus\mathbbm{i})\cup(\mathbbm{h}\cap\mathbbm{i}\cap\mathbbm{j})=(\mathbbm{j}\setminus\mathbbm{g})\cup
((\mathbbm{g}\cap\mathbbm{j})\setminus\mathbbm{i})\cup(\mathbbm{g}\cap\mathbbm{h}\cap\mathbbm{i}\cap\mathbbm{j})=
(\mathbbm{j}\setminus\mathbbm{g})\cup(\mathbbm{g}\cap\mathbbm{h}\cap\mathbbm{j})$ as $\mathbf{g}\triangle\mathbf{i}\preceq\mathbf{h}$. The desired lemma follows from the above discussion and Lemma \ref{L;Lemma2.3}.
\end{proof}
\begin{lem}\label{L;Lemma4.4}
Assume that $\mathbf{g}, \mathbf{h}, \mathbf{i}, \mathbf{j}\in\mathbb{E}$. Then $[\mathbf{g}, \mathbf{h}, \mathbf{i}, \mathbf{i}\cap\mathbf{j}, \mathbf{i}]=[\mathbf{g},\mathbf{g}\cap\mathbf{j},\mathbf{g}, \mathbf{h}, \mathbf{i}]$. Moreover, if $\mathbf{h}\in\mathbb{P}_{\mathbf{g}, \mathbf{i}}$ and $\mathbf{j}\in\mathbb{P}_{\mathbf{1},\mathbf{1}}$, then $\overline{k_{\mathbf{j}\setminus \mathbf{i}}}B_{\mathbf{g},\mathbf{h},\mathbf{i}}B_{\mathbf{i}, \mathbf{i}\cap \mathbf{j}, \mathbf{i}}=\overline{k_{\mathbf{j}\setminus \mathbf{g}}}B_{\mathbf{g}, \mathbf{g}\cap \mathbf{j},\mathbf{g}}B_{\mathbf{g}, \mathbf{h}, \mathbf{i}}$.
\end{lem}
\begin{proof}
Notice that $(\mathbbm{g}\triangle\mathbbm{i})\cup((\mathbbm{h}\cup(\mathbbm{i}\cap\mathbbm{j}))\cap(\mathbbm{g}\cap\mathbbm{i})^\circ)=
(\mathbbm{g}\triangle\mathbbm{i})\cup(((\mathbbm{g}\cap\mathbbm{j})\cup\mathbbm{h})\cap(\mathbbm{g}\cap\mathbbm{i})^\circ)$. The first statement thus follows from Lemma \ref{L;Lemma2.3}. By Lemma \ref{L;Lemma3.16}, it is obvious to see that $[\mathbf{g}, \mathbf{h}, \mathbf{i}, \mathbf{i}\cap\mathbf{j}, \mathbf{i}]\!=\![\mathbf{g},\mathbf{g}\cap\mathbf{j},\mathbf{g}, \mathbf{h}, \mathbf{i}]\!\in\!\mathbb{P}_{\mathbf{g}, \mathbf{i}}$. As $\mathbf{h}\in\mathbb{P}_{\mathbf{g}, \mathbf{i}}$, notice that
$\mathbbm{g}\triangle\mathbbm{i}\subseteq\mathbbm{h}$ and
\begin{align*}
\overline{k_{\mathbf{j}\setminus \mathbf{i}}}B_{\mathbf{g},\mathbf{h},\mathbf{i}}B_{\mathbf{i}, \mathbf{i}\cap \mathbf{j}, \mathbf{i}}&=\overline{k_{\mathbf{j}\setminus \mathbf{i}}k_{\mathbf{h}\cap\mathbf{i}\cap\mathbf{j}}}B_{\mathbf{g},[\mathbf{g}, \mathbf{h},\mathbf{i},\mathbf{i}\cap\mathbf{j},\mathbf{i}] ,\mathbf{i}}\\
&=\overline{k_{\mathbf{j}\setminus \mathbf{g}}k_{\mathbf{g}\cap\mathbf{h}\cap\mathbf{j}}}B_{\mathbf{g},[\mathbf{g}, \mathbf{g}\cap\mathbf{j},\mathbf{g},\mathbf{h},\mathbf{i}] ,\mathbf{i}}=\overline{k_{\mathbf{j}\setminus \mathbf{g}}}B_{\mathbf{g}, \mathbf{g}\cap \mathbf{j},\mathbf{g}}B_{\mathbf{g}, \mathbf{h}, \mathbf{i}}
\end{align*}
by an application of Theorem \ref{T;Theorem3.23} and Lemma \ref{L;Lemma4.3}. The desired lemma follows.
\end{proof}
\begin{lem}\label{L;Lemma4.5}
Assume that $\mathbf{g}\in\mathbb{E}$ and $\mathbf{g}\in\mathbb{P}_{\mathbf{1},\mathbf{1}}$. Then $c_{\mathbf{1}, \mathbf{g}, \mathbf{1}}(C_\mathbf{g})=\overline{1}$ and $C_\mathbf{g}\neq O$.
\end{lem}
\begin{proof}
As $\mathbf{g}\in\mathbb{P}_{\mathbf{1},\mathbf{1}}$, notice that $\mathbf{g}\setminus\mathbf{1}=\mathbf{0}$ and $\mathbf{g}\cap\mathbf{1}=\mathbf{g}$ by Lemma \ref{L;Lemma2.3}. Notice that $k_{\mathbf{g}\setminus\mathbf{1}}=1$. The desired lemma follows from Equation \eqref{Eq;5} and Theorem \ref{T;Theorem3.13}.
\end{proof}
\begin{lem}\label{L;Lemma4.6}
Assume that $\mathbf{g}, \mathbf{h}\in\mathbb{E}$ and $\mathbf{g}, \mathbf{h}\in\mathbb{P}_{\mathbf{1}, \mathbf{1}}$. Then $C_\mathbf{\mathbf{g}}=C_\mathbf{h}$ if and only if $\mathbf{g}=\mathbf{h}$. In particular, $\mathrm{Z}(\mathbb{T})$ has an $\F$-linearly independent subset $\{C_\mathbf{a}: (\mathbf{1}, \mathbf{a}, \mathbf{1})\in\mathbb{P}\}$.
\end{lem}
\begin{proof}
Assume that $C_\mathbf{g}=C_\mathbf{h}$. Hence $\{B_{\mathbf{1}, \mathbf{g},\mathbf{1}}, B_{\mathbf{1}, \mathbf{h},\mathbf{1}}\}\subseteq\mathrm{Supp}_{\mathbb{B}_2}(C_\mathbf{g})\cap\mathrm{Supp}_{\mathbb{B}_2}(C_\mathbf{h})$ by Lemma \ref{L;Lemma4.5}.
Hence $\mathbf{g}=\mathbf{h}$ by Equation \eqref{Eq;5} and Theorem \ref{T;Theorem3.13}. The first statement follows. For the remaining statement, set $\mathbb{U}=\{C_\mathbf{a}: (\mathbf{1}, \mathbf{a}, \mathbf{1})\in\mathbb{P}\}$. So $\mathbb{U}\subseteq\mathrm{Z}(\mathbb{T})$ by combining Equation \eqref{Eq;5}, Lemma \ref{L;Lemma4.4}, Theorems \ref{T;Theorem3.23}, and \ref{T;Theorem3.13}.
Moreover, consider
\begin{align}\label{Eq;6}
\sum_{M\in\mathbb{U}}c_MM=O,
\end{align}
where $c_M\!\in\!\F$ if $M\!\in\!\mathbb{U}$. If $M\!\in\! \mathbb{U}$, then $c_M=\overline{0}$ by combining Equations \eqref{Eq;6}, \eqref{Eq;5}, the first statement, Lemma \ref{L;Lemma4.5}, and Theorem \ref{T;Theorem3.13}. The desired lemma follows.
\end{proof}
\begin{lem}\label{L;Lemma4.7}
Assume that $M\in\mathrm{Z}(\mathbb{T})$. Then $\mathrm{Supp}_{\mathbb{B}_2}(M)\subseteq\{B_{\mathbf{a}, \mathbf{b}, \mathbf{a}}: (\mathbf{a}, \mathbf{b}, \mathbf{a})\in\mathbb{P}\}$.
\end{lem}
\begin{proof}
As $M\in\mathrm{Z}(\mathbb{T})$, Theorem \ref{T;Theorem3.13} implies that $M$ is an $\F$-linear combination of the matrices in $\mathbb{B}_2$. Assume that $\mathrm{Supp}_{\mathbb{B}_2}(M)\not\subseteq\{B_{\mathbf{a}, \mathbf{b}, \mathbf{a}}: (\mathbf{a}, \mathbf{b}, \mathbf{a})\in\mathbb{P}\}$. By combining Equations \eqref{Eq;3}, \eqref{Eq;4}, and \eqref{Eq;2}, there are distinct $\mathbf{g}, \mathbf{h}\in\mathbb{E}$ such that $E_\mathbf{g}^*ME_\mathbf{h}^*\neq O$. As $M\in\mathrm{Z}(\mathbb{T})$, notice that $O\neq E_\mathbf{g}^*ME_\mathbf{h}^*=ME_\mathbf{g}^*E_\mathbf{h}^*=O$ by Equation \eqref{Eq;2}. This is an obvious contradiction. The desired lemma follows from this contradiction.
\end{proof}
\begin{lem}\label{L;Lemma4.8}
Assume that $\mathbf{g}\!\in\!\mathbb{E}$ and $\mathbf{g}\in\mathbb{P}_{\mathbf{1},\mathbf{1}}$. Assume that $M\!\in\!\mathrm{Z}(\mathbb{T})$, $c_{\mathbf{1}, \mathbf{g}, \mathbf{1}}(M)\!=\!\overline{1}$, $\mathrm{Supp}_{\mathbb{B}_2}(E_\mathbf{1}^*M)=\{B_{\mathbf{1}, \mathbf{g}, \mathbf{1}}\}$. Then $M=C_\mathbf{g}$.
\end{lem}
\begin{proof}
Pick $\mathbf{h}\in\mathbb{E}$. Then $\mathbf{g}\cap(\mathbf{h}\triangle\mathbf{1})\cap\mathbf{1}=\mathbf{g}\setminus\mathbf{h}$ and $[\mathbf{h}, \mathbf{h}\triangle\mathbf{1}, \mathbf{1}, \mathbf{g}, \mathbf{1}]=(\mathbf{g}\cap\mathbf{h})\cup(\mathbf{h}\triangle\mathbf{1})$ by Lemma \ref{L;Lemma2.3}. As $c_{\mathbf{1}, \mathbf{g}, \mathbf{1}}(M)=\overline{1}$ and $\mathrm{Supp}_{\mathbb{B}_2}(E_\mathbf{1}^*M)=\{B_{\mathbf{1}, \mathbf{g}, \mathbf{1}}\}$, notice that the combination of Equations \eqref{Eq;4}, \eqref{Eq;2}, Lemma \ref{L;Lemma4.7}, \ and Theorem \ref{T;Theorem3.23} implies that
\begin{align}\label{Eq;7}
B_{\mathbf{h}, \mathbf{h}\triangle\mathbf{1}, \mathbf{1}}M\!=\!B_{\mathbf{h}, \mathbf{h}\triangle\mathbf{1}, \mathbf{1}}E_\mathbf{1}^*M=c_{\mathbf{1}, \mathbf{g}, \mathbf{1}}(M)B_{\mathbf{h}, \mathbf{h}\triangle \mathbf{1}, \mathbf{1}}B_{\mathbf{1},\mathbf{g},\mathbf{1}}=\overline{k_{\mathbf{g}\setminus \mathbf{h}}}B_{\mathbf{h},(\mathbf{g}\cap\mathbf{h})\cup(\mathbf{h}\triangle\mathbf{1}), \mathbf{1}}.
\end{align}

For any $\mathbf{i}\in\mathbb{E}$ and $\mathbf{i}\in\mathbb{P}_{\mathbf{h}, \mathbf{h}}$, Lemma \ref{L;Lemma2.3} implies that $\mathbf{h}\cap \mathbf{i}\cap(\mathbf{h}\triangle\mathbf{1})=\mathbf{0}$ and $[\mathbf{h}, \mathbf{i}, \mathbf{h}, \mathbf{h}\triangle\mathbf{1},\mathbf{1}]\!=\!(\mathbf{h}\triangle\mathbf{1})\cup\mathbf{i}$. It is clear that $\mathbf{i}\!\preceq\!\mathbf{h}\!\cap\!\mathbf{1}$ for any $\mathbf{i}\in\mathbb{E}$ and $\mathbf{i}\in\mathbb{P}_{\mathbf{h}, \mathbf{h}}$. The combination of Equations \eqref{Eq;4}, \eqref{Eq;2}, Lemma \ref{L;Lemma4.7}, \ and Theorem \ref{T;Theorem3.23} implies that
\begin{align}\label{Eq;8}
MB_{\mathbf{h}, \mathbf{h}\triangle \mathbf{1}, \mathbf{1}}=\sum_{\mathbf{i}\in\mathbb{P}_{\mathbf{h}, \mathbf{h}}}c_{\mathbf{h},\mathbf{i},\mathbf{h}}(M)B_{\mathbf{h},\mathbf{i},\mathbf{h}}B_{\mathbf{h}, \mathbf{h}\triangle \mathbf{1}, \mathbf{1}}=\sum_{\mathbf{i}\in\mathbb{P}_{\mathbf{h}, \mathbf{h}}}c_{\mathbf{h},\mathbf{i},\mathbf{h}}(M)B_{\mathbf{h},(\mathbf{h}\triangle\mathbf{1})\cup\mathbf{i},\mathbf{1}}.
\end{align}

For any $\mathbf{i}\in\mathbb{E}$ and $\mathbf{i}\in\mathbb{P}_{\mathbf{h}, \mathbf{h}}$, it is obvious that $\mathbf{g}\cap\mathbf{h}\preceq\mathbf{h}\cap\mathbf{1}$ and $\mathbf{i}\preceq\mathbf{h}\cap\mathbf{1}$. Hence  $(\mathbf{g}\cap\mathbf{h})\cap(\mathbf{h}\triangle\mathbf{1})=\mathbf{i}\cap(\mathbf{h}\triangle\mathbf{1})=\mathbf{0}$ for any $\mathbf{i}\in\mathbb{E}$ and $\mathbf{i}\in\mathbb{P}_{\mathbf{h}, \mathbf{h}}$. For any $\mathbf{i}\in\mathbb{E}$ and $\mathbf{i}\in\mathbb{P}_{\mathbf{h}, \mathbf{h}}$, Theorem \ref{T;Theorem3.13} and Lemma \ref{L;Lemma2.3} imply that $B_{\mathbf{h},(\mathbf{g}\cap\mathbf{h})\cup (\mathbf{h}\triangle\mathbf{1}),\mathbf{1}}=B_{\mathbf{h}, \mathbf{i}\cup(\mathbf{h}\triangle\mathbf{1}),\mathbf{1}}$ if and only if $\mathbf{i}\!=\!\mathbf{g}\cap\mathbf{h}$.
As $M\!\in\!\mathrm{Z}(\mathbb{T})$, the combination of Equations \eqref{Eq;7}, \eqref{Eq;8}, and Theorem \ref{T;Theorem3.13} gives $c_{\mathbf{h}, \mathbf{g}\cap\mathbf{h}, \mathbf{h}}(M)=\overline{k_{\mathbf{g}\!\setminus\!\mathbf{h}}}$ and $c_{\mathbf{h}, \mathbf{i}, \mathbf{h}}(M)\!=\!\overline{0}$ if $\mathbf{i}\in\mathbb{E}$ and $\mathbf{i}\in\mathbb{P}_{\mathbf{h},\mathbf{h}}\setminus\{\mathbf{g}\cap\mathbf{h}\}$. As $\mathbf{h}$ is chosen from $\mathbb{E}$ arbitrarily, the desired lemma follows from Equation \eqref{Eq;5}.
\end{proof}
\begin{lem}\label{L;Lemma4.9}
Assume that $M\in\mathrm{Z}(\mathbb{T})$. Then $M$ is an $\F$-linear combination of the matrices in $\{C_\mathbf{a}\!:\! (\mathbf{1}, \mathbf{a}, \mathbf{1})\in\mathbb{P}\}$. In particular, $\mathrm{Z}(\mathbb{T})$ has an $\F$-basis $\{C_\mathbf{a}\!:\! (\mathbf{1}, \mathbf{a}, \mathbf{1})\!\in\!\mathbb{P}\}$.
\end{lem}
\begin{proof}
By combining Equations \eqref{Eq;5}, \eqref{Eq;4}, \eqref{Eq;2}, Lemmas \ref{L;Lemma4.7}, \ref{L;Lemma4.5}, \ref{L;Lemma4.6}, Theorem \ref{T;Theorem3.13}, there is an $\F$-linear combination $L$ of the matrices in $\{C_\mathbf{a}: (\mathbf{1}, \mathbf{a}, \mathbf{1})\in\mathbb{P}\}$ such that $L+M\in\mathrm{Z}(\mathbb{T})$, $c_{\mathbf{1}, \mathbf{g}, \mathbf{1}}(L+M)=\overline{1}$, and $\mathrm{Supp}_{\mathbb{B}_2}(E_\mathbf{1}^*(L+M))=\{B_{\mathbf{1}, \mathbf{g}, \mathbf{1}}\}$ for some $\mathbf{g}\in\mathbb{E}$ and $\mathbf{g}\in\mathbb{P}_{\mathbf{1}, \mathbf{1}}$.
The desired lemma follows from using Lemmas \ref{L;Lemma4.8} and \ref{L;Lemma4.6}.
\end{proof}
We are now ready to give the desired $\F$-basis of $\mathrm{Z}(\mathbb{T})$ and an additional notation.
\begin{thm}\label{T;Center}
$\mathrm{Z}(\mathbb{T})$ has an $\F$-basis $\{C_\mathbf{a}: (\mathbf{1}, \mathbf{a}, \mathbf{1})\in\mathbb{P}\}$ with cardinality $2^{n_2}$.
\end{thm}
\begin{proof}
By Lemma \ref{L;Lemma4.9}, it is enough to check that $|\{C_\mathbf{a}: (\mathbf{1}, \mathbf{a}, \mathbf{1})\in\mathbb{P}\}|=2^{n_2}$. Notice that $|\{C_\mathbf{a}: (\mathbf{1}, \mathbf{a}, \mathbf{1})\in\mathbb{P}\}|=|\{\mathbf{a}: (\mathbf{1}, \mathbf{a}, \mathbf{1})\in\mathbb{P}\}|=|\{\mathbbm{a}: \mathbbm{a}\subseteq[1,n]^\circ\}|$ by Lemmas \ref{L;Lemma4.9} and \ref{L;Lemma2.3}. The desired theorem follows since $n_2=|\{a: a\in[1, n], |\mathbb{U}_a|>2\}|$.
\end{proof}
\begin{nota}\label{N;Notation4.11}
Let $\mathbb{B}_3=\{C_\mathbf{a}: (\mathbf{1}, \mathbf{a}, \mathbf{1})\in\mathbb{P}\}$. If $M\in\mathrm{Z}(\mathbb{T})$ and $(\mathbf{1}, \mathbf{g}, \mathbf{1})\in\mathbb{P}$, notice that $\mathrm{Supp}_{\mathbb{B}_3}(M)$ is defined and write $c_\mathbf{\mathbf{g}}(M)$ for $c_{\mathbb{B}_3}(M, C_\mathbf{g})$ by Theorem \ref{T;Center}.
\end{nota}
We next investigate the structure constants of $\mathbb{B}_3$ in $\mathrm{Z}(\mathbb{T})$. We first list a lemma.
\begin{lem}\label{L;Lemma4.12}
Assume that $\mathbf{g}, \mathbf{h}, \mathbf{i}\in\mathbb{E}$. Then $k_{\mathbf{g}\setminus \mathbf{i}}k_{\mathbf{h}\setminus \mathbf{i}}k_{\mathbf{g}\cap \mathbf{h}\cap \mathbf{i}}=k_{(\mathbf{g}\cup \mathbf{h})\setminus \mathbf{i}}k_{\mathbf{g}\cap \mathbf{h}}$.
\end{lem}
\begin{proof}
Recall that $k_\mathbf{0}=1$ and $k_\mathbf{g}=\prod_{j\in\mathbbm{g}}(|\mathbb{U}_j|-1)$ if $\mathbf{g}\in\mathbb{E}\setminus\{\mathbf{0}\}$. It is obvious that
$(\mathbbm{g}\setminus(\mathbbm{h}\cup \mathbbm{i}))\cap((\mathbbm{g}\cap \mathbbm{h})\setminus \mathbbm{i})=\varnothing$.
Lemma \ref{L;Lemma2.3} thus implies that $k_{\mathbf{g}\setminus \mathbf{i}}=k_{\mathbf{g}\setminus (\mathbf{h}\cup \mathbf{i})}k_{(\mathbf{g}\cap \mathbf{h})\setminus \mathbf{i}}$ as $\mathbbm{g}\setminus \mathbbm{i}=(\mathbbm{g}\setminus (\mathbbm{h}\cup \mathbbm{i}))\cup ((\mathbbm{g}\cap \mathbbm{h})\setminus \mathbbm{i})$. It is obvious that $(\mathbbm{h}\setminus(\mathbbm{g}\cup \mathbbm{i}))\cap((\mathbbm{g}\cap \mathbbm{h})\setminus \mathbbm{i})=\varnothing$.
Lemma \ref{L;Lemma2.3} thus implies that $k_{\mathbf{h}\setminus \mathbf{i}}=k_{\mathbf{h}\setminus(\mathbf{g}\cup \mathbf{i})}k_{(\mathbf{g}\cap \mathbf{h})\setminus \mathbf{i}}$ as $\mathbbm{h}\setminus \mathbbm{i}=(\mathbbm{h}\setminus(\mathbbm{g}\cup \mathbbm{i}))\cup((\mathbbm{g}\cap \mathbbm{h})\setminus \mathbbm{i})$. It is
obvious that $\mathbbm{g}\cap\mathbbm{h}\cap\mathbbm{i}\cap((\mathbbm{g}\cap\mathbbm{h})\setminus\mathbbm{i})=\varnothing$.
As $\mathbbm{g}\cap \mathbbm{h}=(\mathbbm{g}\cap \mathbbm{h}\cap \mathbbm{i})\cup((\mathbbm{g}\cap \mathbbm{h})\setminus \mathbbm{i})$, Lemma \ref{L;Lemma2.3} thus implies that $k_{\mathbf{g}\cap \mathbf{h}}=k_{\mathbf{g}\cap \mathbf{h}\cap \mathbf{i}}k_{(\mathbf{g}\cap \mathbf{h})\setminus \mathbf{i}}$. It is obvious that $\mathbbm{g}\setminus(\mathbbm{h}\cup\mathbbm{i})$, $\mathbbm{h}\setminus(\mathbbm{g}\cup\mathbbm{i})$, and $(\mathbbm{g}\cap\mathbbm{h})\setminus\mathbbm{i}$ are pairwise disjoint sets. Lemma \ref{L;Lemma2.3} thus implies that $k_{(\mathbf{g}\cup \mathbf{h})\setminus \mathbf{i}}=k_{\mathbf{g}\setminus(\mathbf{h}\cup \mathbf{i})}k_{\mathbf{h}\setminus(\mathbf{g}\cup \mathbf{i})}k_{(\mathbf{g}\cap \mathbf{h})\setminus \mathbf{i}}$ as  $(\mathbbm{g}\cup\mathbbm{h})\setminus \mathbbm{i}=(\mathbbm{g}\setminus(\mathbbm{h}\cup \mathbbm{i}))\cup(\mathbbm{h}\setminus(\mathbbm{g}\cup \mathbbm{i}))\cup ((\mathbbm{g}\cap \mathbbm{h})\setminus \mathbbm{i})$.
The desired lemma follows from the combination of the four numerical equations.
\end{proof}
\begin{thm}\label{T;Theorem4.13}
Assume that $\mathbf{g}, \mathbf{h}, \mathbf{i}\in \mathbb{E}$ and $\mathbf{g}, \mathbf{h}, \mathbf{i}\in\mathbb{P}_{\mathbf{1}, \mathbf{1}}$. Then $\mathbf{g}\cup\mathbf{h}\in\mathbb{P}_{\mathbf{1}, \mathbf{1}}$ and
$$ c_\mathbf{i}(C_\mathbf{g}C_\mathbf{h})=\delta_{\mathbf{i},\mathbf{g}\cup\mathbf{h}}\overline{k_{\mathbf{g}\cap\mathbf{h}}}.$$
\end{thm}
\begin{proof}
As $\mathbf{g}, \mathbf{h}\in\mathbb{P}_{\mathbf{1}, \mathbf{1}}$, notice that $\mathbf{g}\cup\mathbf{h}\in\mathbb{P}_{\mathbf{1}, \mathbf{1}}$ by Corollary \ref{C;Corollary3.24}. The first statement follows. Pick $\mathbf{j}\in\mathbb{E}$. Notice that $[\mathbf{j},\mathbf{g}\cap\mathbf{j},\mathbf{j},\mathbf{h}\cap\mathbf{j},\mathbf{j}]=(\mathbf{g}\cup\mathbf{h})\cap\mathbf{j}$ by Lemma \ref{L;Lemma2.3}.
So
\begin{equation}\label{Eq;9}
\begin{aligned}
\overline{k_{\mathbf{g}\setminus \mathbf{i}}}B_{\mathbf{i},\mathbf{g}\cap \mathbf{i}, \mathbf{i}}\overline{k_{\mathbf{h}\setminus \mathbf{j}}}B_{\mathbf{j},\mathbf{h}\cap \mathbf{j}, \mathbf{j}}&=\delta_{\mathbf{i}, \mathbf{j}}\overline{k_{\mathbf{g}\setminus \mathbf{j}}}\overline{k_{\mathbf{h}\setminus \mathbf{j}}}\overline{k_{\mathbf{g}\cap\mathbf{h}\cap\mathbf{j}}}B_{\mathbf{j}, (\mathbf{g}\cup\mathbf{h})\cap\mathbf{j},\mathbf{j}}\\
&=\delta_{\mathbf{i}, \mathbf{j}}\overline{k_{(\mathbf{g}\cup\mathbf{h})\setminus\mathbf{j}}}\overline{k_{\mathbf{g}\cap\mathbf{h}}}B_{\mathbf{j}, (\mathbf{g}\cup\mathbf{h})\cap\mathbf{j},\mathbf{j}}
\end{aligned}
\end{equation}
by Theorem \ref{T;Theorem3.23} and Lemma \ref{L;Lemma4.12}. As $\mathbf{j}$ is chosen from $\mathbb{E}$ arbitrarily, notice that  $C_\mathbf{g}C_\mathbf{h}=\overline{k_{\mathbf{g}\cap\mathbf{h}}}C_{\mathbf{g}\cup\mathbf{h}}$ by combining Equations \eqref{Eq;5}, \eqref{Eq;9}, and the first statement.
The desired theorem follows as $\mathrm{Supp}_{\mathbb{B}_3}(C_\mathbf{g}C_\mathbf{h})\subseteq\{C_{\mathbf{g}\cup\mathbf{h}}\}$ holds by Theorem \ref{T;Center}.
\end{proof}
We are now ready to deduce two additional corollaries of Theorems \ref{T;Center} and \ref{T;Theorem4.13}.
\begin{cor}\label{C;Corollary4.14}
Assume that $\mathbf{g}\in\mathbb{E}$ and $\mathbf{g}\in\mathbb{P}_{\mathbf{1}, \mathbf{1}}$. Then $C_\mathbf{g}C_\mathbf{g}=\overline{k_\mathbf{g}}C_\mathbf{g}$ and the $\F$-subalgebra $\mathrm{Z}(\mathbb{T})$ of $\mathbb{T}$ has a nilpotent two-sided ideal $\langle\{C_\mathbf{a}: (\mathbf{1}, \mathbf{a}, \mathbf{1})\in\mathbb{P}, p\mid k_\mathbf{a}\}\rangle_\mathbb{T}$.
\end{cor}
\begin{proof}
Recall that $k_\mathbf{0}=1$ and $k_\mathbf{g}=\prod_{h\in\mathbbm{g}}(|\mathbb{U}_h|-1)$ if $\mathbf{g}\in\mathbb{E}\setminus\{\mathbf{0}\}$. Notice that $\mathbf{g}\cap\mathbf{g}=\mathbf{g}\cup\mathbf{g}=\mathbf{g}$ by Lemma \ref{L;Lemma2.3}. The first statement is thus from Theorem \ref{T;Theorem4.13}. If $\mathbf{i}, \mathbf{j}\in\mathbb{P}_{\mathbf{1}, \mathbf{1}}$ and $p\mid k_\mathbf{j}$, it is obvious to notice that $p\mid k_{\mathbf{i}\cup\mathbf{j}}$. The desired corollary follows from the combination of the first statement, Theorems \ref{T;Theorem4.13}, and \ref{T;Center}.
\end{proof}
\begin{cor}\label{C;Corollary4.15}
$$\mathrm{Z}(\mathbb{T})\cong E_\mathbf{1}^*\mathbb{T}E_\mathbf{1}^*\ \text{as $\F$-subalgebras of $\mathbb{T}$}.$$
\end{cor}
\begin{proof}
By combining Lemma \ref{L;Lemma4.5}, Theorems \ref{T;Center}, \ref{T;Theorem4.13}, \ref{T;Theorem3.13}, and Corollary \ref{C;Corollary3.24}, the $\F$-linear map that sends $C_\mathbf{g}$ to $B_{\mathbf{1}, \mathbf{g}, \mathbf{1}}$ is also an algebra isomorphism from the $\F$-subalgebra $\mathrm{Z}(\mathbb{T})$ of $\mathbb{T}$ to the $\F$-subalgebra $E_{\mathbf{1}}^*\mathbb{T}E_\mathbf{l}^*$ of $\mathbb{T}$. Hence $\mathrm{Z}(\mathbb{T})\cong E_\mathbf{1}^*\mathbb{T}E_\mathbf{1}^*$ as $\F$-subalgebras of $\mathbb{T}$. The desired corollary follows from the above discussion.
\end{proof}
We complete this section's presentation by an example of Theorems \ref{T;Center} and \ref{T;Theorem4.13}.
\begin{eg}\label{E;Example4.15}
Assume that $n=|\mathbb{U}_1|=2$ and $|\mathbb{U}_2|=3$. It is clear that $|\mathbb{E}|=4$. Assume that $\mathbf{g}=(0,1)$ and $\mathbf{h}=(1,0)$. Hence $\mathbb{E}=\{\mathbf{0}, \mathbf{g}, \mathbf{h}, \mathbf{1}\}$. According to Theorem \ref{T;Center}, notice that $\mathrm{Z}(\mathbb{T})$ has an $\F$-basis $\{C_\mathbf{0}, C_\mathbf{g}\}$. Theorem \ref{T;Center} implies that
\begin{align*}
C_\mathbf{0}\!=\!B_{\mathbf{0},\mathbf{0},\mathbf{0}}+B_{\mathbf{g},\mathbf{0},\mathbf{g}}+B_{\mathbf{h},\mathbf{0},\mathbf{h}}+B_{\mathbf{1},\mathbf{0},\mathbf{1}}\!=\!I\ \text{and}\ C_\mathbf{g}\!=\!\overline{2}B_{\mathbf{0},\mathbf{0},\mathbf{0}}\!+\!B_{\mathbf{g},\mathbf{g},\mathbf{g}}+\overline{2}B_{\mathbf{h},\mathbf{0},\mathbf{h}}+B_{\mathbf{1},\mathbf{g},\mathbf{1}}.
\end{align*}
According to Theorem \ref{T;Theorem4.13}, write $C_\mathbf{0}C_\mathbf{0}\!=\!C_\mathbf{0}$, $C_\mathbf{0}C_\mathbf{g}=C_\mathbf{g}C_\mathbf{0}=C_\mathbf{g}$, and $C_\mathbf{g}C_\mathbf{g}=\overline{2}C_\mathbf{g}$.
\end{eg}
\section{Algebraic structure of $\mathbb{T}$: Semisimplicity}
In this section, we determine the Jacobson radical of the $\F$-subalgebra $E_\mathbf{g}^*\mathbb{T}E_\mathbf{g}^*$ of $\mathbb{T}$ for any $\mathbf{g}\in\mathbb{E}$. As an application, we also determine the semisimplicity of $\mathbb{T}$. As a preparation, we first recall Notations \ref{N;Notation3.1}, \ref{N;Notation3.9}, \ref{N;Notation3.14}, \ref{N;Notation3.15}, \ref{N;Notation4.1},  \ref{N;Notation4.2} and offer a lemma.
\begin{lem}\label{L;Lemma5.1}
The $\F$-subalgebra $E_\mathbf{g}^*\mathbb{T}E_\mathbf{g}^*$ of $\mathbb{T}$ satisfying $\mathbf{g}\in\mathbb{E}$ has a two-sided ideal
$$\langle\{B_{\mathbf{g}, \mathbf{a}, \mathbf{g}}: (\mathbf{g}, \mathbf{a}, \mathbf{g})\in\mathbb{P}, p\mid k_\mathbf{a}\}\rangle_\mathbb{T}.$$
\end{lem}
\begin{proof}
Recall that $k_\mathbf{0}=1$ and $k_\mathbf{g}=\prod_{h\in\mathbbm{g}}(|\mathbb{U}_h|-1)$ if $\mathbf{g}\in\mathbb{E}\setminus\{\mathbf{0}\}$. Assume that $\mathbf{i}, \mathbf{j}\in \mathbb{P}_{\mathbf{g}, \mathbf{g}}$ and $p\mid k_\mathbf{j}$. Notice that $B_{\mathbf{g}, \mathbf{i}, \mathbf{g}}, B_{\mathbf{g}, \mathbf{j}, \mathbf{g}}\in E_\mathbf{g}^*\mathbb{T}E_\mathbf{g}^*$ by Corollary \ref{C;Corollary3.24}. Notice that $ k_{\mathbf{i}\cup\mathbf{j}}$ as $p\mid \mathbf{k}_\mathbf{j}$. Hence $B_{\mathbf{g},\mathbf{i},\mathbf{g}}B_{\mathbf{g}, \mathbf{j}, \mathbf{g}}\in\langle\{B_{\mathbf{g}, \mathbf{a}, \mathbf{g}}: (\mathbf{g}, \mathbf{a}, \mathbf{g})\in\mathbb{P}, p\mid k_\mathbf{a}\}\rangle_\mathbb{T}$ by Corollary \ref{C;Corollary3.24}. The desired lemma follows from the above discussion and Corollary \ref{C;Corollary3.24}.
\end{proof}
Lemma \ref{L;Lemma5.1} motivates us to present the following notation and two related lemmas.
\begin{nota}\label{N;Notation5.2}
Assume that $\mathbf{g}\in\mathbb{E}$. Define $\mathbb{I}_\mathbf{g}=\langle\{B_{\mathbf{g}, \mathbf{a}, \mathbf{g}}: (\mathbf{g}, \mathbf{a}, \mathbf{g})\in\mathbb{P}, p\mid k_\mathbf{a}\}\rangle_\mathbb{T}$. So
Lemma \ref{L;Lemma5.1} implies that $\mathbb{I}_\mathbf{g}$ is a two-sided ideal of the $\F$-subalgebra $E_\mathbf{g}^*\mathbb{T}E_\mathbf{g}^*$ of $\mathbb{T}$.
\end{nota}
\begin{lem}\label{L;Lemma5.3}
Assume that $\mathbf{g}\in\mathbb{E}$. Then there are $|\{a: a\in\mathbbm{g}, p\mid |\mathbb{U}_a|-1\}|$ matrices contained in $\mathbb{I}_\mathbf{g}$ such that their matrix products are equal to a unique nonzero matrix.
\end{lem}
\begin{proof}
Recall that $k_\mathbf{0}=1$ and $k_\mathbf{g}=\prod_{h\in\mathbbm{g}}(|\mathbb{U}_h|-1)$ if $\mathbf{g}\in\mathbb{E}\setminus\{\mathbf{0}\}$. Lemma \ref{L;Lemma2.3} thus implies that $\mathbb{I}_\mathbf{g}\neq\{O\}$ if and only if $\{a: a\in\mathbbm{g}, p\mid |\mathbb{U}_a|-1\}\neq\varnothing$. Therefore there is no loss to only consider the case $\mathbb{I}_\mathbf{g}\neq\{O\}$. Define $\mathbb{U}=\{a: a\in\mathbbm{g}, p\mid |\mathbb{U}_a|-1\}$.
Assume that $i\in\mathbb{N}$ and $\mathbb{U}=\{j_1, j_2, \ldots, j_i\}$. For any $k\in [1, i]$, let $\mathbf{l}_{(k)}$ be the $n$-tuple in $\mathbb{E}$ whose unique nonzero entry is the $j_k$th-entry. So $\mathbf{l}_{(k)}\in\mathbb{P}_{\mathbf{g},\mathbf{g}}$
and $p\mid k_{\mathbf{l}_{(k)}}$ for any $k\in [1,i]$. By Corollary \ref{C;Corollary3.24}, the matrix product of $B_{\mathbf{g}, \mathbf{l}_{(1)}, \mathbf{g}}, B_{\mathbf{g}, \mathbf{l}_{(2)}, \mathbf{g}},\ldots,B_{\mathbf{g}, \mathbf{l}_{(i)}, \mathbf{g}}$ is independent of the multiplication order of $B_{\mathbf{g}, \mathbf{l}_{(1)}, \mathbf{g}}, B_{\mathbf{g}, \mathbf{l}_{(2)}, \mathbf{g}},\ldots,B_{\mathbf{g}, \mathbf{l}_{(i)}, \mathbf{g}}$. Notice that
$$ \prod_{k=1}^iB_{\mathbf{g},\mathbf{l}_{(k)},\mathbf{g}}\neq O$$
by an application of Corollary \ref{C;Corollary3.24} and Lemma \ref{L;Lemma3.10}. The desired lemma follows.
\end{proof}
\begin{lem}\label{L;Lemma5.4}
Assume that $\mathbf{g}\in\mathbb{E}$. Then $\mathbb{I}_\mathbf{g}$ is a nilpotent two-sided ideal of the $\F$-subalgebra $E_\mathbf{g}^*\mathbb{T}E_\mathbf{g}^*$ of $\mathbb{T}$ whose nilpotent index equals $|\{a: a\in\mathbbm{g}, p\mid |\mathbb{U}_a|-1\}|+1$.
\end{lem}
\begin{proof}
Recall that $k_\mathbf{0}=1$ and $k_\mathbf{g}=\prod_{h\in\mathbbm{g}}(|\mathbb{U}_h|-1)$ if $\mathbf{g}\in\mathbb{E}\setminus\{\mathbf{0}\}$. Hence $p\mid k_\mathbf{g}$ if and only if $\{a: a\in\mathbbm{g}, p\mid |\mathbb{U}_a|-1\}\neq\varnothing$. Lemma \ref{L;Lemma2.3} thus implies that $\mathbb{I}_\mathbf{g}\neq\{O\}$ if and only if
$\{a: a\in\mathbbm{g}, p\mid |\mathbb{U}_a|-1\}\neq\varnothing$. Assume further that $\mathbb{I}_\mathbf{g}\neq\{O\}$. Set
$i=|\{a: a\in\mathbbm{g}, p\mid |\mathbb{U}_a|-1\}|+1$. By Lemmas \ref{L;Lemma5.1} and \ref{L;Lemma5.3}, it suffices to check that the product of
any $i$ matrices in $\{B_{\mathbf{g},\mathbf{a} ,\mathbf{g}}: (\mathbf{g}, \mathbf{a}, \mathbf{g})\in\mathbb{P}, p\mid k_\mathbf{a}\}$
is the zero matrix. Pick $B_{\mathbf{g}, \mathbf{j}_{(1)},\mathbf{g}}, B_{\mathbf{g}, \mathbf{j}_{(2)},\mathbf{g}},\ldots, B_{\mathbf{g}, \mathbf{j}_{(i)},\mathbf{g}}\in\{B_{\mathbf{g},\mathbf{a} ,\mathbf{g}}: (\mathbf{g}, \mathbf{a}, \mathbf{g})\in\mathbb{P}, p\mid k_\mathbf{a}\}$. By the Pigeonhole Principal, there is $k\in\{a: a\in\mathbbm{g}, p\mid |\mathbb{U}_a|-1\}$ such that the $k$th-entries of two distinct $n$-tuples in $\{\mathbf{j}_{(1)}, \mathbf{j}_{(2)}, \ldots, \mathbf{j}_{(i)}\}$ equal one. By Corollary \ref{C;Corollary3.24}, any product of $B_{\mathbf{g}, \mathbf{j}_{(1)},\mathbf{g}}, B_{\mathbf{g}, \mathbf{j}_{(2)},\mathbf{g}},\ldots, B_{\mathbf{g}, \mathbf{j}_{(i)},\mathbf{g}}$ is thus the zero matrix. The desired lemma follows.
\end{proof}
For further discussion, the next notations and a sequence of lemmas are required.
\begin{nota}\label{N;Notation5.5}
Assume that $\mathbf{g}, \mathbf{h}\in\mathbb{E}$. Let $n_{\mathbf{g}, \mathbf{h}}=|\{a: a\in\mathbbm{g}\setminus\mathbbm{h}, p\nmid |\mathbb{U}_a|-1\}|$. For any $i\in[0, n_{\mathbf{g},\mathbf{h}}]$, let $\mathbb{U}_{\mathbf{g}, \mathbf{h}, i}=\{\mathbf{a}:\mathbf{h}\preceq\mathbf{a}\preceq\mathbf{g}, p\nmid k_\mathbf{a}, |\mathbbm{a}\setminus\mathbbm{h}|=i\}$.
Assume that $\mathbf{h}\preceq\mathbf{g}$, $p\nmid k_\mathbf{h}$, $i, j\in [0, n_{\mathbf{g}, \mathbf{h}}]$, $i\neq j$. As
$k_\mathbf{0}=1$ and $k_\mathbf{g}=\prod_{k\in\mathbbm{g}}(|\mathbb{U}_k|-1)$ if $\mathbf{g}\in\mathbb{E}\setminus\{\mathbf{0}\}$,
notice that $\mathbb{U}_{\mathbf{g}, \mathbf{h}, 0}=\{\mathbf{h}\}$ and $\mathbb{U}_{\mathbf{g}, \mathbf{h}, i}\neq\varnothing=\mathbb{U}_{\mathbf{g}, \mathbf{h}, i}\cap\mathbb{U}_{\mathbf{g}, \mathbf{h}, j}$ by Lemma \ref{L;Lemma2.3}. For example, if $p\neq n=\!|\mathbb{U}_1|\!=2$, $\!|\mathbb{U}_2|=3$, and $\mathbf{g}=(0, 1)$, then $n_{\mathbf{1}, \mathbf{g}}=1$, $\mathbb{U}_{\mathbf{1}, \mathbf{g}, 0}=\{\mathbf{g}\}$, $\mathbb{U}_{\mathbf{1}, \mathbf{g}, 1}=\{\mathbf{1}\}$.
\end{nota}
\begin{nota}\label{N;Notation5.6}
Assume that $\mathbf{g}, \mathbf{h}\in\mathbb{E}$. Set $\mathbf{g}\mathbf{h}=[\mathbf{g},\mathbf{1}, \mathbf{1}, \mathbf{1}, \mathbf{h}]$. Notice that $\mathbf{g}\mathbf{h}=\mathbf{h}\mathbf{g}$.
Assume that $\mathbf{i}\in\mathbb{E}$. Lemma \ref{L;Lemma2.5} implies that $\mathbf{i}\in\mathbb{P}_{\mathbf{g}, \mathbf{h}}$ if and only if $\mathbf{g}\triangle\mathbf{h}\preceq\mathbf{i}\preceq\mathbf{g}\mathbf{h}$. Recall that $k_\mathbf{0}=1$ and $k_\mathbf{g}=\prod_{j\in\mathbbm{g}}(|\mathbb{U}_j|-1)$ if $\mathbf{g}\in\mathbb{E}\setminus\{\mathbf{0}\}$. If
$p\nmid k_\mathbf{g}$ or $p\nmid k_\mathbf{h}$, notice that $p\nmid k_{\mathbf{g}\cap\mathbf{h}}$. Assume further that $\mathbf{h}\!\in\!
\mathbb{P}_{\mathbf{g},\mathbf{i}}$ and $p\nmid k_\mathbf{h}$. Hence $\mathbf{g}\triangle\mathbf{i}\preceq\mathbf{h}\preceq\mathbf{g}\mathbf{i}$. Define
\begin{align}\label{Eq;10}
D_{\mathbf{g}, \mathbf{h}, \mathbf{i}}=\sum_{k=0}^{n_{\mathbf{g}\mathbf{i}, \mathbf{h}}}\sum_{\mathbf{l}\in\mathbb{U}_{\mathbf{g}\mathbf{i}, \mathbf{h}, k}}(\overline{-1})^k\overline{k_{\mathbf{i}\cap\mathbf{l}}}^{-1}B_{\mathbf{g}, \mathbf{l}, \mathbf{i}}.
\end{align}
As $\mathbb{P}_{\mathbf{g}, \mathbf{i}}\!=\!\mathbb{P}_{\mathbf{i}, \mathbf{g}}$ and $\mathbf{g}\mathbf{i}\!=\!\mathbf{i}\mathbf{g}$, Equation \eqref{Eq;10} and Lemma \ref{L;Lemma3.10} show that $D_{\mathbf{i}, \mathbf{h}, \mathbf{g}}$ is defined.
\end{nota}
\begin{lem}\label{L;Lemma5.7}
Assume that $\mathbf{g}, \mathbf{h}, \mathbf{i}\in\mathbb{E}$, $\mathbf{h}\in\mathbb{P}_{\mathbf{g},\mathbf{i}}$, $p\nmid k_\mathbf{h}$. Then $\overline{k_{\mathbf{i}\setminus\mathbf{g}}}D_{\mathbf{g}, \mathbf{h}, \mathbf{i}}^T=\overline{k_{\mathbf{g}\setminus\mathbf{i}}}D_{\mathbf{i}, \mathbf{h}, \mathbf{g}}$. Moreover, $O\notin\{D_{\mathbf{g}, \mathbf{h}, \mathbf{i}}, D_{\mathbf{i}, \mathbf{h}, \mathbf{g}}\}$. In particular, if $\mathbf{g}=\mathbf{i}$, then $D_{\mathbf{g}, \mathbf{h}, \mathbf{g}}\in E_\mathbf{g}^*\mathbb{T}E_\mathbf{g}^*\ \setminus\ \{O\}$.
\end{lem}
\begin{proof}
Recall that $k_\mathbf{0}=1$ and $k_\mathbf{g}=\prod_{j\in\mathbbm{g}}(|\mathbb{U}_j|-1)$ if $\mathbf{g}\in\mathbb{E}\setminus\{\mathbf{0}\}$. As $\mathbf{h}\in\mathbb{P}_{\mathbf{g},\mathbf{i}}$, notice that
$\mathbf{g}\triangle\mathbf{i}\preceq\mathbf{h}\preceq\mathbf{l}\preceq\mathbf{g}\mathbf{i}$ and $k_{\mathbf{i}\cap\mathbf{l}}=k_{\mathbf{i}\setminus\mathbf{g}}k_{\mathbf{g}\cap\mathbf{i}\cap\mathbf{l}}$ if $k\in[0, n_{\mathbf{g}\mathbf{i}, \mathbf{h}}]$ and $\mathbf{l}\in\mathbb{U}_{\mathbf{g}\mathbf{i}, \mathbf{h}, k}$. Similarly, $k_{\mathbf{g}\cap\mathbf{l}}\!=\!k_{\mathbf{g}\setminus\mathbf{i}}k_{\mathbf{g}\cap\mathbf{i}\cap\mathbf{l}}$ if $k\!\in\![0, n_{\mathbf{i}\mathbf{g}, \mathbf{h}}]$ and $\mathbf{l}\!\in\!\mathbb{U}_{\mathbf{i}\mathbf{g}, \mathbf{h}, k}$. The desired lemma follows from combining Equation \eqref{Eq;10}, Lemma \ref{L;Lemma3.10}, Theorem \ref{T;Theorem3.13}, and Corollary \ref{C;Corollary3.24}.
\end{proof}
\begin{lem}\label{L;Lemma5.8}
Assume that $\mathbf{g}, \mathbf{h}, \mathbf{i}, \mathbf{j}, \mathbf{k}, \mathbf{l}, \mathbf{m}\in\mathbb{E}$ and $\mathbf{j}\in\mathbb{P}_{\mathbf{i}, \mathbf{k}}$. Assume that there is some
$\mathbb{U}\subseteq(\mathbbm{h}\cap\mathbbm{i}\cap\mathbbm{k})^\circ\setminus\mathbbm{j}$ satisfying the equation $\mathbbm{l}=\mathbbm{j}\cup((\mathbbm{i}\cap\mathbbm{k}\cap\mathbbm{m})^\circ\setminus(\mathbbm{h}\cup\mathbbm{j}))\cup\mathbb{U}$. Then $\mathbf{i}\triangle\mathbf{k}\preceq\mathbf{j}\preceq\mathbf{l}\preceq\mathbf{i}\mathbf{k}$ and $\mathbf{l}\!\in\!\mathbb{P}_{\mathbf{i},\mathbf{k}}$. Moreover, if $\mathbf{j}\preceq\mathbf{m}$, then $[\mathbf{g}, \mathbf{h}, \mathbf{i}, \mathbf{l}, \mathbf{k}]=[\mathbf{g}, \mathbf{h}, \mathbf{i}, \mathbf{m}, \mathbf{k}]$.
\end{lem}
\begin{proof}
As $\mathbf{j}\in\mathbb{P}_{\mathbf{i},\mathbf{k}}$, notice that $\mathbbm{j}=(\mathbbm{i}\triangle\mathbbm{k})\cup(\mathbbm{i}\cap\mathbbm{j}\cap\mathbbm{k})=(\mathbbm{i}\triangle\mathbbm{k})\cup(\mathbbm{i}\cap\mathbbm{j}\cap\mathbbm{k})^\circ$.
As $\mathbbm{l}\subseteq(\mathbbm{i}\triangle\mathbbm{k})\cup(\mathbbm{i}\cap\mathbbm{k})^\circ$, notice that $\mathbf{l}\preceq\mathbf{i}\mathbf{k}$ by Lemma \ref{L;Lemma2.3}. Hence $\mathbf{i}\triangle\mathbf{k}\preceq\mathbf{j}\preceq\mathbf{l}\preceq\mathbf{i}\mathbf{k}$. In particular, notice that $\mathbf{l}\in \mathbb{P}_{\mathbf{i}, \mathbf{k}}$.
For the remaining statement, $\mathbf{j}\preceq\mathbf{m}$ implies that $(\mathbbm{h}\cup\mathbbm{l})\cap(\mathbbm{g}\cap\mathbbm{i}\cap\mathbbm{k})^\circ\!=\!(\mathbbm{h}\!\cup\mathbbm{m})\!\cap\!(\mathbbm{g}\cap\mathbbm{i}\cap\mathbbm{k})^\circ$.
Therefore $[\mathbf{g}, \mathbf{h},\mathbf{i}, \mathbf{l}, \mathbf{k}]=[\mathbf{g}, \mathbf{h},\mathbf{i}, \mathbf{m}, \mathbf{k}]$ by applying Lemma \ref{L;Lemma2.3}. The desired lemma follows from the above discussion.
\end{proof}
\begin{lem}\label{L;Lemma5.9}
Assume that $\mathbf{g}, \mathbf{h}, \mathbf{i}, \mathbf{j}, \mathbf{k}, \mathbf{l}, \mathbf{m}\in\mathbb{E}$ and $\mathbf{i}\setminus\mathbf{g}\preceq\mathbf{h}$. Assume that $\mathbf{j}\in\mathbb{P}_{\mathbf{i},\mathbf{k}}$ and  $\mathbf{j}\preceq\mathbf{l}\cap\mathbf{m}$. Then $\mathbf{l}\in\mathbb{P}_{\mathbf{i}, \mathbf{k}}$ and $[\mathbf{g}, \mathbf{h}, \mathbf{i}, \mathbf{l}, \mathbf{k}]=[\mathbf{g}, \mathbf{h}, \mathbf{i}, \mathbf{m}, \mathbf{k}]$ if and only if there is $\mathbb{U}\subseteq (\mathbbm{h}\cap\mathbbm{i}\cap\mathbbm{k})^\circ\setminus\mathbbm{j}$ that satisfies the equation $\mathbbm{l}=\mathbbm{j}\cup((\mathbbm{i}\cap\mathbbm{k}\cap\mathbbm{m})^\circ\setminus(\mathbbm{h}\cup\mathbbm{j}))\cup\mathbb{U}$.
\end{lem}
\begin{proof}
By Lemma \ref{L;Lemma5.8}, it suffices to check that $\mathbf{l}\in\mathbb{P}_{\mathbf{i}, \mathbf{k}}$ and $[\mathbf{g}, \mathbf{h}, \mathbf{i}, \mathbf{l}, \mathbf{k}]\!=\![\mathbf{g}, \mathbf{h}, \mathbf{i}, \mathbf{m}, \mathbf{k}]$ imply that $\mathbbm{l}=\mathbbm{j}\cup((\mathbbm{i}\cap\mathbbm{k}\cap\mathbbm{m})^\circ\setminus(\mathbbm{h}\cup\mathbbm{j}))\cup\mathbb{U}$ and
$\mathbb{U}\subseteq (\mathbbm{h}\cap\mathbbm{i}\cap\mathbbm{k})^\circ\setminus\mathbbm{j}$.
As $\mathbf{i}\setminus\mathbf{g}\preceq\mathbf{h}$, notice that $(\mathbbm{h}\cup\mathbbm{l})\cap((\mathbbm{i}\cap\mathbbm{k})^\circ\setminus\mathbbm{g})=(\mathbbm{i}\cap\mathbbm{k})^\circ\setminus\mathbbm{g}=(\mathbbm{h}\cup\mathbbm{m})\cap((\mathbbm{i}\cap\mathbbm{k})^\circ
\setminus\mathbbm{g})$. So
$(\mathbbm{h}\cup\mathbbm{l})\cap(\mathbbm{i}\cap\mathbbm{k})^\circ\!=\!(\mathbbm{h}\cup\mathbbm{m})\cap(\mathbbm{i}\cap\mathbbm{k})^\circ$ if and only if
$(\mathbbm{h}\cup\mathbbm{l})\cap(\mathbbm{g}\cap\mathbbm{i}\cap\mathbbm{k})^\circ=(\mathbbm{h}\cup\mathbbm{m})\cap(\mathbbm{g}\cap\mathbbm{i}\cap\mathbbm{k})^\circ$.

Assume that $\mathbf{l}\in\mathbb{P}_{\mathbf{i}, \mathbf{k}}$. As $\mathbf{j}\in\mathbb{P}_{\mathbf{i}, \mathbf{k}}$ and $\mathbf{j}\preceq\mathbf{l}$, $(\mathbbm{i}\triangle\mathbbm{k})\subseteq\mathbbm{j}\subseteq\mathbbm{l}\subseteq(\mathbbm{i}\triangle\mathbbm{k})\cup(\mathbbm{i}\cap\mathbbm{k})^\circ$.
Assume that $[\mathbf{g}, \mathbf{h}, \mathbf{i}, \mathbf{l}, \mathbf{k}]=[\mathbf{g}, \mathbf{h}, \mathbf{i}, \mathbf{m}, \mathbf{k}]$. The above discussion and Lemma \ref{L;Lemma2.3} give
$(\mathbbm{h}\cup\mathbbm{l})\cap(\mathbbm{i}\cap\mathbbm{k})^\circ=(\mathbbm{h}\cup\mathbbm{m})\cap(\mathbbm{i}\cap\mathbbm{k})^\circ$. As $\mathbf{j}\preceq\mathbf{m}$ and
$(\mathbbm{h}\cup\mathbbm{l})\cap(\mathbbm{i}\cap\mathbbm{k})^\circ=(\mathbbm{h}\cup\mathbbm{m})\cap(\mathbbm{i}\cap\mathbbm{k})^\circ$, it is obvious to see that
$\mathbbm{l}=\mathbbm{j}\cup((\mathbbm{i}\cap\mathbbm{k}\cap\mathbbm{m})^\circ\setminus(\mathbbm{h}\cup\mathbbm{j}))\cup\mathbb{U}$ and $\mathbb{U}\subseteq (\mathbbm{h}\cap\mathbbm{i}\cap\mathbbm{k})^\circ\setminus\mathbbm{j}$ by a direct computation. The desired lemma follows from the above discussion.
\end{proof}
\begin{lem}\label{L;Lemma5.10}
Assume that $\mathbf{g}, \mathbf{h}, \mathbf{i}, \mathbf{j}, \mathbf{k}, \mathbf{l},\mathbf{m}\in\mathbb{E}$, $\mathbf{g}\triangle\mathbf{i}\preceq\mathbf{h}$, $\mathbf{i}\triangle\mathbf{k}\preceq\mathbf{j}$. Assume that
there is some $\mathbb{U}\subseteq(\mathbbm{h}\cap\mathbbm{i}\cap\mathbbm{k})^\circ\setminus\mathbbm{j}$ satisfying the equation $\mathbbm{l}=\mathbbm{j}\cup((\mathbbm{i}\cap\mathbbm{k}\cap\mathbbm{m})^\circ\setminus(\mathbbm{h}\cup\mathbbm{j}))\cup\mathbb{U}$. Assume that $p\!\nmid\! k_\mathbf{h}k_\mathbf{j}k_\mathbf{m}$. Then $p\!\nmid\! k_{\mathbf{l}}$. Moreover, $p\!\nmid\! k_{[\mathbf{g},\mathbf{h},\mathbf{i}, \mathbf{l}, \mathbf{k}]}$. In particular, $p\!\nmid\! k_{[\mathbf{g}, \mathbf{h}, \mathbf{i}, \mathbf{j}, \mathbf{k}]}$.
\end{lem}
\begin{proof}
Recall that $k_\mathbf{0}=1$ and $k_\mathbf{g}=\prod_{q\in\mathbbm{g}}(|\mathbb{U}_q|-1)$ if $\mathbf{g}\in\mathbb{E}\setminus\{\mathbf{0}\}$. As $\mathbf{i}\triangle\mathbf{k}\preceq\mathbf{j}$ and $p\nmid k_\mathbf{h}k_\mathbf{j}k_\mathbf{m}$, it is obvious that $p\nmid k_\mathbf{l}$. Notice that $\mathbbm{g}\setminus\mathbbm{k}=(\mathbbm{g}\setminus(\mathbbm{i}\cup\mathbbm{k}))\cup((\mathbbm{g}\cap\mathbbm{i})\setminus\mathbbm{k})$ and
$\mathbbm{k}\setminus\mathbbm{g}=(\mathbbm{k}\setminus(\mathbbm{g}\cup\mathbbm{i}))\cup((\mathbbm{i}\cap\mathbbm{k})\setminus\mathbbm{g})$.
As $\mathbf{g}\triangle\mathbf{i}\preceq\mathbf{h}$, $\mathbf{i}\triangle\mathbf{k}\preceq\mathbf{j}$, and $p\nmid k_\mathbf{h}k_\mathbf{j}$,
notice that $\mathbbm{g}\triangle\mathbbm{k}\subseteq\mathbbm{h}\cup\mathbbm{j}$ and $p\nmid |\mathbb{U}_r|-1$ for any $r\in \mathbbm{g}\triangle\mathbbm{k}$. As $\mathbf{i}\triangle\mathbf{k}\preceq\mathbf{j}$ and $p\nmid k_\mathbf{j}$, notice that $(\mathbbm{g}\cap\mathbbm{k})^\circ\setminus\mathbbm{i}\subseteq\mathbbm{j}$ and $p\nmid |\mathbb{U}_r|-1$ for any $r\in(\mathbbm{g}\cap\mathbbm{k})^\circ\setminus\mathbbm{i}$. As $p\nmid k_\mathbf{h}k_\mathbf{l}$, notice that $p\nmid |\mathbb{U}_r|-1$ for any $r\in(\mathbbm{h}\cup\mathbbm{l})\cap(\mathbbm{g}\cap\mathbbm{i}\cap\mathbbm{k})^\circ$. This thus implies that $p\nmid k_{[\mathbf{g},\mathbf{h},\mathbf{i}, \mathbf{l}, \mathbf{k}]}$. For the remaining statement, it is obvious that $\mathbf{j}$ satisfies the assumptions of $\mathbf{l}$ by changing $\mathbf{m}, \mathbb{U}$ to $\mathbf{j}, \varnothing$. The desired lemma follows from the shown statement $p\!\nmid\! k_{[\mathbf{g},\mathbf{h},\mathbf{i}, \mathbf{l}, \mathbf{k}]}$.
\end{proof}
\begin{lem}\label{L;Lemma5.11}
Assume that $\mathbf{g}, \mathbf{h}, \mathbf{i}, \mathbf{j}, \mathbf{k}, \mathbf{l},\mathbf{m}\in\mathbb{E}$, $\mathbf{i}\setminus\mathbf{g}\preceq\mathbf{h}$, $\mathbf{j}, \mathbf{l}\in\mathbb{P}_{\mathbf{i}, \mathbf{k}}$, $\mathbf{j}\preceq\mathbf{l}\cap\mathbf{m}$. Assume that $[\mathbf{g},\mathbf{h}, \mathbf{i}, \mathbf{l}, \mathbf{k}]=[\mathbf{g},\mathbf{h}, \mathbf{i}, \mathbf{m}, \mathbf{k}]$. Then $k_{(\mathbf{i}\cap\mathbf{k}\cap\mathbf{m})\setminus(\mathbf{h}\cup\mathbf{j})}k_{\mathbf{j}\cap\mathbf{k}}k_{\mathbf{h}\cap\mathbf{i}\cap\mathbf{l}}
=k_{\mathbf{h}\cap\mathbf{i}\cap\mathbf{j}}k_{\mathbf{k}\cap\mathbf{l}}$.
\end{lem}
\begin{proof}
Recall that $k_\mathbf{0}=1$ and $k_\mathbf{g}=\prod_{q\in\mathbbm{g}}(|\mathbb{U}_q|-1)$ if $\mathbf{g}\in\mathbb{E}\setminus\{\mathbf{0}\}$. Lemma \ref{L;Lemma5.9} thus implies that $\mathbbm{l}=\mathbbm{j}\cup((\mathbbm{i}\cap\mathbbm{k}\cap\mathbbm{m})^\circ\setminus(\mathbbm{h}\cup\mathbbm{j}))\cup\mathbb{U}$ for some $\mathbb{U}\subseteq(\mathbbm{h}\cap\mathbbm{i}\cap\mathbbm{k})^\circ\setminus\mathbbm{j}$. Notice that $\mathbbm{h}\cap\mathbbm{i}\cap\mathbbm{l}=\mathbbm{h}\cap\mathbbm{i}\cap(\mathbbm{j}\cup((\mathbbm{i}\cap\mathbbm{k}\cap\mathbbm{m})^\circ\setminus(\mathbbm{h}\cup\mathbbm{j}))\cup\mathbb{U})=(\mathbbm{h}
\cap\mathbbm{i}\cap\mathbbm{j})\cup\mathbb{U}$ and $\mathbbm{k}\cap\mathbbm{l}=\mathbbm{k}\cap(\mathbbm{j}\cup((\mathbbm{i}\cap\mathbbm{k}\cap\mathbbm{m})^\circ\setminus(\mathbbm{h}\cup\mathbbm{j}))\cup\mathbb{U})
=(\mathbbm{j}\cap\mathbbm{k})\cup((\mathbbm{i}\cap\mathbbm{k}\cap\mathbbm{m})^\circ\setminus(\mathbbm{h}\cup\mathbbm{j}))\cup\mathbb{U}$. As $(\mathbbm{h}\cap\mathbbm{i}\cap\mathbbm{j})\cap\mathbb{U}=\varnothing$ and $\mathbbm{j}\cap\mathbbm{k}$, $(\mathbbm{i}\cap\mathbbm{k}\cap\mathbbm{m})^\circ\setminus(\mathbbm{h}\cup\mathbbm{j})$, $\mathbb{U}$ are pairwise
distinct sets, it is obvious that $k_{(\mathbf{i}\cap\mathbf{k}\cap\mathbf{m})\setminus(\mathbf{h}\cup\mathbf{j})}k_{\mathbf{j}\cap\mathbf{k}}k_{\mathbf{h}\cap\mathbf{i}\cap\mathbf{l}}=k_{\mathbf{h}\cap\mathbf{i}\cap\mathbf{j}}k_{\mathbf{k}\cap\mathbf{l}}$.
The desired lemma follows.
\end{proof}
\begin{lem}\label{L;Lemma5.12}
Assume that $\mathbf{g}, \mathbf{h}, \mathbf{i}, \mathbf{j}, \mathbf{k}, \mathbf{l}, \mathbf{m}\in\mathbb{E}$, $\mathbf{g}\triangle\mathbf{i}\preceq\mathbf{h}$, $\mathbf{j}\in\mathbb{P}_{\mathbf{i},\mathbf{k}}$,
$\mathbf{j}\preceq\mathbf{l}\preceq\mathbf{i}\mathbf{k}$, $\mathbf{j}\preceq\mathbf{m}$, $p\nmid k_{\mathbf{h}}k_{\mathbf{j}}k_{\mathbf{m}}$. Then $p\nmid k_\mathbf{l}$ and $[\mathbf{g},\mathbf{h}, \mathbf{i}, \mathbf{l}, \mathbf{k}]=[\mathbf{g},\mathbf{h}, \mathbf{i}, \mathbf{m}, \mathbf{k}]$ if and only if there is $\mathbb{U}\subseteq (\mathbbm{h}\cap\mathbbm{i}\cap\mathbbm{k})^\circ\setminus\mathbbm{j}$ that satisfies the equation
$\mathbbm{l}=\mathbbm{j}\cup((\mathbbm{i}\cap\mathbbm{k}\cap\mathbbm{m})^\circ\setminus(\mathbbm{h}\cup\mathbbm{j}))\cup\mathbb{U}$. Moreover, if $\mathbf{h}\in\mathbb{P}_{\mathbf{g}, \mathbf{i}}$, then $B_{\mathbf{g}, \mathbf{h}, \mathbf{i}}D_{\mathbf{i}, \mathbf{j}, \mathbf{k}}\neq O$ only if $(\mathbbm{h}\cap\mathbbm{i}\cap\mathbbm{k})^\circ\subseteq\mathbbm{j}$.
\end{lem}
\begin{proof}
As $\mathbf{j}\in\mathbb{P}_{\mathbf{i},\mathbf{k}}$ and $\mathbf{j}\preceq\mathbf{l}\preceq\mathbf{i}\mathbf{k}$, it is obvious that $\mathbf{l}\in\mathbb{P}_{\mathbf{i},\mathbf{k}}$. The first statement thus follows from Lemmas \ref{L;Lemma5.9} and \ref{L;Lemma5.10}. For the remaining statement, assume that $B_{\mathbf{g}, \mathbf{h}, \mathbf{i}}D_{\mathbf{i}, \mathbf{j}, \mathbf{k}}\neq O$ and $(\mathbbm{h}\cap\mathbbm{i}\cap\mathbbm{k})^\circ\not\subseteq\mathbbm{j}$. Assume that $\mathbf{j}\preceq\mathbf{m}\preceq\mathbf{i}\mathbf{k}$. As $p\nmid k_\mathbf{m}$, notice that
$B_{\mathbf{i}, \mathbf{m}, \mathbf{k}}\in\mathrm{Supp}_{\mathbb{B}_2}(D_{\mathbf{i}, \mathbf{j}, \mathbf{k}})$ by Lemma \ref{L;Lemma5.7} and Equation \eqref{Eq;10}. As $B_{\mathbf{g}, \mathbf{h}, \mathbf{i}}D_{\mathbf{i}, \mathbf{j}, \mathbf{k}}\neq O$, assume further that $B_{\mathbf{g}, [\mathbf{g},\mathbf{h}, \mathbf{i},\mathbf{m}, \mathbf{k}], \mathbf{k}}\in\mathrm{Supp}_{\mathbb{B}_2}(B_{\mathbf{g},\mathbf{h},\mathbf{i}}D_{\mathbf{i}, \mathbf{j}, \mathbf{k}})$ by Theorem \ref{T;Theorem3.23}. Write $q$ for $|(\mathbbm{h}\cap\mathbbm{i}\cap\mathbbm{k})^\circ\setminus\mathbbm{j}|$. Notice that $q>0$ as $(\mathbbm{h}\cap\mathbbm{i}\cap\mathbbm{k})^\circ\not\subseteq\mathbbm{j}$. Set $r=|(\mathbbm{i}\cap\mathbbm{k}\cap\mathbbm{m})^\circ\setminus(\mathbbm{h}\cup\mathbbm{j})|$. By combining
the first statement, Equation \eqref{Eq;10}, Theorem \ref{T;Theorem3.23}, and Lemma \ref{L;Lemma5.11},
$$
c_{\mathbf{g}, [\mathbf{g},\mathbf{h},\mathbf{i},\mathbf{m},\mathbf{k}], \mathbf{k}}(B_{\mathbf{g},\mathbf{h},\mathbf{i}}D_{\mathbf{i},\mathbf{j},\mathbf{k}})=
\sum_{s=0}^q(\overline{-1})^{r+s}\overline{k_{\mathbf{h}\cap\mathbf{i}\cap\mathbf{j}}}(\overline{k_{(\mathbf{i}\cap\mathbf{k}\cap\mathbf{m})\setminus(\mathbf{h}\cup\mathbf{j})}}\overline{k_{\mathbf{j}\cap\mathbf{k}}})^{-1}
\overline{{q\choose s}}=\overline{0},$$ which is absurd as $B_{\mathbf{g}, [\mathbf{g},\mathbf{h}, \mathbf{i},\mathbf{m}, \mathbf{k}], \mathbf{k}}\!\!\in\!\mathrm{Supp}_{\mathbb{B}_2}\!(\!B_{\mathbf{g},\mathbf{h},\mathbf{i}}D_{\mathbf{i}, \mathbf{j}, \mathbf{k}})$. The desired lemma follows.
\end{proof}
\begin{lem}\label{L;Lemma5.13}
Assume that $\mathbf{g}, \mathbf{h}, \mathbf{i}\in\mathbb{E}$. Assume that $\mathbf{h}, \mathbf{i}\in\mathbb{P}_{\mathbf{g}, \mathbf{g}}$ and $p\nmid k_\mathbf{h}k_\mathbf{i}$. Then
\[B_{\mathbf{g},\mathbf{h} ,\mathbf{g}}D_{\mathbf{g},\mathbf{i} ,\mathbf{g}}=D_{\mathbf{g},\mathbf{i} ,\mathbf{g}}B_{\mathbf{g},\mathbf{h} ,\mathbf{g}}=\begin{cases}
\overline{k_\mathbf{h}}D_{\mathbf{g},\mathbf{i} ,\mathbf{g}}, &\ \text{if}\ \mathbf{h}\preceq\mathbf{i},\\
O, &\ \text{otherwise}.
\end{cases}\]
\end{lem}
\begin{proof}
If $\mathbf{h}\preceq\mathbf{i}$, then $B_{\mathbf{g},\mathbf{h} ,\mathbf{g}}D_{\mathbf{g},\mathbf{i} ,\mathbf{g}}=D_{\mathbf{g},\mathbf{i} ,\mathbf{g}}B_{\mathbf{g},\mathbf{h} ,\mathbf{g}}=\overline{k_\mathbf{h}}D_{\mathbf{g},\mathbf{i} ,\mathbf{g}}$ by Equation \eqref{Eq;10} and Corollary \ref{C;Corollary3.24}. Otherwise, notice that $\mathbbm{h}^\circ=\mathbbm{h}\!\not\subseteq\!\mathbbm{i}$ and $B_{\mathbf{g},\mathbf{h} ,\mathbf{g}}D_{\mathbf{g},\mathbf{i} ,\mathbf{g}}\!=\!D_{\mathbf{g},\mathbf{i} ,\mathbf{g}}B_{\mathbf{g},\mathbf{h} ,\mathbf{g}}=O$ by an application of Lemma \ref{L;Lemma5.12} and Corollary \ref{C;Corollary3.24}. The desired lemma follows.
\end{proof}
\begin{lem}\label{L;Lemma5.14}
Assume that $\mathbf{g}, \mathbf{h}, \mathbf{i}\in\mathbb{E}$. Assume that $\mathbf{h}, \mathbf{i}\in\mathbb{P}_{\mathbf{g}, \mathbf{g}}$ and $p\nmid k_\mathbf{h}k_\mathbf{i}$. Then
$$D_{\mathbf{g}, \mathbf{h},\mathbf{g}}D_{\mathbf{g}, \mathbf{i},\mathbf{g}}=D_{\mathbf{g}, \mathbf{i},\mathbf{g}}D_{\mathbf{g}, \mathbf{h},\mathbf{g}}=\delta_{\mathbf{h},\mathbf{i}}D_{\mathbf{g}, \mathbf{h},\mathbf{g}}.$$
\end{lem}
\begin{proof}
Notice that $D_{\mathbf{g}, \mathbf{h},\mathbf{g}}D_{\mathbf{g}, \mathbf{i},\mathbf{g}}\!\!=\!\!D_{\mathbf{g}, \mathbf{i},\mathbf{g}}D_{\mathbf{g}, \mathbf{h},\mathbf{g}}$ by Lemma \ref{L;Lemma5.7} and Corollary \ref{C;Corollary3.24}. By Equation \eqref{Eq;10} and Lemma \ref{L;Lemma5.13}, notice that $D_{\mathbf{g}, \mathbf{h},\mathbf{g}}D_{\mathbf{g}, \mathbf{i},\mathbf{g}}\!\!=\!\!D_{\mathbf{g}, \mathbf{i},\mathbf{g}}D_{\mathbf{g}, \mathbf{h},\mathbf{g}}\neq O$ only if $\mathbf{h}\preceq\mathbf{i}\preceq\mathbf{h}$. By Lemma \ref{L;Lemma2.3}, $D_{\mathbf{g}, \mathbf{h},\mathbf{g}}D_{\mathbf{g}, \mathbf{i},\mathbf{g}}\!\!=\!\!D_{\mathbf{g}, \mathbf{i},\mathbf{g}}D_{\mathbf{g}, \mathbf{h},\mathbf{g}}\!\neq\!O$ only if $\mathbf{h}\!=\!\mathbf{i}$. The desired lemma follows as $D_{\mathbf{g}, \mathbf{h}, \mathbf{g}}D_{\mathbf{g}, \mathbf{h}, \mathbf{g}}\!=\!D_{\mathbf{g}, \mathbf{h}, \mathbf{g}}$ holds
by Equation \eqref{Eq;10} and Lemma \ref{L;Lemma5.13}.
\end{proof}
We are now ready to give the first main result of this section and some corollaries.
\begin{thm}\label{T;Theorem5.15}
Assume that $\mathbf{g}\in\mathbb{E}$. Then the $\F$-subalgebra $E_\mathbf{g}^*\mathbb{T}E_\mathbf{g}^*$ of $\mathbb{T}$ satisfies the formula $E_\mathbf{g}^*\mathbb{T}E_\mathbf{g}^*/\mathbb{I}_\mathbf{g}\cong 2^{n_{\mathbf{g}\mathbf{g},\mathbf{0}}}\mathrm{M}_1(\F)$ as $\F$-algebras. Moreover, $\mathrm{Rad}(E_\mathbf{g}^*\mathbb{T}E_\mathbf{g}^*)=\mathbb{I}_\mathbf{g}$. In particular, the nilpotent index of $\mathrm{Rad}(E_\mathbf{g}^*\mathbb{T}E_\mathbf{g}^*)$ equals $|\{a: a\in\mathbbm{g}, p\mid |\mathbb{U}_a|-1\}|+1$.
\end{thm}
\begin{proof}
Set $\mathbb{U}=\{D_{\mathbf{g}, \mathbf{a}, \mathbf{g}}+\mathbb{I}_\mathbf{g}:(\mathbf{g}, \mathbf{a}, \mathbf{g})\in\mathbb{P}, p\nmid k_\mathbf{a}\}$. Then $O+\mathbb{I}_\mathbf{g}\notin\mathbb{U}$ by combining Lemma \ref{L;Lemma5.7}, Equation \eqref{Eq;10}, and Theorem \ref{T;Theorem3.13}. Corollary \ref{C;Corollary3.24} and Lemma \ref{L;Lemma5.14} thus imply that $\mathbb{U}$ is an $\F$-basis of $E_\mathbf{g}^*\mathbb{T}E_\mathbf{g}^*/\mathbb{I}_\mathbf{g}$. Lemmas \ref{L;Lemma5.14} and \ref{L;Lemma2.3} also imply that $|\mathbb{U}|\!=\!|\{a: (\mathbf{g}, \mathbf{a}, \mathbf{g})\!\in\!\mathbb{P}, p\!\nmid\! k_\mathbf{a}\}|\!=\!2^{n_{\mathbf{g}\mathbf{g},\mathbf{0}}}$. The first statement is from Lemma \ref{L;Lemma5.14}. The desired theorem follows from combining the first statement, Lemmas \ref{L;Lemma2.7}, \ref{L;Lemma5.4}.
\end{proof}
\begin{cor}\label{C;Corollary5.16}
Assume that $\mathbf{g}, \mathbf{h}\in\mathbb{E}$, $\mathbf{h}\!\in\!\mathbb{P}_{\mathbf{g},\mathbf{g}}$, $p\nmid k_\mathbf{h}$, $\mathbbm{Irr}_{\mathbf{g},\mathbf{h}}=\langle\{D_{\mathbf{g},\mathbf{h},\mathbf{g}}+\mathbb{I}_\mathbf{g}\}\rangle_{\mathbb{T}/\mathbb{I}_\mathbf{g}}$. Then $\mathbbm{Irr}_{\mathbf{g},\mathbf{h}}$ is a module of the $\F$-subalgebra $E_\mathbf{g}^*\mathbb{T}E_\mathbf{g}^*$ of $\mathbb{T}$ via the left multiplication. Moreover, $\{\mathbbm{Irr}_{\mathbf{g},\mathbf{a}}\!:\! (\mathbf{g}, \mathbf{a}, \mathbf{g})\!\in\!\mathbb{P}, p\!\nmid\! k_\mathbf{a}\}$ is a complete set of distinct representatives of all isomorphism classes of
the irreducible modules of the $\F$-subalgebra $E_\mathbf{g}^*\mathbb{T}E_\mathbf{g}^*$ of $\mathbb{T}$.
\end{cor}
\begin{proof}
The first statement follows from combining Lemmas \ref{L;Lemma5.13}, \ref{L;Lemma5.1}, Corollary \ref{C;Corollary3.24}. The desired corollary follows from combining Theorem \ref{T;Theorem5.15}, Lemmas \ref{L;Lemma5.14}, \ref{L;Lemma2.8}.
\end{proof}
\begin{cor}\label{C;Corollary5.17}
Assume that $\mathbf{g}\in\mathbb{E}$. Then an irreducible module of the $\F$-subalgebra $E_\mathbf{g}^*\mathbb{T}E_\mathbf{g}^*$ of $\mathbb{T}$ is also an absolutely irreducible module of the $\F$\!-\!subalgebra $E_\mathbf{g}^*\mathbb{T}E_\mathbf{g}^*$ of $\mathbb{T}$. In particular, the set of all irreducible modules of the $\F$-subalgebra $E_\mathbf{g}^*\mathbb{T}E_\mathbf{g}^*$ of $\mathbb{T}$ is exactly the set of all absolutely irreducible
modules of the $\F$-subalgebra $E_\mathbf{g}^*\mathbb{T}E_\mathbf{g}^*$ of $\mathbb{T}$.
\end{cor}
\begin{proof}
The desired corollary follows from using Theorem \ref{T;Theorem5.15} and Lemma \ref{L;Lemma2.9}.
\end{proof}
\begin{cor}\label{C;Corollary5.18}
Assume that $\mathbf{g}\in\mathbb{E}$. Then an irreducible module of the $\F$-subalgebra $E_\mathbf{g}^*\mathbb{T}E_\mathbf{g}^*$ of $\mathbb{T}$ is also a self-contragredient module of the $\F$-subalgebra $E_\mathbf{g}^*\mathbb{T}E_\mathbf{g}^*$ of $\mathbb{T}$ with respect to $\alpha_T$. In particular, all irreducible modules of the $\F$\!-\!subalgebra $E_\mathbf{g}^*\mathbb{T}E_\mathbf{g}^*$ of $\mathbb{T}$ are self-contragredient modules of the $\F$-subalgebra $E_\mathbf{g}^*\mathbb{T}E_\mathbf{g}^*$ of $\mathbb{T}$ with respect to $\alpha_T$.
\end{cor}
\begin{proof}
The combination of Corollary \ref{C;Corollary3.24}, Equation \eqref{Eq;10}, and Theorem \ref{T;Theorem3.23} implies that $\{B_{\mathbf{g}, \mathbf{a}, \mathbf{g}}: (\mathbf{g}, \mathbf{a}, \mathbf{g})\in\mathbb{P}, p\mid k_\mathbf{a}\}\cup\{D_{\mathbf{g}, \mathbf{a}, \mathbf{g}}: (\mathbf{g}, \mathbf{a}, \mathbf{g})\in\mathbb{P}, p\nmid k_\mathbf{a}\}$ is an $\F$-basis of the $\F$-subalgebra $E_\mathbf{g}^*\mathbb{T}E_\mathbf{g}^*$ of $\mathbb{T}$. Then the desired corollary follows from combining Corollary \ref{C;Corollary5.16}, Lemmas \ref{L;Lemma3.10}, \ref{L;Lemma5.1}, \ref{L;Lemma5.7}, \ref{L;Lemma2.3}, \ref{L;Lemma5.14}, and the above discussion.
\end{proof}
\begin{cor}\label{C;Corollary5.19}
Assume that $\mathbf{g}\in\mathbb{E}$. Then the $\F$-subalgebra $E_\mathbf{g}^*\mathbb{T}E_\mathbf{g}^*$ of $\mathbb{T}$ is a semisimple $\F$-subalgebra of $\mathbb{T}$ if and only if $\mathbf{g}=\mathbf{0}$ or $\mathbf{g}\neq\mathbf{0}$ and $p\nmid\prod_{h\in\mathbbm{g}}(|\mathbb{U}_h|-1)$.
\end{cor}
\begin{proof}
The desired corollary follows from using Theorem \ref{T;Theorem5.15} and Lemma \ref{L;Lemma2.6}.
\end{proof}
We are now ready to get the second main result of this section and some corollaries.
\begin{thm}\label{T;Semisimplicity}
$\mathbb{T}$ is a semisimple $\F$-algebra if and only if $\mathbb{S}$ is a $p'$-valenced scheme.
\end{thm}
\begin{proof}
Assume that $\mathbb{S}$ is a $p'$-valenced scheme. So $p\nmid k_\mathbf{g}$ for any $\mathbf{g}\in\mathbb{E}$. Theorem \ref{T;Theorem5.15} shows that the $\F$-subalgebra $E_\mathbf{g}^*\mathbb{T}E_\mathbf{g}^*$ of $\mathbb{T}$ satisfies the equation $\mathrm{Rad}(E_\mathbf{g}^*\mathbb{T}E_\mathbf{g}^*)\!=\!\{O\}$ for any $\mathbf{g}\in\mathbb{E}$. Pick $M\in\mathrm{Rad}(\mathbb{T})$. Assume further that $M\neq O$. Lemma \ref{L;Lemma2.6} implies that $E_\mathbf{g}^*ME_\mathbf{g}^*=O$ for any $\mathbf{g}\in\mathbb{E}$. As $M\neq O$ and Equation \eqref{Eq;3} holds, there are distinct $\mathbf{h}, \mathbf{i}\in\mathbb{E}$ such that $E_\mathbf{h}^*ME_\mathbf{i}^*\neq O$. Then there are $j\in\mathbb{N}$ and pairwise distinct $\mathbf{k}_{(1)}, \mathbf{k}_{(2)}, \ldots, \mathbf{k}_{(j)}\!\in\!\mathbb{P}_{\mathbf{h}, \mathbf{i}}$ such that $\mathrm{Supp}_{\mathbb{B}_2}(E_\mathbf{h}^*ME_\mathbf{i}^*)\!=\!\{B_{\mathbf{h}, \mathbf{k}_{(1)},\mathbf{i}}, B_{\mathbf{h}, \mathbf{k}_{(2)},\mathbf{i}}, \ldots, B_{\mathbf{h}, \mathbf{k}_{(j)},\mathbf{i}}\}$.

As $\mathbf{k}_{(1)}, \mathbf{k}_{(2)}, \ldots, \mathbf{k}_{(j)}$ are pairwise distinct $n$-tuples in $\mathbb{P}_{\mathbf{h}, \mathbf{i}}$, Lemmas \ref{L;Lemma2.5} and \ref{L;Lemma2.3} imply that $\mathbf{h}\cap\mathbf{i}\cap\mathbf{k}_{(1)}, \mathbf{h}\cap\mathbf{i}\cap\mathbf{k}_{(2)},\ldots, \mathbf{h}\cap\mathbf{i}\cap\mathbf{k}_{(j)}$ are pairwise distinct
$n$-tuples. Lemma \ref{L;Lemma2.3} thus implies that $[\mathbf{h},\mathbf{k}_{(1)}, \mathbf{i}, \mathbf{h}\triangle\mathbf{i}, \mathbf{h}], [\mathbf{h},\mathbf{k}_{(2)}, \mathbf{i}, \mathbf{h}\triangle\mathbf{i}, \mathbf{h}], \ldots, [\mathbf{h},\mathbf{k}_{(j)}, \mathbf{i}, \mathbf{h}\triangle\mathbf{i}, \mathbf{h}]$ are pairwise distinct $n$-tuples. Moreover, Lemma \ref{L;Lemma2.5} also gives $\mathbf{h}\triangle\mathbf{i}\in\mathbb{P}_{\mathbf{h}, \mathbf{i}}$. Notice that
$$c_{\mathbf{h}, [\mathbf{h},\mathbf{k}_{(1)}, \mathbf{i}, \mathbf{h}\triangle\mathbf{i}, \mathbf{h}],\mathbf{h}}(E_\mathbf{h}^*ME_\mathbf{i}^*B_{\mathbf{i}, \mathbf{h}\triangle\mathbf{i},\mathbf{h}})
=c_{\mathbf{h}, \mathbf{k}_{(1)}, \mathbf{i}}(E_\mathbf{h}^*ME_\mathbf{i}^*)\overline{k_{(\mathbf{h}\triangle\mathbf{i})\cap\mathbf{i}\cap\mathbf{k}_{(1)}}}\in\mathbb{F}^\times$$
by Theorems \ref{T;Theorem3.23} and \ref{T;Theorem3.13}. Theorem \ref{T;Theorem3.13} thus implies that $E_\mathbf{h}^*ME_\mathbf{i}^*B_{\mathbf{i}, \mathbf{h}\triangle\mathbf{i},\mathbf{h}}\neq O$. Equation \eqref{Eq;4}
and Lemma \ref{L;Lemma2.6} thus show that $E_\mathbf{h}^*ME_\mathbf{i}^*B_{\mathbf{i}, \mathbf{h}\triangle\mathbf{i},\mathbf{h}}\!\in\!\mathrm{Rad}(E_\mathbf{h}^*\mathbb{T}E_\mathbf{h}^*)\!\setminus\!\{O\}$. This is a contradiction. The desired theorem follows from Lemmas \ref{L;Lemma2.6} and \ref{L;Lemma2.12}.
\end{proof}
\begin{cor}\label{C;Corollary5.21}
The following are equivalent: $\mathbb{S}$ is a $p'$-valenced scheme; $\mathbb{T}$ is a semisimple $\F$-algebra; $E_\mathbf{g}^*\mathbb{T}E_\mathbf{g}^*$ is a semisimple $\F$-subalgebra of $\mathbb{T}$ for any $\mathbf{g}\in\mathbb{E}$; $E_\mathbf{1}^*\mathbb{T}E_\mathbf{1}^*$ is a semisimple $\F$-subalgebra of $\mathbb{T}$. In particular, $\mathbb{T}$ is a semisimple $\F$-algebra if and only if $p\nmid\prod_{h=1}^n(|\mathbb{U}_h|-1)$.
\end{cor}
\begin{proof}
Recall that $k_\mathbf{0}=1$ and $k_\mathbf{g}=\prod_{i\in\mathbbm{g}}(|\mathbb{U}_i|-1)$ if $\mathbf{g}\in\mathbb{E}\setminus\{\mathbf{0}\}$. For any $\mathbf{g}\in\mathbb{E}$, Corollary \ref{C;Corollary5.19} implies that $p\nmid k_\mathbf{g}$ if and only if $E_\mathbf{g}^*\mathbb{T}E_\mathbf{g}^*$ is a semisimple $\F$-subalgebra of $\mathbb{T}$. It is obvious that $p\nmid k_\mathbf{1}$ if and only if $\mathbb{S}$ is a $p'$-valenced scheme. The desired corollary follows from combining Lemma \ref{L;Lemma2.6}, Corollary \ref{C;Corollary5.19}, and Theorem \ref{T;Semisimplicity}.
\end{proof}
\begin{cor}\label{C;Corollary5.22}
$\mathbb{T}$ is a semisimple $\F$-algebra if and only if $\mathrm{Z}(\mathbb{T})$ is a semisimple $\F$-subalgebra of $\mathbb{T}$. More explicitly, the $\F$-subalgebra $\mathrm{Z}(\mathbb{T})$ of $\mathbb{T}$ satisfies the
formulas $\mathrm{Rad}(\mathrm{Z}(\mathbb{T}))\!\!=\!\!\langle\{C_\mathbf{a}:(\mathbf{1}, \mathbf{a}, \mathbf{1})\in\mathbb{P}, p\mid k_\mathbf{a}\}\rangle_\mathbb{T}$ and $\mathrm{Z}(\mathbb{T})/\mathrm{Rad}(\mathrm{Z}(\mathbb{T}))\cong 2^{n_{\mathbf{1}\mathbf{1},\mathbf{0}}}\mathrm{M}_1(\F)$ as $\F$-algebras. The nilpotent index of $\mathrm{Rad}(\mathrm{Z}(\mathbb{T}))$ equals $|\{a: a\in[1, n], p\mid |\mathbb{U}_a|-1\}|+1$.
\end{cor}
\begin{proof}
The first statement is from combining Lemma \ref{L;Lemma2.6}, Corollaries \ref{C;Corollary4.15}, and \ref{C;Corollary5.21}. The second statement is from combining Corollaries \ref{C;Corollary4.14}, \ref{C;Corollary4.15}, Theorems \ref{T;Theorem5.15}, \ref{T;Theorem3.13}, and \ref{T;Theorem4.13}. The desired corollary follows from Corollary \ref{C;Corollary4.15} and Theorem \ref{T;Theorem5.15}.
\end{proof}
We finish the discussion of this section by an example of Theorems \ref{T;Theorem5.15} and \ref{T;Semisimplicity}.
\begin{eg}\label{E;Example5.23}
Assume that $n=|\mathbb{U}_1|=2$ and $|\mathbb{U}_2|=3$. It is clear that $|\mathbb{E}|=4$. Assume that $\mathbf{g}=(0,1)$ and $\mathbf{h}=(1,0)$. Hence $\mathbb{E}=\{\mathbf{0}, \mathbf{g}, \mathbf{h}, \mathbf{1}\}$.
By Theorem \ref{T;Theorem5.15}, the $\F$-subalgebras $E_\mathbf{0}^*\mathbb{T}E_\mathbf{0}^*, E_\mathbf{h}^*\mathbb{T}E_\mathbf{h}^*$ of $\mathbb{T}$ satisfy the formula $E_\mathbf{0}^*\mathbb{T}E_\mathbf{0}^*\!\cong\! E_\mathbf{h}^*\mathbb{T}E_\mathbf{h}^*\!\cong\!\mathrm{M}_1(\F)$ as $\F$-algebras. If $p\neq 2$, Theorem \ref{T;Theorem5.15} implies that the $\F$-subalgebras $E_\mathbf{g}^*\mathbb{T}E_\mathbf{g}^*,E_\mathbf{1}^*\mathbb{T}E_\mathbf{1}^*$ of $\mathbb{T}$ satisfy the formula $E_\mathbf{g}^*\mathbb{T}E_\mathbf{g}^*\!\cong\!E_\mathbf{1}^*\mathbb{T}E_\mathbf{1}^*\!\cong\!2\mathrm{M}_1(\F)$ as $\F$-algebras. If $p=2$, Theorem \ref{T;Theorem5.15} also implies that the $\F$-subalgebras $E_\mathbf{g}^*\mathbb{T}E_\mathbf{g}^*,E_\mathbf{1}^*\mathbb{T}E_\mathbf{1}^*$ of $\mathbb{T}$ satisfy the formula
$E_\mathbf{g}^*\mathbb{T}E_\mathbf{g}^*/\mathrm{Rad}(E_\mathbf{g}^*\mathbb{T}E_\mathbf{g}^*)\cong E_\mathbf{1}^*\mathbb{T}E_\mathbf{1}^*/\mathrm{Rad}(E_\mathbf{1}^*\mathbb{T}E_\mathbf{1}^*)\cong \mathrm{M}_1(\F)$
as $\F$-algebras. According to Theorem \ref{T;Semisimplicity}, it is clear to see that $\mathbb{T}$ is a semisimple $\F$-algebra if and only if $p\neq 2$.
\end{eg}
\section{Algebraic structure of $\mathbb{T}$: Jacobson radical}
In this section, we determine $\mathrm{Rad}(\mathbb{T})$ and compute the nilpotent index of $\mathrm{Rad}(\mathbb{T})$. We first recall Notations \ref{N;Notation3.1}, \ref{N;Notation3.9}, \ref{N;Notation3.14}, \ref{N;Notation3.15}, \ref{N;Notation4.1} and introduce a preliminary lemma.
\begin{lem}\label{L;Lemma6.1}
$\mathbb{T}$ has a two-sided ideal $\langle\{B_{\mathbf{a}, \mathbf{b}, \mathbf{c}}: (\mathbf{a}, \mathbf{b}, \mathbf{c})\in\mathbb{P}, p\mid k_\mathbf{b}\}\rangle_\mathbb{T}$. In particular, the $\F$-dimension of $\langle\{B_{\mathbf{a}, \mathbf{b}, \mathbf{c}}\!: (\mathbf{a}, \mathbf{b}, \mathbf{c})\!\in\!\mathbb{P}, p\mid k_\mathbf{b}\}\rangle_\mathbb{T}$ is independent of the choice of $\mathbf{x}$.
\end{lem}
\begin{proof}
Recall that $k_\mathbf{0}=1$ and $k_\mathbf{g}=\prod_{h\in\mathbbm{g}}(|\mathbb{U}_h|-1)$ if $\mathbf{g}\in\mathbb{E}\setminus\{\mathbf{0}\}$. Assume that $\mathbf{i}, \mathbf{j}, \mathbf{k}, \mathbf{l}, \mathbf{m},\mathbf{q}\in\mathbb{E}$, $\mathbf{j}\in\mathbb{P}_{\mathbf{i}, \mathbf{k}}$, $\mathbf{m}\in\mathbb{P}_{\mathbf{l}, \mathbf{q}}$, $p\mid k_\mathbf{m}$. If $\mathbf{k}\neq\mathbf{l}$,
notice that $B_{\mathbf{i}, \mathbf{j}, \mathbf{k}}B_{\mathbf{l}, \mathbf{m}, \mathbf{q}}=O$ by Equations \eqref{Eq;4} and \eqref{Eq;2}. Consider the case $\mathbf{k}=\mathbf{l}$.
As $\mathbf{j}\in\mathbb{P}_{\mathbf{i}, \mathbf{k}}$, it is obvious to see that $\mathbbm{j}=(\mathbbm{i}\triangle\mathbbm{k})\cup(\mathbbm{i}\cap\mathbbm{j}\cap\mathbbm{k})=(\mathbbm{i}\triangle\mathbbm{k})\cup(\mathbbm{i}\cap\mathbbm{j}\cap\mathbbm{k})^\circ$.
As $\mathbf{m}\in\mathbb{P}_{\mathbf{k}, \mathbf{q}}$, it is obvious to see that $\mathbbm{m}=(\mathbbm{k}\triangle\mathbbm{q})\cup(\mathbbm{k}\cap\mathbbm{m}\cap\mathbbm{q})=
(\mathbbm{k}\triangle\mathbbm{q})\cup(\mathbbm{k}\cap\mathbbm{m}\cap\mathbbm{q})^\circ$. As $p\mid k_\mathbf{m}$, there is $r\in\mathbbm{m}$ such that $p\mid |\mathbb{U}_r|-1$. If
$r\in(\mathbbm{k}\setminus(\mathbbm{i}\cup\mathbbm{q}))\cup((\mathbbm{k}\cap\mathbbm{m}\cap\mathbbm{q})^\circ\setminus\mathbbm{i})$, then $r\in\mathbbm{j}\cap\mathbbm{k}\cap\mathbbm{m}$ and $p\mid k_{\mathbf{j}\cap\mathbf{k}\cap\mathbf{m}}$. If
$r\in((\mathbbm{i}\cap\mathbbm{k})\setminus\mathbbm{q})\cup(\mathbbm{q}\setminus(\mathbbm{i}\cup\mathbbm{k}))\cup(\mathbbm{i}\cap\mathbbm{k}\cap\mathbbm{m}\cap\mathbbm{q})^\circ$,
then $r\in(\mathbbm{i}\triangle\mathbbm{q})\cup(\mathbbm{i}\cap\mathbbm{k}\cap\mathbbm{m}\cap\mathbbm{q})^\circ$ and $p\mid k_{[\mathbf{i}, \mathbf{j}, \mathbf{k}, \mathbf{m}, \mathbf{q}]}$.
If $r\in(\mathbbm{i}\cap\mathbbm{q})\setminus\mathbbm{k}$, notice that $r\in(\mathbbm{i}\cap\mathbbm{q})^\circ\setminus\mathbbm{k}$ and $p\mid k_{[\mathbf{i}, \mathbf{j}, \mathbf{k}, \mathbf{m}, \mathbf{q}]}$ as
$p\mid |\mathbb{U}_r|-1$. In conclusion, $p\mid k_{\mathbf{j}\cap\mathbf{k}\cap\mathbf{m}}k_{[\mathbf{i}, \mathbf{j}, \mathbf{k}, \mathbf{m}, \mathbf{q}]}$.
This thus implies that $B_{\mathbf{i}, \mathbf{j}, \mathbf{k}}B_{\mathbf{k}, \mathbf{m},\mathbf{q}}\in\langle\{B_{\mathbf{a}, \mathbf{b}, \mathbf{c}}: (\mathbf{a}, \mathbf{b}, \mathbf{c})\in\mathbb{P}, p\mid k_\mathbf{b}\}\rangle_\mathbb{T}$ by Theorem \ref{T;Theorem3.23}. Hence the above distinct cases show that $B_{\mathbf{i}, \mathbf{j}, \mathbf{k}}B_{\mathbf{l}, \mathbf{m}, \mathbf{q}}\in\langle\{B_{\mathbf{a}, \mathbf{b}, \mathbf{c}}: (\mathbf{a}, \mathbf{b}, \mathbf{c})\in\mathbb{P}, p\mid k_\mathbf{b}\}\rangle_\mathbb{T}$. If $M\in\langle\{B_{\mathbf{a}, \mathbf{b}, \mathbf{c}}: (\mathbf{a}, \mathbf{b}, \mathbf{c})\in\mathbb{P}, p\mid k_\mathbf{b}\}\rangle_\mathbb{T}$, then $M^T\in\langle\{B_{\mathbf{a}, \mathbf{b}, \mathbf{c}}: (\mathbf{a}, \mathbf{b}, \mathbf{c})\in\mathbb{P}, p\mid k_\mathbf{b}\}\rangle_\mathbb{T}$ by Lemma \ref{L;Lemma3.10}.
This thus shows that $B_{\mathbf{l}, \mathbf{m}, \mathbf{q}}B_{\mathbf{i}, \mathbf{j}, \mathbf{k}}\in\langle\{B_{\mathbf{a}, \mathbf{b}, \mathbf{c}}: (\mathbf{a}, \mathbf{b}, \mathbf{c})\in\mathbb{P}, p\mid k_\mathbf{b}\}\rangle_\mathbb{T}$ by the above discussion and Lemma \ref{L;Lemma3.10}. The first statement thus can be checked by the proved fact $B_{\mathbf{i}, \mathbf{j}, \mathbf{k}}B_{\mathbf{l}, \mathbf{m}, \mathbf{q}}, B_{\mathbf{l}, \mathbf{m}, \mathbf{q}}B_{\mathbf{i}, \mathbf{j}, \mathbf{k}}\in\langle\{B_{\mathbf{a}, \mathbf{b}, \mathbf{c}}: (\mathbf{a}, \mathbf{b}, \mathbf{c})\in\mathbb{P}, p\mid k_\mathbf{b}\}\rangle_\mathbb{T}$ and Theorem \ref{T;Theorem3.13}. The desired lemma follows from an  application of Theorem \ref{T;Theorem3.13}.
\end{proof}
Lemma \ref{L;Lemma6.1} motivates us to introduce the following notation and another lemma.
\begin{nota}\label{N;Notation6.2}
Let $\mathbb{I}=\langle\{B_{\mathbf{a}, \mathbf{b}, \mathbf{c}}: (\mathbf{a}, \mathbf{b}, \mathbf{c})\in\mathbb{P}, p\mid k_\mathbf{b}\}\rangle_\mathbb{T}$. So Lemma \ref{L;Lemma6.1} implies that $\mathbb{I}$ is a two-sided ideal of $\mathbb{T}$ whose $\F$-dimension is independent of the choice of $\mathbf{x}$.
\end{nota}
\begin{lem}\label{L;Lemma6.3}
There are $2|\{a: a\in[1,n], p\mid |\mathbb{U}_a|-1\}|$ matrices contained in $\mathbb{I}$ such that a particular matrix product of all these matrices is not equal to the zero matrix.
\end{lem}
\begin{proof}
Recall that $k_\mathbf{0}=1$ and $k_\mathbf{g}=\prod_{h\in\mathbbm{g}}(|\mathbb{U}_h|-1)$ if $\mathbf{g}\in\mathbb{E}\setminus\{\mathbf{0}\}$. Lemma \ref{L;Lemma2.3} thus implies that $\mathbb{I}\neq\{O\}$ if and only if $\{a: a\in[1,n], p\mid |\mathbb{U}_a|-1\}\neq\varnothing$. Therefore there is no loss to only consider the case $\mathbb{I}\neq\{O\}$. Define $\mathbb{U}=\{a: a\in[1,n], p\mid |\mathbb{U}_a|-1\}$. Assume that $i\in\mathbb{N}$ and $\mathbb{U}=\{j_1, j_2, \ldots, j_i\}$. For any $k\in[1,i]$, let $\mathbf{l}_{(k)}$ be the $n$-tuples in $\mathbb{E}$ whose unique nonzero entry is the $j_k$th-entry. So $\mathbf{l}_{(k)}\in\mathbb{P}_{\mathbf{1}, \mathbf{1}}$ and $p\mid k_{\mathbf{l}_{(k)}}$ for any $k\in[1, i]$. By Corollary \ref{C;Corollary3.24}, the matrix product of $B_{\mathbf{1}, \mathbf{l}_{(1)}, \mathbf{1}}, B_{\mathbf{1}, \mathbf{l}_{(2)}, \mathbf{1}}, \ldots, B_{\mathbf{1}, \mathbf{l}_{(i)}, \mathbf{1}}$ is independent of the multiplication order of $B_{\mathbf{1}, \mathbf{l}_{(1)}, \mathbf{1}}, B_{\mathbf{1}, \mathbf{l}_{(2)}, \mathbf{1}}, \ldots, B_{\mathbf{1}, \mathbf{l}_{(i)}, \mathbf{1}}$. Notice that
\begin{align}\label{Eq;11}
\prod_{k=1}^iB_{\mathbf{1}, \mathbf{l}_{(k)}, \mathbf{1}}\neq O
\end{align}
by Corollary \ref{C;Corollary3.24} and Lemma \ref{L;Lemma3.10}. For any $k\in [1,i]$, notice that $\mathbf{l}_{(k)}\in\mathbb{P}_{\mathbf{1},\mathbf{1}\setminus\mathbf{l}_{(k)}}$ and $B_{\mathbf{1}, \mathbf{l}_{(k)}, \mathbf{1}\setminus\mathbf{l}_{(k)}}B_{\mathbf{1}\setminus\mathbf{l}_{(k)}, \mathbf{l}_{(k)}, \mathbf{1}}=B_{\mathbf{1}, \mathbf{l}_{(k)}, \mathbf{1}}$ by combining Lemmas \ref{L;Lemma2.5}, \ref{L;Lemma2.3}, \ref{L;Lemma3.10}, and Theorem \ref{T;Theorem3.23}. The desired lemma follows from the above discussion and Equation \eqref{Eq;11}.
\end{proof}
The following lemmas allow us to show that $\mathbb{I}$ is a nilpotent two-sided ideal of $\mathbb{T}$.
\begin{lem}\label{L;Lemma6.4}
Assume that $\mathbf{g}, \mathbf{h}, \mathbf{i}, \mathbf{j}, \mathbf{k}\in\mathbb{E}$, $\mathbf{h}\in\mathbb{P}_{\mathbf{g}, \mathbf{i}}$, $\mathbf{j}\in\mathbb{P}_{\mathbf{i}, \mathbf{k}}$. Assume that $p\nmid k_{\mathbf{h}\cap\mathbf{i}\cap\mathbf{j}}$.
Then $\{a: a\in\mathbbm{h}\cup\mathbbm{j}, p\mid |\mathbb{U}_a|-1\}\subseteq(\mathbbm{g}\triangle\mathbbm{k})\cup((\mathbbm{g}\cap\mathbbm{k})^\circ\setminus\mathbbm{i})\cup((\mathbbm{h}
\cup\mathbbm{j})\cap(\mathbbm{g}\cap\mathbbm{i}\cap\mathbbm{k})^\circ)$.
\end{lem}
\begin{proof}
Recall that $k_\mathbf{0}=1$ and $k_\mathbf{g}=\prod_{\ell\in\mathbbm{g}}(|\mathbb{U}_\ell|-1)$ if $\mathbf{g}\in\mathbb{E}\setminus\{\mathbf{0}\}$. As $\mathbf{h}\in\mathbb{P}_{\mathbf{g}, \mathbf{i}}$, it is obvious to see that $\mathbbm{h}=(\mathbbm{g}\triangle\mathbbm{i})\cup(\mathbbm{g}\cap\mathbbm{h}\cap\mathbbm{i})=(\mathbbm{g}\triangle\mathbbm{i})\cup(\mathbbm{g}\cap\mathbbm{h}\cap\mathbbm{i})^\circ$.
As $\mathbf{j}\in\mathbb{P}_{\mathbf{i}, \mathbf{k}}$, it is obvious to see that $\mathbbm{j}=(\mathbbm{i}\triangle\mathbbm{k})\cup(\mathbbm{i}\cap\mathbbm{j}\cap\mathbbm{k})\!=\!(\mathbbm{i}\triangle\mathbbm{k})\cup(\mathbbm{i}\cap\mathbbm{j}\cap\mathbbm{k})^\circ$.
There is no loss to only consider the case $\{a: a\in\mathbbm{h}\cup\mathbbm{j}, p\mid |\mathbb{U}_a|-1\}\neq\varnothing$. Assume that $m\in\{a: a\in\mathbbm{h}\cup\mathbbm{j}, p\mid |\mathbb{U}_a|-1\}$. If $m\in(\mathbbm{g}\setminus(\mathbbm{i}\cup\mathbbm{k}))\cup((\mathbbm{g}\cap\mathbbm{i})\setminus\mathbbm{k})\cup((\mathbbm{g}\cap\mathbbm{h}\cap\mathbbm{i})^\circ\setminus\mathbbm{k})$, then $m\in\mathbbm{g}\setminus\mathbbm{k}$. If $m\in((\mathbbm{i}\cap\mathbbm{k})\setminus\mathbbm{g})\cup(\mathbbm{k}\setminus(\mathbbm{g}\cup\mathbbm{i}))
\cup((\mathbbm{i}\cap\mathbbm{j}\cap\mathbbm{k})^\circ\setminus\mathbbm{g})$, then $m\in\mathbbm{k}\setminus\mathbbm{g}$.
If $m\in(\mathbbm{g}\cap\mathbbm{k})\setminus\mathbbm{i}$, notice that $m\in(\mathbbm{g}\cap\mathbbm{k})^\circ\setminus\mathbbm{i}$ as $p\mid |\mathbb{U}_m|-1$.
As $p\nmid k_{\mathbf{h}\cap\mathbf{i}\cap\mathbf{j}}$, notice that $m\notin\mathbbm{i}\setminus(\mathbbm{g}\cup\mathbbm{k})$.
The case $m\in(\mathbbm{g}\cap\mathbbm{h}\cap\mathbbm{i}\cap\mathbbm{k})^\circ\cup(\mathbbm{g}\cap\mathbbm{i}\cap\mathbbm{j}\cap\mathbbm{k})^\circ$
is obvious. As $m$ is chosen from $\{a: a\in\mathbbm{h}\cup\mathbbm{j}, p\mid |\mathbb{U}_a|-1\}$ arbitrarily, the desired lemma follows.
\end{proof}
\begin{lem}\label{L;Lemma6.5}
Assume that $\mathbf{g}, \mathbf{h}, \mathbf{i}, \mathbf{j}, \mathbf{k}, \mathbf{l}, \mathbf{m}\in\mathbb{E}$, $\mathbf{h}\in\mathbb{P}_{\mathbf{g}, \mathbf{i}}$, $\mathbf{j}\in\mathbb{P}_{\mathbf{i}, \mathbf{k}}$, $\mathbf{l}\in\mathbb{P}_{\mathbf{k},\mathbf{m}}$. Assume that there is $q\in \mathbbm{h}\cap\mathbbm{j}\cap\mathbbm{l}$ such that
$p\mid |\mathbb{U}_q|-1$. Then $B_{\mathbf{g}, \mathbf{h}, \mathbf{i}}B_{\mathbf{i}, \mathbf{j}, \mathbf{k}}B_{\mathbf{k}, \mathbf{l}, \mathbf{m}}=O$.
\end{lem}
\begin{proof}
Recall that $k_\mathbf{0}=1$ and $k_\mathbf{g}=\prod_{q\in\mathbbm{g}}(|\mathbb{U}_q|-1)$ if $\mathbf{g}\in\mathbb{E}\setminus\{\mathbf{0}\}$. As $\mathbf{j}\in\mathbb{P}_{\mathbf{i}, \mathbf{k}}$, it is obvious to see that
$\mathbbm{j}=(\mathbbm{i}\triangle\mathbbm{k})\cup(\mathbbm{i}\cap\mathbbm{j}\cap\mathbbm{k})=(\mathbbm{i}\triangle\mathbbm{k})\cup(\mathbbm{i}\cap\mathbbm{j}\cap\mathbbm{k})^\circ$. If
$q\in(\mathbbm{i}\setminus\mathbbm{k})\cup(\mathbbm{i}\cap\mathbbm{j}\cap\mathbbm{k})^\circ$, then $q\in\mathbbm{h}\cap\mathbbm{i}\cap\mathbbm{j}$ and $p\mid k_{\mathbf{h}\cap\mathbf{i}\cap\mathbf{j}}$.
If $q\in\mathbbm{k}\setminus\mathbbm{i}$, then $q\in\mathbbm{j}\cap\mathbbm{k}\cap\mathbbm{l}$ and $p\mid k_{\mathbf{j}\cap\mathbf{k}\cap\mathbf{l}}$. In conclusion, $p\mid k_{\mathbf{h}\cap\mathbf{i}\cap\mathbf{j}}k_{\mathbf{j}\cap\mathbf{k}\cap\mathbf{l}}$. The desired lemma follows from Theorem \ref{T;Theorem3.23}.
\end{proof}
\begin{lem}\label{L;Lemma6.6}
Assume that $g\in\mathbb{N}\setminus[1,2]$ and $\mathbf{i}_{(h)}, \mathbf{j}_{(h)}, \mathbf{k}_{(h)}\in\mathbb{E}$ for any $h\in[1, g]$. Assume that $\mathbf{j}_{(h)}\in\mathbb{P}_{\mathbf{i}_{(h)}, \mathbf{k}_{(h)}}$ for any $h\in[1, g]$. Assume that there are pairwise distinct $\mathbf{l}, \mathbf{m}, \mathbf{q}\!\in\!\{\mathbf{j}_{(1)}, \mathbf{j}_{(2)},\ldots, \mathbf{j}_{(g)}\}$ such that $p\mid |\mathbb{U}_r|\!-\!1$ for some $r\!\in\!\mathbbm{l}\cap\mathbbm{m}\cap\mathbbm{q}$. Then
$$B_{\mathbf{i}_{(1)}, \mathbf{j}_{(1)}, \mathbf{k}_{(1)}}B_{\mathbf{i}_{(2)}, \mathbf{j}_{(2)}, \mathbf{k}_{(2)}}\cdots B_{\mathbf{i}_{(g)}, \mathbf{j}_{(g)}, \mathbf{k}_{(g)}}=O.$$
\end{lem}
\begin{proof}
Assume that the desired equation does not hold. As $\mathbf{l}, \mathbf{m}, \mathbf{q}$ are pairwise distinct, there is no loss to require that the subscripts of $\mathbf{l}, \mathbf{m}, \mathbf{q}$ increase strictly. So $\mathbf{m}=\mathbf{j}_{(s)}$ for some $s\in [2, g-1]$. By combining Theorem \ref{T;Theorem3.23}, Lemmas \ref{L;Lemma6.4}, \ref{L;Lemma6.5}, there are $t\in\mathbb{F}^\times$ and $\mathbf{u}, \mathbf{v}\in\mathbb{E}$ such that $\mathbf{u}\!\in\!\mathbb{P}_{\mathbf{i}_{(1)},\mathbf{i}_{(s)}}$, $\mathbf{v}\!\in\!\mathbb{P}_{\mathbf{k}_{(s)},\mathbf{k}_{(g)}}$, $r\!\in\!\mathbbm{m}\cap\mathbbm{u}\cap\mathbbm{v}$, and
$$B_{\mathbf{i}_{(1)}, \mathbf{j}_{(1)}, \mathbf{k}_{(1)}}B_{\mathbf{i}_{(2)}, \mathbf{j}_{(2)}, \mathbf{k}_{(2)}}\cdots B_{\mathbf{i}_{(g)}, \mathbf{j}_{(g)}, \mathbf{k}_{(g)}}=tB_{\mathbf{i}_{(1)},\mathbf{u} ,\mathbf{i}_{(s)}}B_{\mathbf{i}_{(s)},\mathbf{m} ,\mathbf{k}_{(s)}}B_{\mathbf{k}_{(s)},\mathbf{v} ,\mathbf{k}_{(g)}}=O.$$
This is a clear contradiction. The desired lemma follows from this contradiction.
\end{proof}
\begin{lem}\label{L;Lemma6.7}
Assume that $g\in\mathbb{N}$. Assume that $\mathbb{U}$ is a set and the cardinality of $\mathbb{U}$ is equal to $g$. Assume that $\mathbb{V}_1, \mathbb{V}_2,\ldots,\mathbb{V}_{2g+1}$ is a sequence of nonempty subsets of $\mathbb{U}$. Then there are pairwise distinct $h,i,j\in[1, 2g+1]$ such that $\mathbb{V}_h\cap\mathbb{V}_i\cap\mathbb{V}_j\neq\varnothing$.
\end{lem}
\begin{proof}
Work by induction on $g$. If $g=1$, then $\mathbb{V}_1\cap\mathbb{V}_2\cap\mathbb{V}_3=\mathbb{U}\neq\varnothing$. The base case is  checked. Assume that $g>1$ and any sequence of $2g-1$ nonempty subsets of a set with cardinality $g-1$ has a nonempty intersection of three members. By the Pigeonhole Principle, there are $k\in\mathbb{U}$ and distinct $\ell, m\in[1, 2g+1]$ such that $k\in\mathbb{V}_\ell\cap\mathbb{V}_m$. There is no loss to let $\ell<m$. If $k\in\mathbb{V}_q$ for some $q\in[1, 2g+1]\setminus\{\ell, m\}$, then the desired inequality holds. Otherwise, notice that there is a sequence of the subsets $\mathbb{V}_1, \mathbb{V}_2, \ldots, \mathbb{V}_{\ell-1}, \mathbb{V}_{\ell+1}, \ldots, \mathbb{V}_{m-1}, \mathbb{V}_{m+1},\ldots, \mathbb{V}_{2g+1}$ of $\mathbb{U}\setminus\{k\}$. Hence the desired lemma follows from the above discussion and the inductive hypothesis.
\end{proof}
\begin{lem}\label{L;Lemma6.8}
The matrix product of any $2|\{a: a\in[1,n], p\mid |\mathbb{U}_a|-1\}|+1$ matrices contained in $\mathbb{I}$ is equal to the zero matrix. In particular, $\mathbb{I}$ is a nilpotent two-sided ideal of $\mathbb{T}$ and
the nilpotent index of $\mathbb{I}$ is equal to $2|\{a: a\in[1,n], p\mid |\mathbb{U}_a|-1\}|+1$.
\end{lem}
\begin{proof}
Recall that $k_\mathbf{0}=1$ and $k_\mathbf{g}=\prod_{h\in\mathbbm{g}}(|\mathbb{U}_h|-1)$ if $\mathbf{g}\in\mathbb{E}\setminus\{\mathbf{0}\}$. Assume that $\mathbb{U}=\{a: a\in[1,n], p\mid |\mathbb{U}_a|-1\}$ and $i=|\mathbb{U}|$. If $i=0$, notice that $\mathrm{Rad}(\mathbb{T})=\mathbb{I}=\{O\}$ by Corollary \ref{C;Corollary5.21} and Lemma \ref{L;Lemma2.6}. So there is no loss to assume further that $i>0$. For any $j\in [1, 2i+1]$, assume that $\mathbf{k}_{(j)}, \mathbf{l}_{(j)}, \mathbf{m}_{(j)}\in\mathbb{E}$, $\mathbf{l}_{(j)}\in\mathbb{P}_{\mathbf{k}_{(j)},\mathbf{m}_{(j)}}$,
$p\mid k_{\mathbf{l}_{(j)}}$. For any $\mathbf{q}\in\{\mathbf{l}_{(1)}, \mathbf{l}_{(2)}, \ldots, \mathbf{l}_{(2i+1)}\}$,
let $\mathbb{V}_\mathbf{q}=\{a: a\in\mathbbm{q}, p\mid |\mathbb{U}_a|-1\}$. As $p\mid k_\mathbf{q}$, notice that
$\varnothing\neq\mathbb{V}_\mathbf{q}\subseteq\mathbb{U}$. Lemma \ref{L;Lemma6.7} thus implies that $\varnothing\neq\mathbb{V}_\mathbf{r}\cap\mathbb{V}_\mathbf{s}\cap\mathbb{V}_\mathbf{t}\subseteq\mathbbm{r}\cap\mathbbm{s}\cap\mathbbm{t}$ for some pairwise distinct $\mathbf{r}, \mathbf{s}, \mathbf{t}\in\{\mathbf{l}_{(1)}, \mathbf{l}_{(2)},\ldots, \mathbf{l}_{(2i+1)}\}$.
Hence Lemma \ref{L;Lemma6.6} implies that
\begin{align}\label{Eq;12}
B_{\mathbf{k}_{(1)}, \mathbf{l}_{(1)}, \mathbf{m}_{(1)}}B_{\mathbf{k}_{(2)}, \mathbf{l}_{(2)}, \mathbf{m}_{(2)}}\cdots B_{\mathbf{k}_{(2i+1)}, \mathbf{l}_{(2i+1)}, \mathbf{m}_{(2i+1)}}=O.
\end{align}
As Theorem \ref{T;Theorem3.13} implies that $\mathbb{I}$ has an $\F$-basis $\{B_{\mathbf{a}, \mathbf{b}, \mathbf{c}}: (\mathbf{a}, \mathbf{b}, \mathbf{c})\in\mathbb{P}, p\mid k_\mathbf{b}\}$, the desired lemma follows from combining Equation \eqref{Eq;12}, Lemmas \ref{L;Lemma6.1}, and \ref{L;Lemma6.3}.
\end{proof}
We are now ready to deduce the main result of this section and another corollary.
\begin{thm}\label{T;Jacobson}
Assume that $M\in\mathrm{Rad}(\mathbb{T})$. Then $M\in\mathbb{I}$. In particular, $\mathrm{Rad}(\mathbb{T})=\mathbb{I}$. Moreover, the nilpotent index of $\mathrm{Rad}(\mathbb{T})$ is equal to $2|\{a: a\in[1,n], p\mid |\mathbb{U}_a|-1\}|+1$.
\end{thm}
\begin{proof}
Recall that $k_\mathbf{0}=1$ and $k_\mathbf{g}=\prod_{h\in\mathbbm{g}}(|\mathbb{U}_h|-1)$ if $\mathbf{g}\in\mathbb{E}\setminus\{\mathbf{0}\}$. Assume that $M\in\mathrm{Rad}(\mathbb{T})\setminus\mathbb{I}$. Equation \eqref{Eq;3} thus implies that $E_\mathbf{i}^*ME_\mathbf{j}^*\in\mathrm{Rad}(\mathbb{T})\setminus\mathbb{I}$ for some $\mathbf{i}, \mathbf{j}\in\mathbb{E}$. By Lemma \ref{L;Lemma6.8}, there is no loss to require that
$\mathrm{Supp}_{\mathbb{B}_2}(E_\mathbf{i}^*ME_\mathbf{j}^*)\cap\mathbb{I}=\varnothing$. By combining Equations \eqref{Eq;4}, \eqref{Eq;2}, Theorem \ref{T;Theorem3.13}, there are $k\!\in\!\mathbb{N}$ and pairwise
distinct $\mathbf{l}_{(1)}, \mathbf{l}_{(2)}, \ldots, \mathbf{l}_{(k)}\!\in\!\mathbb{P}_{\mathbf{i}, \mathbf{j}}$ such that $\mathrm{Supp}_{\mathbb{B}_2}(E_\mathbf{i}^*ME_\mathbf{j}^*)\!=\!\{B_{\mathbf{i},\mathbf{l}_{(1)},\mathbf{j}}, B_{\mathbf{i},\mathbf{l}_{(2)},\mathbf{j}},\ldots, B_{\mathbf{i},\mathbf{l}_{(k)},\mathbf{j}}\}$ and
$p\nmid k_{\mathbf{l}_{(m)}}$ for any $m\in [1, k]$. As $\mathbf{l}_{(1)}, \mathbf{l}_{(2)}, \ldots, \mathbf{l}_{(k)}$ are pairwise distinct $n$-tuples in $\mathbb{P}_{\mathbf{i},\mathbf{j}}$, Lemmas \ref{L;Lemma2.5} and \ref{L;Lemma2.3} imply that $\mathbf{i}\cap\mathbf{j}\cap\mathbf{l}_{(1)}, \mathbf{i}\cap\mathbf{j}\cap\mathbf{l}_{(2)},\ldots,\mathbf{i}\cap\mathbf{j}\cap\mathbf{l}_{(k)}$ are pairwise distinct $n$-tuples.
So $[\mathbf{i},\mathbf{l}_{(1)}, \mathbf{j}, \mathbf{i}\triangle\mathbf{j}, \mathbf{i}], [\mathbf{i},\mathbf{l}_{(2)}, \mathbf{j}, \mathbf{i}\triangle\mathbf{j}, \mathbf{i}],\ldots, [\mathbf{i},\mathbf{l}_{(k)}, \mathbf{j}, \mathbf{i}\triangle\mathbf{j}, \mathbf{i}]$
are pairwise distinct $n$-tuples by Lemma \ref{L;Lemma2.3}. As $\mathbf{l}_{(m)}\in\mathbb{P}_{\mathbf{i}, \mathbf{j}}$ and $p\nmid k_{\mathbf{l}_{(m)}}$ for any $m\in [1, k]$, notice that $p\nmid k_{[\mathbf{i},\mathbf{l}_{(m)}, \mathbf{j}, \mathbf{i}\triangle\mathbf{j}, \mathbf{i}]}$ for any $m\!\in\![1, k]$. Lemma \ref{L;Lemma2.5} shows that $\mathbf{i}\triangle\mathbf{j}\!\in\!\mathbb{P}_{\mathbf{i}, \mathbf{j}}$. Notice that
$$ c_{\mathbf{i},[\mathbf{i},\mathbf{l}_{(1)}, \mathbf{j}, \mathbf{i}\triangle\mathbf{j}, \mathbf{i}],\mathbf{i}}(E_\mathbf{i}^*ME_\mathbf{j}^*B_{\mathbf{j}, \mathbf{i}\triangle\mathbf{j}, \mathbf{i}})=c_{\mathbf{i}, \mathbf{l}_{(1)}, \mathbf{j}}(E_\mathbf{i}^*ME_\mathbf{j}^*)\overline{k_{(\mathbf{i}\triangle\mathbf{j})\cap\mathbf{j}\cap\mathbf{l}_{(1)}}}\in\F^\times$$
by Theorems \ref{T;Theorem3.23} and \ref{T;Theorem3.13}. Theorem \ref{T;Theorem3.13} thus implies that $E_\mathbf{i}^*ME_\mathbf{j}^*B_{\mathbf{j}, \mathbf{i}\triangle\mathbf{j}, \mathbf{i}}\neq O$. Equation \eqref{Eq;4} and Lemma \ref{L;Lemma2.6} thus imply that $E_\mathbf{i}^*ME_\mathbf{j}^*B_{\mathbf{j}, \mathbf{i}\triangle\mathbf{j}, \mathbf{i}}\!\in\!\mathrm{Rad}(E_\mathbf{i}^*\mathbb{T}E_\mathbf{i}^*)\setminus\{O\}$. This is an obvious contradiction by Theorems \ref{T;Theorem5.15} and \ref{T;Theorem3.13}. So the first statement follows. The desired theorem follows from the first statement and Lemma \ref{L;Lemma6.8}.
\end{proof}
\begin{cor}\label{C;Corollary6.10}
The $\F$-dimension of $\mathrm{Rad}(\mathbb{T})$ is independent of the choice of $\mathbf{x}$.
\end{cor}
\begin{proof}
The desired corollary follows from applying Theorem \ref{T;Jacobson} and Lemma \ref{L;Lemma6.1}.
\end{proof}
We are now ready to close this section by presenting an example of Theorem \ref{T;Jacobson}.
\begin{eg}\label{E;Example6.11}
Assume that $n=|\mathbb{U}_1|=2$ and $|\mathbb{U}_2|=3$. It is clear that $|\mathbb{E}|=4$. Assume that $\mathbf{g}=(0, 1)$ and $\mathbf{h}=(1,0)$. Hence $\mathbb{E}=\{\mathbf{0}, \mathbf{g}, \mathbf{h}, \mathbf{1}\}$. Assume that $p=2$ by Theorem \ref{T;Jacobson}. By Theorems \ref{T;Jacobson} and \ref{T;Theorem3.13}, an $\F$-basis of $\mathrm{Rad}(\mathbb{T})$ contains
exactly $B_{\mathbf{0}, \mathbf{g}, \mathbf{g}}$, $B_{\mathbf{0}, \mathbf{1}, \mathbf{1}}$, $B_{\mathbf{g}, \mathbf{g}, \mathbf{0}}$,
$B_{\mathbf{g}, \mathbf{g}, \mathbf{g}}$, $B_{\mathbf{g}, \mathbf{1}, \mathbf{h}}$, $B_{\mathbf{g}, \mathbf{1}, \mathbf{1}}$,
$B_{\mathbf{h},\mathbf{g}, \mathbf{1}}$, $B_{\mathbf{h}, \mathbf{1}, \mathbf{g}}$, $B_{\mathbf{1}, \mathbf{g}, \mathbf{h}}$,
$B_{\mathbf{1}, \mathbf{g}, \mathbf{1}}$, $B_{\mathbf{1}, \mathbf{1}, \mathbf{0}}$, $B_{\mathbf{1}, \mathbf{1}, \mathbf{g}}$.
By Theorem \ref{T;Jacobson}, it is obvious that the nilpotent index of $\mathrm{Rad}(\mathbb{T})$ is equal to three.
\end{eg}
\section{Algebraic structure of $\mathbb{T}$: Quotient $\F$-algebra}
In this section, we present an $\F$-basis of $\mathbb{T}/\mathrm{Rad}(\mathbb{T})$. Moreover, we determine the structure constants of this $\F$-basis in $\mathbb{T}/\mathrm{Rad}(\mathbb{T})$. As a preparation, we first recall Notations \ref{N;Notation3.1}, \ref{N;Notation3.9}, \ref{N;Notation3.15}, \ref{N;Notation4.1}, \ref{N;Notation5.5}, \ref{N;Notation5.6} and begin our presentation with three lemmas. %with three preliminary lemmas.
\begin{lem}\label{L;Lemma7.1}
Assume that $\mathbf{g}, \mathbf{h}, \mathbf{i}, \mathbf{j}, \mathbf{k}, \mathbf{l}\!\in\!\mathbb{E}$, $\mathbf{h}\in\mathbb{P}_{\mathbf{g}, \mathbf{i}}$, $\mathbf{k}\in\mathbb{P}_{\mathbf{j}, \mathbf{l}}$. Assume that $p\nmid k_\mathbf{h}k_\mathbf{k}$. Then $D_{\mathbf{g}, \mathbf{h}, \mathbf{i}}=D_{\mathbf{j}, \mathbf{k}, \mathbf{l}}$ if and only if $\mathbf{g}=\mathbf{j}$, $\mathbf{h}=\mathbf{k}$, and
$\mathbf{i}=\mathbf{l}$.
\end{lem}
\begin{proof}
Only check that $D_{\mathbf{g}, \mathbf{h}, \mathbf{i}}=D_{\mathbf{j}, \mathbf{k}, \mathbf{l}}$ only if $\mathbf{g}=\mathbf{j}$, $\mathbf{h}=\mathbf{k}$,
$\mathbf{i}=\mathbf{l}$. The desired lemma follows from combining Equations \eqref{Eq;10}, \eqref{Eq;4}, \eqref{Eq;2}, Theorem \ref{T;Theorem3.13}, Lemma \ref{L;Lemma2.3}.
\end{proof}
\begin{lem}\label{L;Lemma7.2}
$\mathbb{T}$ has an $\F$-linearly independent subset $\{D_{\mathbf{a}, \mathbf{b}, \mathbf{c}}: (\mathbf{a},\mathbf{b}, \mathbf{c})\in\mathbb{P}, p\nmid k_\mathbf{b}\}$.
\end{lem}
\begin{proof}
Set $\mathbb{U}=\{D_{\mathbf{a}, \mathbf{b}, \mathbf{c}}: (\mathbf{a},\mathbf{b}, \mathbf{c})\in\mathbb{P}, p\nmid k_\mathbf{b}\}$. It is obvious to see that $D_{\mathbf{0}, \mathbf{0},\mathbf{0}}\in\mathbb{U}$. Lemma \ref{L;Lemma5.7} thus implies that $M\neq O$ for any $M\!\in\!\mathbb{U}$. Let $L$ be a nonzero $\F$-linear combination of the matrices in $\mathbb{U}$. Assume that $L=O$. If $M\in\mathbb{U}$, let $c_M$ be the coefficient of $M$ in $L$. There is $N\in\mathbb{U}$ such that $c_N\in\F^\times$. The combination of Equations \eqref{Eq;10}, \eqref{Eq;4}, and \eqref{Eq;2} thus shows that $N=E_\mathbf{g}^*NE_\mathbf{h}^*$ for some $\mathbf{g}, \mathbf{h}\in\mathbb{E}$. So $\mathbb{V}\!=\!\{A: A\!\in\!\mathbb{U}, c_A\in\F^\times, A=E_\mathbf{g}^*AE_\mathbf{h}^*\}\neq\varnothing$. By Lemma \ref{L;Lemma7.1}, there are $i\in\mathbb{N}$ and pairwise distinct $\mathbf{j}_{(1)}, \mathbf{j}_{(2)}, \ldots, \mathbf{j}_{(i)}\!\!\in\!\!\mathbb{P}_{\mathbf{g}, \mathbf{h}}$ such that $\mathbb{V}\!=\!\{D_{\mathbf{g},\mathbf{j}_{(1)}, \mathbf{h}}, D_{\mathbf{g},\mathbf{j}_{(2)}, \mathbf{h}},\ldots,D_{\mathbf{g},\mathbf{j}_{(i)}, \mathbf{h}}\}$.

Assume that $i=1$. The combination of Equations \eqref{Eq;10}, \eqref{Eq;4}, and \eqref{Eq;2} gives $O=E_\mathbf{g}^*LE_\mathbf{h}^*=cD_{\mathbf{g}, \mathbf{j}_{(1)}, \mathbf{h}}$ for some $c\in\F^\times$.
This is a contradiction by Lemma \ref{L;Lemma5.7}. Assume further that $i>1$. By Lemma \ref{L;Lemma2.3}, there is no loss to
assume that $\mathbf{j}_{(1)}$ is minimal in $\{\mathbf{j}_{(1)}, \mathbf{j}_{(2)},\ldots, \mathbf{j}_{(i)}\}$ with respect to the partial order $\preceq$. As $E_\mathbf{g}^*LE_\mathbf{h}^*=O$, $D_{\mathbf{g}, \mathbf{j}_{(1)}, \mathbf{h}}$ is an $\F$-linear combination of the matrices in $\{D_{\mathbf{g},\mathbf{j}_{(2)} ,\mathbf{h}},D_{\mathbf{g},\mathbf{j}_{(3)} ,\mathbf{h}},\ldots, D_{\mathbf{g},\mathbf{j}_{(i)} ,\mathbf{h}}\}$. By Equation \eqref{Eq;10} and Theorem \ref{T;Theorem3.13}, there is $\mathbf{k}\in\{\mathbf{j}_{(2)}, \mathbf{j}_{(3)},\ldots, \mathbf{j}_{(i)}\}$ such that $\mathbf{k}\preceq\mathbf{j}_{(1)}$.
This implies that $\mathbf{j}_{(1)}\in\{\mathbf{j}_{(2)}, \mathbf{j}_{(3)},\ldots, \mathbf{j}_{(i)}\}$ by the choice of $\mathbf{j}_{(1)}$. This is a contradiction. These contradictions yield $L\neq O$. The desired lemma follows.
\end{proof}
\begin{lem}\label{L;Lemma7.3}
$|\{B_{\mathbf{a}, \mathbf{b}, \mathbf{c}}\!:\! (\mathbf{a},\mathbf{b}, \mathbf{c})\!\in\!\mathbb{P}, p\mid k_\mathbf{b}\}\cup\{D_{\mathbf{a}, \mathbf{b}, \mathbf{c}}\!:\! (\mathbf{a},\mathbf{b}, \mathbf{c})\in\mathbb{P}, p\nmid k_\mathbf{b}\}|=2^{2n_1}5^{n_2}$. Moreover, $\mathbb{T}$ has an $\F$-basis $\{B_{\mathbf{a}, \mathbf{b}, \mathbf{c}}\!:\!\  (\mathbf{a},\mathbf{b}, \mathbf{c})\!\in\!\mathbb{P}, p\!\mid\! k_\mathbf{b}\}\!\cup\!\{D_{\mathbf{a}, \mathbf{b}, \mathbf{c}}\!:\! (\mathbf{a},\mathbf{b}, \mathbf{c})\!\in\!\mathbb{P}, p\!\nmid\! k_\mathbf{b}\}$.
\end{lem}
\begin{proof}
The first statement follows from Lemma \ref{L;Lemma7.2} and Theorem \ref{T;Theorem3.13}. By the first statement and Theorem \ref{T;Theorem3.13}, it suffices to check that $\mathbb{T}$ has an $\F$-linearly independent subset $\{B_{\mathbf{a}, \mathbf{b}, \mathbf{c}}:(\mathbf{a},\mathbf{b}, \mathbf{c})\in\mathbb{P}, p\mid k_\mathbf{b}\}\cup\{D_{\mathbf{a}, \mathbf{b}, \mathbf{c}}: (\mathbf{a},\mathbf{b}, \mathbf{c})\in\mathbb{P}, p\nmid k_\mathbf{b}\}$. The desired lemma follows from combining Equation \eqref{Eq;10}, Theorem \ref{T;Theorem3.13}, and Lemma \ref{L;Lemma7.2}.
\end{proof}
Lemma \ref{L;Lemma7.3} allows us to deduce the first main result of this section and a notation.
\begin{thm}\label{T;Theorem7.4}
$\mathbb{T}/\mathrm{Rad}(\mathbb{T})$ has an $\F$-basis $\{D_{\mathbf{a}, \mathbf{b}, \mathbf{c}}+\mathrm{Rad}(\mathbb{T}):(\mathbf{a}, \mathbf{b}, \mathbf{c})\in\mathbb{P}, p\nmid k_\mathbf{b}\}$. In particular, the $\F$-dimension of $\mathbb{T}/\mathrm{Rad}(\mathbb{T})$ is also independent of the choice of $\mathbf{x}$.
\end{thm}
\begin{proof}
The first statement follows from combining Theorem \ref{T;Jacobson}, Lemmas \ref{L;Lemma6.1}, and \ref{L;Lemma7.3}. So the desired theorem follows from an application of the first statement.
\end{proof}
\begin{nota}\label{N;Notation7.5}
Let $\mathbb{B}_4=\{D_{\mathbf{a}, \mathbf{b}, \mathbf{c}}+\mathrm{Rad}(\mathbb{T}):(\mathbf{a},\mathbf{b}, \mathbf{c})\in\mathbb{P}, p\nmid k_\mathbf{b}\}$. If $M\in\mathbb{T}/\mathrm{Rad}(\mathbb{T})$, $(\mathbf{g}, \mathbf{h}, \mathbf{i})\in\mathbb{P}$, and an additional condition $p\nmid k_\mathbf{h}$ holds, notice that $\mathrm{Supp}_{\mathbb{B}_4}(M)$ is defined and write $c_{\mathbf{g}, \mathbf{i}}^{\mathbf{h}}(M)$ for $c_{\mathbb{B}_4}(M, D_{\mathbf{g}, \mathbf{h}, \mathbf{i}}\!+\!\mathrm{Rad}(\mathbb{T}))$ as a deduction of Theorem \ref{T;Theorem7.4}.
\end{nota}
We next use some lemmas to investigate the structure constants of $\mathbb{B}_4$ in $\mathbb{T}\!/\mathrm{Rad}(\mathbb{T})$.
\begin{lem}\label{L;Lemma7.6}
Assume that $\mathbf{g}, \mathbf{h}, \mathbf{i}, \mathbf{j}, \mathbf{k}\!\in\!\mathbb{E}$, $\mathbf{g}\triangle\mathbf{i}\!\preceq\!\mathbf{h}$, $\mathbf{i}\triangle\mathbf{k}\!\preceq\!\mathbf{j}$, $p\nmid k_\mathbf{h}k_\mathbf{j}$, $(\mathbbm{h}\cap\mathbbm{i}\cap\mathbbm{k})^\circ\!\!\subseteq\!\!\mathbbm{j}$. Then $p\!\nmid\! k_{[\mathbf{g}, \mathbf{h}, \mathbf{i}, \mathbf{j}, \mathbf{k}]}$. Moreover, $\mathbf{i}\mathbf{k}\setminus\mathbf{j}\!=\!\mathbf{g}\mathbf{k}\setminus[\mathbf{g}, \mathbf{h}, \mathbf{i}, \mathbf{j}, \mathbf{k}]$
and $n_{\mathbf{i}\mathbf{k}, \mathbf{j}}\!=\!n_{\mathbf{g}\mathbf{k}, [\mathbf{g}, \mathbf{h}, \mathbf{i}, \mathbf{j}, \mathbf{k}]}$
if $\mathbf{j}\!\in\!\mathbb{P}_{\mathbf{i}, \mathbf{k}}$.
\end{lem}
\begin{proof}
As $\mathbf{i}\triangle\mathbf{k}\preceq\mathbf{j}$ and the remaining assumptions in Lemma \ref{L;Lemma5.10} are satisfied by changing
$\mathbf{m}, \mathbb{U}$ to $\mathbf{j},\varnothing$, the first statement is from Lemma \ref{L;Lemma5.10}. For the remaining equations,
notice that $\mathbbm{j}=(\mathbbm{i}\triangle\mathbbm{k})\cup(\mathbbm{i}\cap\mathbbm{j}\cap\mathbbm{k})=
(\mathbbm{i}\triangle\mathbbm{k})\cup(\mathbbm{i}\cap\mathbbm{j}\cap\mathbbm{k})^\circ$ as $\mathbf{j}\in\mathbb{P}_{\mathbf{i}, \mathbf{k}}$. So
$((\mathbbm{i}\triangle\mathbbm{k})\cup(\mathbbm{i}\cap\mathbbm{k})^\circ)\setminus\mathbbm{j}=(\mathbbm{i}\cap\mathbbm{k})^\circ
\setminus\mathbbm{j}$. As $\mathbf{g}\triangle\mathbf{i}\preceq\mathbf{h}$ and $(\mathbbm{h}\cap\mathbbm{i}\cap\mathbbm{k})^\circ\subseteq\mathbbm{j}$, notice that $(\mathbbm{g}\triangle\mathbbm{k})
\cup((\mathbbm{g}\cap\mathbbm{k})^\circ\setminus\mathbbm{i})\cup((\mathbbm{h}\cup\mathbbm{j})\cap
(\mathbbm{g}\cap\mathbbm{i}\cap\mathbbm{k})^\circ)=(\mathbbm{g}\triangle\mathbbm{k})
\cup((\mathbbm{g}\cap\mathbbm{k})^\circ\setminus\mathbbm{i})\cup(\mathbbm{g}\cap\mathbbm{i}\cap\mathbbm{j}\cap\mathbbm{k})^\circ$
and $((\mathbbm{g}\triangle\mathbbm{k})\cup(\mathbbm{g}\cap\mathbbm{k})^\circ)\setminus((\mathbbm{g}\triangle\mathbbm{k})
\cup((\mathbbm{g}\cap\mathbbm{k})^\circ\setminus\mathbbm{i})\cup(\mathbbm{g}\cap\mathbbm{i}\cap\mathbbm{j}\cap\mathbbm{k})^\circ)=(\mathbbm{i}\cap\mathbbm{k})^\circ
\setminus\mathbbm{j}$. This implies that $\mathbf{i}\mathbf{k}\setminus\mathbf{j}=\mathbf{g}\mathbf{k}\setminus[\mathbf{g}, \mathbf{h}, \mathbf{i}, \mathbf{j}, \mathbf{k}]$ by Lemma \ref{L;Lemma2.3}. The desired lemma follows.
\end{proof}
\begin{lem}\label{L;Lemma7.7}
Assume that $\mathbf{g}, \mathbf{h}, \mathbf{i}, \mathbf{j}, \mathbf{k}, \mathbf{l}\!\in\!\mathbb{E}$, $\mathbf{g}\triangle\mathbf{i}\preceq\mathbf{h}$, $\mathbf{i}\triangle\mathbf{k}\!\preceq\!\mathbf{j}$, $p\!\nmid\! k_\mathbf{h}k_\mathbf{j}$, $m\!\in\![0, n_{\mathbf{i}\mathbf{k},\mathbf{j}}]$, $\mathbf{l}\!\in\!\mathbb{U}_{\mathbf{i}\mathbf{k}, \mathbf{j}, m}$, $(\mathbbm{h}\cap\mathbbm{i}\cap\mathbbm{k})^\circ\subseteq\mathbbm{j}$. Then $m\!\in\![0, n_{\mathbf{g}\mathbf{k}, [\mathbf{g}, \mathbf{h}, \mathbf{i}, \mathbf{j}, \mathbf{k}]}]$ and $[\mathbf{g},\mathbf{h},\mathbf{i}, \mathbf{l}, \mathbf{k}]\!\in\!\mathbb{U}_{\mathbf{g}\mathbf{k},[\mathbf{g}, \mathbf{h},\mathbf{i},\mathbf{j},\mathbf{k}], m}$.
\end{lem}
\begin{proof}
As $\mathbf{l}\in\mathbb{U}_{\mathbf{i}\mathbf{k}, \mathbf{j}, m}$, notice that $\mathbf{j}\preceq\mathbf{l}\preceq\mathbf{i}\mathbf{k}$, $p\nmid k_\mathbf{l}$, and $|\mathbbm{l}\setminus\mathbbm{j}|=m$.
Lemma \ref{L;Lemma3.16} thus implies that $[\mathbf{g}, \mathbf{h},\mathbf{i},\mathbf{j},\mathbf{k}]\preceq[\mathbf{g}, \mathbf{h},\mathbf{i},\mathbf{l},\mathbf{k}]\preceq\mathbf{g}\mathbf{k}$. As $\mathbf{i}\triangle\mathbf{k}\preceq\mathbf{j}\preceq\mathbf{l}\preceq\mathbf{i}\mathbf{k}$, notice that $\mathbf{j},\mathbf{l}\in\mathbb{P}_{\mathbf{i}, \mathbf{k}}$ and $\mathbf{i}\triangle\mathbf{k}\preceq\mathbf{l}$. As $p\nmid k_\mathbf{h}k_\mathbf{l}$ and $(\mathbbm{h}\cap\mathbbm{i}\cap\mathbbm{k})^\circ\subseteq\mathbbm{j}\subseteq\mathbbm{l}$, notice that $p\nmid k_{[\mathbf{g},\mathbf{h},\mathbf{i}, \mathbf{l}, \mathbf{k}]}$ by Lemma \ref{L;Lemma7.6}. As $\mathbf{j}, \mathbf{l}\in\mathbb{P}_{\mathbf{i}, \mathbf{k}}$, notice that $\mathbbm{j}=(\mathbbm{i}\triangle\mathbbm{k})\cup(\mathbbm{i}\cap\mathbbm{j}\cap\mathbbm{k})=
(\mathbbm{i}\triangle\mathbbm{k})\cup(\mathbbm{i}\cap\mathbbm{j}\cap\mathbbm{k})^\circ$ and $\mathbbm{l}=(\mathbbm{i}\triangle\mathbbm{k})\cup(\mathbbm{i}\cap\mathbbm{k}\cap\mathbbm{l})=
(\mathbbm{i}\triangle\mathbbm{k})\cup(\mathbbm{i}\cap\mathbbm{k}\cap\mathbbm{l})^\circ$. So $\mathbbm{l}\setminus\mathbbm{j}=(\mathbbm{i}\cap\mathbbm{k}\cap\mathbbm{l})^\circ\setminus\mathbbm{j}$. Notice that $(\mathbbm{g}\triangle\mathbbm{k})
\cup((\mathbbm{g}\cap\mathbbm{k})^\circ\setminus\mathbbm{i})\cup((\mathbbm{h}\cup\mathbbm{j})\cap
(\mathbbm{g}\cap\mathbbm{i}\cap\mathbbm{k})^\circ)=(\mathbbm{g}\triangle\mathbbm{k})
\cup((\mathbbm{g}\cap\mathbbm{k})^\circ\setminus\mathbbm{i})\cup(\mathbbm{g}\cap\mathbbm{i}\cap\mathbbm{j}\cap\mathbbm{k})^\circ$ and
$(\mathbbm{g}\triangle\mathbbm{k})
\cup((\mathbbm{g}\cap\mathbbm{k})^\circ\setminus\mathbbm{i})\cup((\mathbbm{h}\cup\mathbbm{l})\cap
(\mathbbm{g}\cap\mathbbm{i}\cap\mathbbm{k})^\circ)=(\mathbbm{g}\triangle\mathbbm{k})
\cup((\mathbbm{g}\cap\mathbbm{k})^\circ\setminus\mathbbm{i})\cup(\mathbbm{g}\cap\mathbbm{i}\cap\mathbbm{k}\cap\mathbbm{l})^\circ$ as
$(\mathbbm{h}\cap\mathbbm{i}\cap\mathbbm{k})^\circ\subseteq\mathbbm{j}\subseteq\mathbbm{l}$. Notice that $(\mathbbm{g}\cap\mathbbm{i}\cap\mathbbm{k}\cap\mathbbm{l})^\circ\setminus\mathbbm{j}=(\mathbbm{i}\cap\mathbbm{k}\cap\mathbbm{l})^\circ\setminus\mathbbm{j}$ as $\mathbf{g}\triangle\mathbf{i}\preceq\mathbf{h}$ and $(\mathbbm{h}\cap\mathbbm{i}\cap\mathbbm{k})^\circ\subseteq\mathbbm{j}$. The desired lemma follows from applying Lemma \ref{L;Lemma7.6} again.
\end{proof}
\begin{lem}\label{L;Lemma7.8}
Assume that $\mathbf{g}, \mathbf{h}, \mathbf{i}, \mathbf{j}, \mathbf{k}, \mathbf{l}, \mathbf{m}\!\!\in\!\!\mathbb{E}$, $\mathbf{g}\triangle\mathbf{i}\!\preceq\!\mathbf{h}$, $\mathbf{i}\triangle\mathbf{k}\!\!\preceq\!\!\mathbf{j}$, $p\!\nmid\! k_\mathbf{h}k_\mathbf{j}$, $q\!\!\in\!\![0, n_{\mathbf{i}\mathbf{k},\mathbf{j}}]$, $\mathbf{l}, \mathbf{m}\!\!\in\!\!\mathbb{U}_{\mathbf{i}\mathbf{k}, \mathbf{j}, q}$,
$(\mathbbm{h}\cap\mathbbm{i}\cap\mathbbm{k})^\circ\!\!\subseteq\!\!\mathbbm{j}$. Then $\mathbf{l}\!=\!\mathbf{m}$ if and only if $[\mathbf{g},\mathbf{h},\mathbf{i}, \mathbf{l}, \mathbf{k}]\!\!=\!\![\mathbf{g},\mathbf{h},\mathbf{i}, \mathbf{m}, \mathbf{k}]$.
Moreover, the map that sends $\mathbf{r}$ to $[\mathbf{g},\mathbf{h},\mathbf{i}, \mathbf{r}, \!\mathbf{k}]$ is injective from
$\mathbb{U}_{\mathbf{i}\mathbf{k}, \mathbf{j}, q}$ to $\mathbb{U}_{\mathbf{g}\mathbf{k}, [\mathbf{g},\mathbf{h},\mathbf{i}, \mathbf{j}, \mathbf{k}], q}$.
\end{lem}
\begin{proof}
By Lemma \ref{L;Lemma7.7}, it is enough to check that $[\mathbf{g},\mathbf{h},\mathbf{i}, \mathbf{l}, \mathbf{k}]=[\mathbf{g},\mathbf{h},\mathbf{i}, \mathbf{m}, \mathbf{k}]$ only if $\mathbf{l}=\mathbf{m}$. As $\mathbf{l}, \mathbf{m}\in\mathbb{U}_{\mathbf{i}\mathbf{k},\mathbf{j}, q}$, notice that $\mathbbm{l}=(\mathbbm{i}\triangle\mathbbm{k})\cup
(\mathbbm{i}\cap\mathbbm{k}\cap\mathbbm{l})=(\mathbbm{i}\triangle\mathbbm{k})\cup
(\mathbbm{i}\cap\mathbbm{k}\cap\mathbbm{l})^\circ$, $\mathbbm{m}=(\mathbbm{i}\triangle\mathbbm{k})\cup
(\mathbbm{i}\cap\mathbbm{k}\cap\mathbbm{m})=(\mathbbm{i}\triangle\mathbbm{k})\cup
(\mathbbm{i}\cap\mathbbm{k}\cap\mathbbm{m})^\circ$, and $(\mathbbm{h}\cap\mathbbm{i}\cap\mathbbm{k})^\circ\subseteq\mathbbm{j}\subseteq\mathbbm{l}\cap\mathbbm{m}$.
So the equation $[\mathbf{g},\mathbf{h},\mathbf{i}, \mathbf{l}, \mathbf{k}]=[\mathbf{g},\mathbf{h},\mathbf{i}, \mathbf{m}, \mathbf{k}]$
shows that $(\mathbbm{g}\cap\mathbbm{i}\cap\mathbbm{k}\cap\mathbbm{l})^\circ=(\mathbbm{g}\cap\mathbbm{i}\cap\mathbbm{k}\cap\mathbbm{m})^\circ$. Moreover, notice that $(\mathbbm{i}\cap\mathbbm{k}\cap\mathbbm{l})^\circ\setminus\mathbbm{g}=
(\mathbbm{i}\cap\mathbbm{k})^\circ\setminus\mathbbm{g}=(\mathbbm{i}\cap\mathbbm{k}\cap\mathbbm{m})^\circ\setminus\mathbbm{g}$
as $\mathbf{g}\triangle\mathbf{i}\preceq\mathbf{h}$. So $(\mathbbm{i}\cap\mathbbm{k}\cap\mathbbm{l})^\circ=(\mathbbm{i}\cap\mathbbm{k}\cap\mathbbm{m})^\circ$ and $\mathbf{l}=\mathbf{m}$ by Lemma \ref{L;Lemma2.3}. The desired lemma follows.
\end{proof}
\begin{lem}\label{L;Lemma7.9}
Assume that $\mathbf{g}, \mathbf{h}, \mathbf{i}, \mathbf{j}, \mathbf{k}\in\mathbb{E}$, $\mathbf{g}\triangle\mathbf{i}\preceq\mathbf{h}$, $\mathbf{j}\in\mathbb{P}_{\mathbf{i}, \mathbf{k}}$, $p\nmid k_\mathbf{h}k_\mathbf{j}$,
$\ell\in[0, n_{\mathbf{i}\mathbf{k},\mathbf{j}}]$, $(\mathbbm{h}\cap\mathbbm{i}\cap\mathbbm{k})^\circ\subseteq\mathbbm{j}$. Then $|\mathbb{U}_{\mathbf{i}\mathbf{k}, \mathbf{j},\ell}|=|\mathbb{U}_{\mathbf{g}\mathbf{k}, [\mathbf{g},\mathbf{h},\mathbf{i}, \mathbf{j}, \mathbf{k}], \ell}|>0$. Moreover, the map that sends $\mathbf{m}$ to $[\mathbf{g},\mathbf{h},\mathbf{i}, \mathbf{m}, \mathbf{k}]$ is bijective from $\mathbb{U}_{\mathbf{i}\mathbf{k}, \mathbf{j},\ell}$ to $\mathbb{U}_{\mathbf{g}\mathbf{k}, [\mathbf{g},\mathbf{h},\mathbf{i}, \mathbf{j}, \mathbf{k}], \ell}$. In particular, for any $\mathbf{q}\in\mathbb{U}_{\mathbf{g}\mathbf{k}, [\mathbf{g},\mathbf{h},\mathbf{i}, \mathbf{j}, \mathbf{k}], \ell}$, there is some $\mathbf{r}\in\mathbb{U}_{\mathbf{i}\mathbf{k}, \mathbf{j},\ell}$ that satisfies the equation $\mathbf{q}=[\mathbf{g},\mathbf{h},\mathbf{i}, \mathbf{r}, \mathbf{k}]$.
\end{lem}
\begin{proof}
Recall that $k_\mathbf{0}=1$ and $k_\mathbf{g}=\prod_{s\in\mathbbm{g}}(|\mathbb{U}_s|-1)$ if $\mathbf{g}\in\mathbb{E}\setminus\{\mathbf{0}\}$. As $\mathbf{j}\in\mathbb{P}_{\mathbf{i}, \mathbf{k}}$, it is obvious to see that
$\mathbf{i}\triangle\mathbf{k}\preceq\mathbf{j}\preceq\mathbf{i}\mathbf{k}$. This implies that $p\nmid k_{[\mathbf{g},\mathbf{h},\mathbf{i}, \mathbf{j}, \mathbf{k}]}$ by Lemma \ref{L;Lemma7.6}. Notice that $[\mathbf{g},\mathbf{h},\mathbf{i}, \mathbf{j}, \mathbf{k}]\!\preceq\!\mathbf{g}\mathbf{k}$ by Lemma \ref{L;Lemma3.16}.
Hence Lemmas \ref{L;Lemma2.3} and \ref{L;Lemma7.6} imply that
$$|\mathbb{U}_{\mathbf{i}\mathbf{k}, \mathbf{j},\ell}|={n_{\mathbf{i}\mathbf{k}, \mathbf{j}}\choose \ell}={n_{\mathbf{g}\mathbf{k}, [\mathbf{g},\mathbf{h},\mathbf{i}, \mathbf{j}, \mathbf{k}]}\choose \ell}=|\mathbb{U}_{\mathbf{g}\mathbf{k}, [\mathbf{g},\mathbf{h},\mathbf{i}, \mathbf{j}, \mathbf{k}], \ell}|>0.$$
The desired lemma follows from the above displayed inequality and Lemma \ref{L;Lemma7.8}.
\end{proof}
\begin{lem}\label{L;Lemma7.10}
Assume that $\mathbf{g}, \mathbf{h}, \mathbf{i}, \mathbf{j}, \mathbf{k}, \mathbf{l}\in\mathbb{E}$,
$\mathbf{i}\setminus\mathbf{g}\preceq\mathbf{h}$, $\mathbf{i}\triangle\mathbf{k}\preceq\mathbf{j}$,
$\mathbf{l}\in\mathbb{P}_{\mathbf{i}, \mathbf{k}}$. Assume that $(\mathbbm{h}\cap\mathbbm{i}\cap\mathbbm{k})^\circ\subseteq\mathbbm{j}\subseteq\mathbbm{l}$. Then $k_{\mathbf{h}\cap\mathbf{i}}=k_{\mathbf{h}\cap\mathbf{i}\cap\mathbf{l}}$. Moreover, $k_{\mathbf{k}\cap\mathbf{l}}=k_{\mathbf{k}\cap[\mathbf{g},\mathbf{h},\mathbf{i},\mathbf{l}, \mathbf{k}]}$.
\end{lem}
\begin{proof}
Recall that $k_\mathbf{0}=1$ and $k_\mathbf{g}=\prod_{m\in\mathbbm{g}}(|\mathbb{U}_m|-1)$ if $\mathbf{g}\in\mathbb{E}\setminus\{\mathbf{0}\}$. As $\mathbf{i}\triangle\mathbf{k}\preceq\mathbf{j}$ and
$(\mathbbm{h}\cap\mathbbm{i}\cap\mathbbm{k})^\circ\subseteq\mathbbm{j}\subseteq\mathbbm{l}$, notice that $(\mathbbm{h}\cap\mathbbm{i})^\circ=((\mathbbm{h}\cap\mathbbm{i})^\circ\setminus\mathbbm{k})\cup(\mathbbm{h}\cap\mathbbm{i}\cap\mathbbm{k})^\circ
\subseteq\mathbbm{j}\subseteq\mathbbm{l}$. Hence $(\mathbbm{h}\cap\mathbbm{i})^\circ=(\mathbbm{h}\cap\mathbbm{i}\cap\mathbbm{l})^\circ$ and $k_{\mathbf{h}\cap\mathbf{i}}=k_{\mathbf{h}\cap\mathbf{i}\cap\mathbf{l}}$ by Lemma \ref{L;Lemma2.3}. As $\mathbf{l}\in\mathbb{P}_{\mathbf{i}, \mathbf{k}}$,
notice that $\mathbbm{l}=(\mathbbm{i}\triangle\mathbbm{k})\cup(\mathbbm{i}\cap\mathbbm{k}\cap\mathbbm{l})=
(\mathbbm{i}\triangle\mathbbm{k})\cup(\mathbbm{i}\cap\mathbbm{k}\cap\mathbbm{l})^\circ$. So $\mathbbm{k}^\circ\setminus(\mathbbm{g}\cup\mathbbm{l})=(\mathbbm{i}\cap\mathbbm{k})^\circ\setminus
(\mathbbm{g}\cup(\mathbbm{i}\cap\mathbbm{k}\cap\mathbbm{l})^\circ)$. As $\mathbf{i}\setminus\mathbf{g}\preceq\mathbf{h}$ and
$(\mathbbm{h}\cap\mathbbm{i}\cap\mathbbm{k})^\circ\subseteq\mathbbm{j}\subseteq\mathbbm{l}$, $\mathbbm{k}^\circ\setminus(\mathbbm{g}\cup\mathbbm{l})=\varnothing$ and $\mathbbm{k}^\circ\setminus\mathbbm{g}=(\mathbbm{k}\cap\mathbbm{l})^\circ\setminus\mathbbm{g}$. As $\mathbf{i}\triangle\mathbf{k}\preceq\mathbf{j}\preceq\mathbf{l}$, notice that $(\mathbbm{g}\cap\mathbbm{k})^\circ\setminus\mathbbm{i}=(\mathbbm{g}\cap\mathbbm{j}\cap\mathbbm{k})^\circ\setminus\mathbbm{i}=
(\mathbbm{g}\cap\mathbbm{k}\cap\mathbbm{l})^\circ\setminus\mathbbm{i}$. This thus implies that $(\mathbbm{k}\cap\mathbbm{l})^\circ=(\mathbbm{k}^\circ\setminus\mathbbm{g})\cup((\mathbbm{g}\cap\mathbbm{k})^\circ\setminus\mathbbm{i})
\cup(\mathbbm{g}\cap\mathbbm{i}\cap\mathbbm{k}\cap\mathbbm{l})^\circ$. As $(\mathbbm{h}\cap\mathbbm{i}\cap\mathbbm{k})^\circ\subseteq\mathbbm{l}$, notice that $(\mathbbm{k}\cap\mathbbm{l})^\circ=(\mathbbm{k}^\circ\setminus\mathbbm{g})\cup((\mathbbm{g}\cap\mathbbm{k})^\circ\setminus\mathbbm{i})
\cup((\mathbbm{h}\cup\mathbbm{l})\cap(\mathbbm{g}\cap\mathbbm{i}\cap\mathbbm{k})^\circ)$. This thus yields the other equation $k_{\mathbf{k}\cap\mathbf{l}}=k_{\mathbf{k}\cap[\mathbf{g},\mathbf{h},\mathbf{i},\mathbf{l}, \mathbf{k}]}$ by Lemma \ref{L;Lemma2.3}. The desired lemma follows.
\end{proof}
\begin{lem}\label{L;Lemma7.11}
Assume that $\mathbf{g}, \mathbf{h}, \mathbf{i}, \mathbf{j}, \mathbf{k}\in\mathbb{E}$, $\mathbf{h}\in\mathbb{P}_{\mathbf{g}, \mathbf{i}}$, $\mathbf{j}\in\mathbb{P}_{\mathbf{i},\mathbf{k}}$, $p\nmid k_\mathbf{h}k_\mathbf{j}$. Assume that $(\mathbbm{h}\cap\mathbbm{i}\cap\mathbbm{k})^\circ\subseteq\mathbbm{j}$. Then $[\mathbf{g}, \mathbf{h}, \mathbf{i}, \mathbf{j}, \mathbf{k}]\in\mathbb{P}_{\mathbf{g}, \mathbf{k}}$, $p\nmid k_{[\mathbf{g}, \mathbf{h}, \mathbf{i}, \mathbf{j}, \mathbf{k}]}$, $B_{\mathbf{g}, \mathbf{h}, \mathbf{i}}D_{\mathbf{i},\mathbf{j},\mathbf{k}}=\overline{k_{\mathbf{h}\cap\mathbf{i}}}D_{\mathbf{g}, [\mathbf{g}, \mathbf{h}, \mathbf{i}, \mathbf{j}, \mathbf{k}],\mathbf{k}}$.
\end{lem}
\begin{proof}
Notice that $\mathbf{g}\triangle\mathbf{i}\preceq\mathbf{h}$ and $\mathbf{i}\triangle\mathbf{k}\preceq\mathbf{j}\preceq\mathbf{i}\mathbf{k}$ as $\mathbf{h}\in\mathbb{P}_{\mathbf{g}, \mathbf{i}}$ and $\mathbf{j}\in\mathbb{P}_{\mathbf{i},\mathbf{k}}$. Lemmas \ref{L;Lemma3.16} and \ref{L;Lemma7.6} thus imply that $[\mathbf{g}, \mathbf{h}, \mathbf{i}, \mathbf{j}, \mathbf{k}]\in\mathbb{P}_{\mathbf{g}, \mathbf{k}}$ and $p\nmid k_{[\mathbf{g}, \mathbf{h}, \mathbf{i}, \mathbf{j}, \mathbf{k}]}$. By Lemma \ref{L;Lemma5.7}, $D_{\mathbf{i}, \mathbf{j}, \mathbf{k}}\neq O$. By combining Equation \eqref{Eq;10}, Lemmas \ref{L;Lemma7.6}, \ref{L;Lemma7.9}, \ref{L;Lemma7.10}, and Theorem \ref{T;Theorem3.23},
\begin{align*}
B_{\mathbf{g}, \mathbf{h}, \mathbf{i}}D_{\mathbf{i}, \mathbf{j}, \mathbf{k}}&=\sum_{\ell=0}^{n_{\mathbf{i}\mathbf{k}, \mathbf{j}}}\sum_{\mathbf{m}\in\mathbb{U}_{\mathbf{i}\mathbf{k},\mathbf{j},\ell}}{(\overline{-1})^\ell\overline{k_{\mathbf{k}\cap\mathbf{m}}}}^{-1}B_{\mathbf{g}, \mathbf{h}, \mathbf{i}}B_{\mathbf{i}, \mathbf{m}, \mathbf{k}}\\
&=\sum_{\ell=0}^{n_{\mathbf{i}\mathbf{k}, \mathbf{j}}}\sum_{\mathbf{m}\in\mathbb{U}_{\mathbf{i}\mathbf{k},\mathbf{j},\ell}}(\overline{-1})^\ell
\overline{k_{\mathbf{h}\cap\mathbf{i}\cap\mathbf{m}}}(\overline{k_{\mathbf{k}\cap\mathbf{m}}})^{-1}B_{\mathbf{g}, [\mathbf{g}, \mathbf{h}, \mathbf{i}, \mathbf{m}, \mathbf{k}], \mathbf{k}}\\
&=\sum_{\ell=0}^{n_{\mathbf{g}\mathbf{k}, [\mathbf{g}, \mathbf{h}, \mathbf{i}, \mathbf{j}, \mathbf{k}]}}\sum_{\mathbf{m}\in\mathbb{U}_{\mathbf{g}\mathbf{k},[\mathbf{g}, \mathbf{h}, \mathbf{i}, \mathbf{j}, \mathbf{k}],\ell}}(\overline{-1})^\ell\overline{k_{\mathbf{h}\cap\mathbf{i}}}(\overline{k_{\mathbf{k}\cap\mathbf{m}}})^{-1}B_{\mathbf{g},\mathbf{m},\mathbf{k}}.
\end{align*}
So $B_{\mathbf{g}, \mathbf{h}, \mathbf{i}}D_{\mathbf{i}, \mathbf{j}, \mathbf{k}}=
\overline{k_{\mathbf{h}\cap\mathbf{i}}}D_{\mathbf{g},[\mathbf{g}, \mathbf{h}, \mathbf{i}, \mathbf{j}, \mathbf{k}], \mathbf{k}}$ by Equation \eqref{Eq;10}. The desired lemma follows.
\end{proof}
\begin{lem}\label{L;Lemma7.12}
Assume that $\mathbf{g}, \mathbf{h}, \mathbf{i}, \mathbf{j}, \mathbf{k}\in\mathbb{E}$, $\mathbf{i}\setminus\mathbf{g}\preceq\mathbf{h}$, $\mathbf{i}\setminus\mathbf{k}\preceq\mathbf{j}$. Then $(\mathbbm{g}\cap\mathbbm{i}\cap\mathbbm{j})^\circ\subseteq\mathbbm{h}$ and $(\mathbbm{h}\cap\mathbbm{i}\cap\mathbbm{k})^\circ\subseteq\mathbbm{j}$ if and only if $(\mathbbm{g}\cap\mathbbm{i})^\circ\setminus\mathbbm{h}=(\mathbbm{i}\cap\mathbbm{k})^\circ\setminus\mathbbm{j}=
(\mathbbm{g}\cap\mathbbm{i}\cap\mathbbm{k})^\circ\setminus(\mathbbm{h}\cup\mathbbm{j})$. In particular, $(\mathbbm{g}\cap\mathbbm{i}\cap\mathbbm{j})^\circ\subseteq\mathbbm{h}$ and $(\mathbbm{h}\cap\mathbbm{i}\cap\mathbbm{k})^\circ\subseteq\mathbbm{j}$ if and only if $(\mathbbm{g}\cap\mathbbm{i})^\circ\setminus\mathbbm{h}=(\mathbbm{i}\cap\mathbbm{k})^\circ\setminus\mathbbm{j}$.
\end{lem}
\begin{proof}
For one direction, assume that $(\mathbbm{g}\cap\mathbbm{i}\cap\mathbbm{j})^\circ\subseteq\mathbbm{h}$ and $(\mathbbm{h}\cap\mathbbm{i}\cap\mathbbm{k})^\circ\subseteq\mathbbm{j}$. Notice that $(\mathbbm{g}\cap\mathbbm{i})^\circ\setminus\mathbbm{h}=(\mathbbm{g}\cap\mathbbm{i})^\circ\setminus(\mathbbm{h}\cup\mathbbm{j})
=(\mathbbm{g}\cap\mathbbm{i}\cap\mathbbm{k})^\circ\setminus(\mathbbm{h}\cup\mathbbm{j})$ as $\mathbf{i}\setminus\mathbf{k}\preceq\mathbf{j}$. Notice that $(\mathbbm{i}\cap\mathbbm{k})^\circ\setminus\mathbbm{j}=
(\mathbbm{i}\cap\mathbbm{k})^\circ\setminus(\mathbbm{h}\cup\mathbbm{j})=(\mathbbm{g}\cap\mathbbm{i}\cap\mathbbm{k})^\circ\setminus(\mathbbm{h}\cup\mathbbm{j})$
as $\mathbf{i}\setminus\mathbf{g}\preceq\mathbf{h}$. For the other direction, assume that $(\mathbbm{g}\cap\mathbbm{i})^\circ\setminus\mathbbm{h}=(\mathbbm{i}\cap\mathbbm{k})^\circ\setminus\mathbbm{j}=
(\mathbbm{g}\cap\mathbbm{i}\cap\mathbbm{k})^\circ\setminus(\mathbbm{h}\cup\mathbbm{j})$. Then  $(\mathbbm{g}\cap\mathbbm{i}\cap\mathbbm{j})^\circ\subseteq\mathbbm{h}$ and $(\mathbbm{h}\cap\mathbbm{i}\cap\mathbbm{k})^\circ\subseteq\mathbbm{j}$. Notice that $(\mathbbm{g}\cap\mathbbm{i})^\circ\setminus\mathbbm{h}=(\mathbbm{i}\cap\mathbbm{k})^\circ\setminus\mathbbm{j}$ if and only if $(\mathbbm{g}\cap\mathbbm{i})^\circ\setminus\mathbbm{h}=(\mathbbm{i}\cap\mathbbm{k})^\circ\setminus\mathbbm{j}=
(\mathbbm{g}\cap\mathbbm{i}\cap\mathbbm{k})^\circ\setminus(\mathbbm{h}\cup\mathbbm{j})$. The desired lemma follows.
\end{proof}
\begin{lem}\label{L;Lemma7.13}
Assume that $\mathbf{g}, \mathbf{h}, \mathbf{i}, \mathbf{j}, \mathbf{k}\in\mathbb{E}$, $\mathbf{h}\in\mathbb{P}_{\mathbf{g}, \mathbf{i}}$, $\mathbf{j}\in\mathbb{P}_{\mathbf{i}, \mathbf{k}}$, $p\nmid k_\mathbf{h}k_\mathbf{j}$. Then the inequality $D_{\mathbf{g}, \mathbf{h}, \mathbf{i}}D_{\mathbf{i}, \mathbf{j}, \mathbf{k}}\neq O$ holds only if $(\mathbbm{g}\cap\mathbbm{i})^\circ \setminus\mathbbm{h}=(\mathbbm{i}\cap\mathbbm{k})^\circ\setminus\mathbbm{j}=
(\mathbbm{g}\cap\mathbbm{i}\cap\mathbbm{k})^\circ\setminus(\mathbbm{h}\cup\mathbbm{j})$.
\end{lem}
\begin{proof}
If $D_{\mathbf{g}, \mathbf{h}, \mathbf{i}}D_{\mathbf{i}, \mathbf{j}, \mathbf{k}}\neq O$, then
$(\mathbbm{g}\cap\mathbbm{i}\cap\mathbbm{j})^\circ\subseteq\mathbbm{h}$ and $(\mathbbm{h}\cap\mathbbm{i}\cap\mathbbm{k})^\circ\subseteq\mathbbm{j}$ by combining Equation \eqref{Eq;10}, Lemmas \ref{L;Lemma5.7}, \ref{L;Lemma5.12}. The desired lemma follows from Lemma \ref{L;Lemma7.12}.
\end{proof}
\begin{lem}\label{L;Lemma7.14}
Assume that $\mathbf{g}, \mathbf{h}, \mathbf{i}, \mathbf{j}, \mathbf{k},\mathbf{l}\in\mathbb{E}$, $\mathbf{h},\mathbf{l} \!\in\!\mathbb{P}_{\mathbf{g}, \mathbf{i}}$, $\mathbf{i}\setminus\mathbf{k}\preceq\mathbf{j}$, $\mathbf{h}\preceq\mathbf{l}$, $(\mathbbm{i}\cap\mathbbm{k}\cap\mathbbm{l})^\circ\!\!\subseteq\!\!\mathbbm{j}$. Assume that $(\mathbbm{g}\cap\mathbbm{i})^\circ \setminus\mathbbm{h}=(\mathbbm{i}\cap\mathbbm{k})^\circ\setminus\mathbbm{j}=(\mathbbm{g}\cap\mathbbm{i}\cap\mathbbm{k})^\circ\setminus(\mathbbm{h}\cup\mathbbm{j})$.
Then $\mathbf{h}=\mathbf{l}$. Moreover, if $\mathbf{j}\!\in\!\mathbb{P}_{\mathbf{i},\mathbf{k}}$ and $p\nmid k_\mathbf{h}k_\mathbf{j}$, then $[\mathbf{g}, \mathbf{h}, \mathbf{i},\mathbf{j}, \mathbf{k}]\in\mathbb{P}_{\mathbf{g},\mathbf{k}}$, $p\nmid k_{[\mathbf{g}, \mathbf{h},\mathbf{i},\mathbf{j},\mathbf{k}]}$, $D_{\mathbf{g}, \mathbf{h}, \mathbf{i}}D_{\mathbf{i}, \mathbf{j}, \mathbf{k}}=D_{\mathbf{g}, [\mathbf{g}, \mathbf{h}, \mathbf{i}, \mathbf{j}, \mathbf{k}], \mathbf{k}}$.
\end{lem}
\begin{proof}
As $(\mathbbm{g}\cap\mathbbm{i})^\circ \setminus\mathbbm{h}=(\mathbbm{i}\cap\mathbbm{k})^\circ\setminus\mathbbm{j}$ and $\mathbf{h}\preceq\mathbf{l}$, notice that $(\mathbbm{g}\cap\mathbbm{i}\cap\mathbbm{j})^\circ\subseteq\mathbbm{h}\subseteq\mathbbm{l}$.  So $(\mathbbm{g}\cap\mathbbm{i})^\circ\setminus\mathbbm{l}=(\mathbbm{g}\cap\mathbbm{i})^\circ\setminus(\mathbbm{j}\cup\mathbbm{l})=
(\mathbbm{g}\cap\mathbbm{i}\cap\mathbbm{k})^\circ\setminus(\mathbbm{j}\cup\mathbbm{l})\subseteq(\mathbbm{g}\cap\mathbbm{i}\cap\mathbbm{k})^\circ\setminus(\mathbbm{h}\cup\mathbbm{j})$
as $\mathbf{i}\setminus\mathbf{k}\preceq\mathbf{j}$ and $\mathbf{h}\preceq\mathbf{l}$. So $(\mathbbm{g}\cap\mathbbm{i})^\circ\setminus\mathbbm{l}\subseteq(\mathbbm{g}\cap\mathbbm{i}\cap\mathbbm{k})^\circ\setminus(\mathbbm{h}\cup\mathbbm{j})=(\mathbbm{g}\cap\mathbbm{i})^\circ\setminus\mathbbm{h}$. As $(\mathbbm{i}\cap\mathbbm{k}\cap\mathbbm{l})^\circ\!\subseteq\!\mathbbm{j}$, notice that
$(\mathbbm{g}\cap\mathbbm{i})^\circ\setminus\mathbbm{h}=(\mathbbm{g}\cap\mathbbm{i}\cap\mathbbm{k})^\circ\setminus(\mathbbm{h}\cup\mathbbm{j})=
(\mathbbm{g}\cap\mathbbm{i}\cap\mathbbm{k})^\circ\setminus(\mathbbm{h}\cup\mathbbm{j}\cup\mathbbm{l})\subseteq(\mathbbm{g}\cap\mathbbm{i})^\circ\setminus\mathbbm{l}$.
So $(\mathbbm{g}\cap\mathbbm{h}\cap\mathbbm{i})^\circ=(\mathbbm{g}\cap\mathbbm{i}\cap\mathbbm{l})^\circ$ and  $\mathbbm{h}=(\mathbbm{g}\triangle\mathbbm{i})\cup(\mathbbm{g}\cap\mathbbm{h}\cap\mathbbm{i})^\circ=
(\mathbbm{g}\triangle\mathbbm{i})\cup(\mathbbm{g}\cap\mathbbm{i}\cap\mathbbm{l})^\circ=\mathbbm{l}$ as $\mathbf{h},\mathbf{l} \in\mathbb{P}_{\mathbf{g}, \mathbf{i}}$. The first statement thus follows from Lemma \ref{L;Lemma2.3}. The combination of
Lemmas \ref{L;Lemma3.16}, \ref{L;Lemma7.12}, \ref{L;Lemma7.6} shows that $[\mathbf{g}, \mathbf{h}, \mathbf{i}, \mathbf{j}, \mathbf{k}]\in\mathbb{P}_{\mathbf{g}, \mathbf{k}}$ and $p\nmid k_{[\mathbf{g}, \mathbf{h}, \mathbf{i}, \mathbf{j}, \mathbf{k}]}$.
The desired lemma follows from combining Equation \eqref{Eq;10}, the first statement, Lemma \ref{L;Lemma7.11}.
\end{proof}
We are now ready to list the second main result of this section and two corollaries.
\begin{thm}\label{T;Theorem7.15}
Assume that $\mathbf{g}, \mathbf{h}, \mathbf{i}, \mathbf{j}, \mathbf{k}, \mathbf{l},\mathbf{m}, \mathbf{q}, \mathbf{r}\in\mathbb{E}$,
$\mathbf{h}\in\mathbb{P}_{\mathbf{g}, \mathbf{i}}$, $\mathbf{k}\in\mathbb{P}_{\mathbf{j}, \mathbf{l}}$, $\mathbf{q}\in\mathbb{P}_{\mathbf{m}, \mathbf{r}}$, $p\nmid k_\mathbf{h}k_\mathbf{k}k_\mathbf{q}$. Then $[\mathbf{g}, \mathbf{h}, \mathbf{i}, \mathbf{k}, \mathbf{l}]\in\mathbb{P}_{\mathbf{g}, \mathbf{l}}$ and $p\nmid k_{[\mathbf{g}, \mathbf{h}, \mathbf{i}, \mathbf{k}, \mathbf{l}]}$ if $\mathbf{i}=\mathbf{j}$ and $(\mathbbm{h}\cap\mathbbm{i}\cap\mathbbm{l})^\circ\subseteq\mathbbm{k}$.
Moreover, $c_{\mathbf{m},\mathbf{r}}^\mathbf{q}((D_{\mathbf{g}, \mathbf{h}, \mathbf{i}}\!+\!\mathrm{Rad}(\mathbb{T}))(D_{\mathbf{j}, \mathbf{k}, \mathbf{l}}\!+\!\mathrm{Rad}(\mathbb{T})))\!\!=\!\!\delta_{\mathbf{i}, \mathbf{j}}\delta_{\mathbf{m}, \mathbf{g}}\delta_{\mathbf{q}, \mathbf{[\mathbf{g}, \mathbf{h}, \mathbf{i},\mathbf{k},\mathbf{l}]}}\delta_{\mathbf{r},\mathbf{l}}\delta_{(\mathbbm{g}\cap\mathbbm{i})^\circ\setminus\mathbbm{h},(\mathbbm{i}\cap\mathbbm{l})^\circ\setminus\mathbbm{k}}.$
\end{thm}
\begin{proof}
If $\mathbf{i}=\mathbf{j}$ and $(\mathbbm{h}\cap\mathbbm{i}\cap\mathbbm{l})^\circ\subseteq\mathbbm{k}$, notice that $[\mathbf{g}, \mathbf{h}, \mathbf{i}, \mathbf{k}, \mathbf{l}]\in\mathbb{P}_{\mathbf{g}, \mathbf{l}}$ and $p\nmid k_{[\mathbf{g}, \mathbf{h}, \mathbf{i}, \mathbf{k}, \mathbf{l}]}$ by Lemmas \ref{L;Lemma3.16} and \ref{L;Lemma7.6}. By combining Equations \eqref{Eq;10}, \eqref{Eq;4}, \eqref{Eq;2}, and Theorem \ref{T;Theorem3.13}, there is no loss to assume that $\mathbf{i}=\mathbf{j}$, $\mathbf{m}=\mathbf{g}$, and $\mathbf{r}=\mathbf{l}$. By Lemma \ref{L;Lemma7.13}, there is also no loss to add an additional assumption that $(\mathbbm{g}\cap\mathbbm{i})^\circ\setminus\mathbbm{h}=(\mathbbm{i}\cap\mathbbm{l})^\circ\setminus\mathbbm{k}$.
By Lemma \ref{L;Lemma7.12}, the equation
$(\mathbbm{g}\cap\mathbbm{i})^\circ\setminus\mathbbm{h}=(\mathbbm{i}\cap\mathbbm{l})^\circ\setminus\mathbbm{k}$ shows that
$(\mathbbm{h}\cap\mathbbm{i}\cap\mathbbm{l})^\circ\subseteq\mathbbm{k}$. So
$[\mathbf{g}, \mathbf{h}, \mathbf{i}, \mathbf{k}, \mathbf{l}]\in\mathbb{P}_{\mathbf{g}, \mathbf{l}}$,
$p\nmid k_{[\mathbf{g}, \mathbf{h}, \mathbf{i}, \mathbf{k}, \mathbf{l}]}$,
$(D_{\mathbf{g}, \mathbf{h}, \mathbf{i}}+\mathrm{Rad}(\mathbb{T}))(D_{\mathbf{i}, \mathbf{k}, \mathbf{l}}+\mathrm{Rad}(\mathbb{T}))=D_{\mathbf{g}, [\mathbf{g}, \mathbf{h}, \mathbf{i}, \mathbf{k}, \mathbf{l}], \mathbf{l}}+\mathrm{Rad}(\mathbb{T})$ by Lemma \ref{L;Lemma7.14}. The desired theorem follows from Theorem \ref{T;Theorem7.4} and
the above discussion.
\end{proof}
\begin{cor}\label{C;Corollary7.16}
The structure constants of $\mathbb{B}_4$ in $\mathbb{T}/\mathrm{Rad}(\mathbb{T})$ are contained in $\{\overline{0}, \overline{1}\}$.
\end{cor}
\begin{proof}
The desired corollary follows from a direct application of Theorem \ref{T;Theorem7.15}.
\end{proof}
\begin{cor}\label{C;Corollary7.17}
Assume that $\mathbf{g}, \mathbf{h}, \mathbf{i}, \mathbf{j}, \mathbf{k}, \mathbf{l}\!\in\!\mathbb{E}$, $\mathbf{h}\in\mathbb{P}_{\mathbf{g}, \mathbf{i}}$, $\mathbf{j}\!\in\!\mathbb{P}_{\mathbf{i},\mathbf{k}}$, $\mathbf{l}\!\in\!\mathbb{P}_{\mathbf{g}, \mathbf{k}}$, $p\!\nmid\! k_\mathbf{h}k_\mathbf{j}k_\mathbf{l}$. Then
the equation $(D_{\mathbf{g}, \mathbf{h}, \mathbf{i}}+\mathrm{Rad}(\mathbb{T}))(D_{\mathbf{i}, \mathbf{j}, \mathbf{k}}+\mathrm{Rad}(\mathbb{T}))\!=\!D_{\mathbf{g}, \mathbf{l}, \mathbf{k}}+\mathrm{Rad}(\mathbb{T})$ holds only if
$$(\mathbbm{g}\cap\mathbbm{i})^\circ \setminus\mathbbm{h}=(\mathbbm{i}\cap\mathbbm{k})^\circ\setminus\mathbbm{j}=(\mathbbm{g}\cap\mathbbm{k})^\circ\setminus\mathbbm{l}
=(\mathbbm{g}\cap\mathbbm{i}\cap\mathbbm{k})^\circ\setminus(\mathbbm{h}\cup\mathbbm{j}).$$
\end{cor}
\begin{proof}
The equation $(D_{\mathbf{g}, \mathbf{h}, \mathbf{i}}+\mathrm{Rad}(\mathbb{T}))(D_{\mathbf{i}, \mathbf{j}, \mathbf{k}}+\mathrm{Rad}(\mathbb{T}))\!=\!D_{\mathbf{g}, \mathbf{l}, \mathbf{k}}+\mathrm{Rad}(\mathbb{T})$ implies that
$D_{\mathbf{g}, \mathbf{h}, \mathbf{i}}D_{\mathbf{i}, \mathbf{j}, \mathbf{k}}\neq O$ and $\mathbf{l}=[\mathbf{g}, \mathbf{h}, \mathbf{i}, \mathbf{j}, \mathbf{k}]$ by Theorems \ref{T;Theorem7.4} and \ref{T;Theorem7.15}. The desired corollary follows from the above discussion and an application of Lemma \ref{L;Lemma7.13}.
\end{proof}
We accomplish this section's discussion by an example of Theorems \ref{T;Theorem7.4} and \ref{T;Theorem7.15}.
\begin{eg}\label{E;Example7.18}
Assume that $n=|\mathbb{U}_1|=2$ and $|\mathbb{U}_2|=3$. It is clear that $|\mathbb{E}|=4$. Assume that $\mathbf{g}=(0, 1)$ and $\mathbf{h}=(1,0)$. Hence $\mathbb{E}=\{\mathbf{0}, \mathbf{g}, \mathbf{h}, \mathbf{1}\}$. Assume that $p\neq2$. Then $\mathrm{Rad}(\mathbb{T})=\{O\}$ and $\mathbb{T}$ has an $\F$-basis containing exactly $D_{\mathbf{0}, \mathbf{0}, \mathbf{0}}$,
$D_{\mathbf{0}, \mathbf{g}, \mathbf{g}}$, $D_{\mathbf{0}, \mathbf{h}, \mathbf{h}}$, $D_{\mathbf{0}, \mathbf{1}, \mathbf{1}}$, $D_{\mathbf{g}, \mathbf{0}, \mathbf{g}}$, $D_{\mathbf{g}, \mathbf{g}, \mathbf{0}}$,
$D_{\mathbf{g}, \mathbf{g}, \mathbf{g}}$, $D_{\mathbf{g}, \mathbf{h},\mathbf{1}}$, $D_{\mathbf{g}, \mathbf{1},\mathbf{h}}$, $D_{\mathbf{g}, \mathbf{1},\mathbf{1}}$, $D_{\mathbf{h}, \mathbf{0},\mathbf{h}}$,
$D_{\mathbf{h}, \mathbf{g},\mathbf{1}}$, $D_{\mathbf{h}, \mathbf{h},\mathbf{0}}$, $D_{\mathbf{h}, \mathbf{1},\mathbf{g}}$, $D_{\mathbf{1}, \mathbf{0},\mathbf{1}}$, $D_{\mathbf{1}, \mathbf{g},\mathbf{h}}$,
$D_{\mathbf{1}, \mathbf{g},\mathbf{1}}$, $D_{\mathbf{1}, \mathbf{h},\mathbf{g}}$, $D_{\mathbf{1}, \mathbf{1},\mathbf{0}}$, $D_{\mathbf{1}, \mathbf{1},\mathbf{g}}$ by Theorems \ref{T;Jacobson} and \ref{T;Theorem7.4}.
Assume that $p\!=\!2$. Theorem \ref{T;Jacobson} thus implies that $\mathrm{Rad}(\mathbb{T})\!\neq\!\{O\}$. Theorem \ref{T;Theorem7.4} also implies that $\mathbb{T}/\mathrm{Rad}(\mathbb{T})$ has an $\F$-basis containing exactly $D_{\mathbf{0}, \mathbf{0},\mathbf{0}}\!+\!\mathrm{Rad}(\mathbb{T})$, $D_{\mathbf{0}, \mathbf{h},\mathbf{h}}\!+\!\mathrm{Rad}(\mathbb{T})$, $D_{\mathbf{g}, \mathbf{0},\mathbf{g}}\!+\!\mathrm{Rad}(\mathbb{T})$, $D_{\mathbf{g}, \mathbf{h},\mathbf{1}}\!+\!\mathrm{Rad}(\mathbb{T})$, $D_{\mathbf{h}, \mathbf{0},\mathbf{h}}\!+\!\mathrm{Rad}(\mathbb{T})$, $D_{\mathbf{h}, \mathbf{h},\mathbf{0}}\!+\!\mathrm{Rad}(\mathbb{T})$, $D_{\mathbf{1}, \mathbf{0},\mathbf{1}}\!+\!\mathrm{Rad}(\mathbb{T})$, $D_{\mathbf{1}, \mathbf{h},\mathbf{g}}\!+\!\mathrm{Rad}(\mathbb{T})$. By Theorem \ref{T;Theorem7.15}, $(D_{\mathbf{g}, \mathbf{h}, \mathbf{1}}\!+\!\mathrm{Rad}(\mathbb{T}))(D_{\mathbf{1}, \mathbf{h}, \mathbf{g}}\!+\!\mathrm{Rad}(\mathbb{T}))\!\!=\!\!D_{\mathbf{g},\mathbf{0},\mathbf{g}}\!+\!\mathrm{Rad}(\mathbb{T})$.
\end{eg}
\section{Algebraic structure of $\mathbb{T}$: Wedderburn-Artin decomposition}
In this section, we present the Wedderburn-Artin decomposition of $\mathbb{T}$. This means that we decompose the semisimple $\F$-algebra $\mathbb{T}/\mathrm{Rad}(\mathbb{T})$ into a direct sum of some full matrix $\F$-algebras. As an application, we also determine all Terwilliger $\F$-algebras of factorial schemes that are the symmetric $\F$-algebras or the Frobenius $\F$-algebras. We first recall Notations \ref{N;Notation3.1}, \ref{N;Notation3.9}, \ref{N;Notation3.14}, \ref{N;Notation3.15}, \ref{N;Notation4.1}, \ref{N;Notation5.5}, \ref{N;Notation5.6} and start with a notation.
\begin{nota}\label{N;Notation8.1}
Let $\mathbb{D}=\{(\mathbf{a},\mathbf{b}, \mathbf{c}): (\mathbf{a}, \mathbf{b}, \mathbf{c})\in\mathbb{P}, p\nmid k_\mathbf{b}\}$. So $\mathbb{D}\neq\varnothing$ as $(\mathbf{0},\mathbf{0},\mathbf{0})\in\mathbb{D}$.
Theorem \ref{T;Theorem7.4} implies that $\mathbb{T}/\mathrm{Rad}(\mathbb{T})$ has an $\F$-basis $\{D_{\mathbf{a},\mathbf{b},\mathbf{c}}+\mathrm{Rad}(\mathbb{T}):(\mathbf{a}, \mathbf{b}, \mathbf{c})\in\mathbb{D}\}$.
If $(\mathbf{g}, \mathbf{h}, \mathbf{i}), (\mathbf{j}, \mathbf{k}, \mathbf{l})\in\mathbb{D}$, write $(\mathbf{g}, \mathbf{h}, \mathbf{i})\sim(\mathbf{j}, \mathbf{k}, \mathbf{l})$ if and only if $(\mathbbm{g}\cap\mathbbm{i})^\circ\setminus\mathbbm{h}=(\mathbbm{j}\cap\mathbbm{l})^\circ\setminus\mathbbm{k}$. This implies that $\sim$ is an equivalence relation on $\mathbb{D}$. Hence there is $n_\sim\in\mathbb{N}$ such that $\mathbb{D}_1,\mathbb{D}_2,\ldots, \mathbb{D}_{n_\sim}$ are exactly all pairwise distinct equivalence classes of $\mathbb{D}$ with respect to $\sim$. Assume that $m\!\in\![1, n_\sim]$. Let $\mathbb{D}(m)\!\!=\!\!\{\mathbf{a}\!: \exists\ \mathbf{b}\!\in\!\mathbb{E}, (\mathbf{a}, \mathbf{b}, \mathbf{a})\!\in\!\mathbb{D}_m\}$ and $\mathbb{I}(m)\!=\!\langle\{D_{\mathbf{a}, \mathbf{b}, \mathbf{c}}\!+\!\mathrm{Rad}(\mathbb{T})\!: (\mathbf{a},\mathbf{b},\mathbf{c})\!\in\!\mathbb{D}_m\}\rangle_{\mathbb{T}/\mathrm{Rad}(\mathbb{T})}$.
The $\F$-dimension of $\mathbb{I}(m)$ is $|\mathbb{D}_m|$.
\end{nota}
The following sequence of lemmas aims at studying the objects in Notation \ref{N;Notation8.1}.
\begin{lem}\label{L;Lemma8.2}
Assume that $g\in[1,n_\sim]$. Then $\mathbb{I}(g)$ is a two-sided ideal of $\mathbb{T}/\mathrm{Rad}(\mathbb{T})$. Moreover, $\mathbb{T}/\mathrm{Rad}(\mathbb{T})$ is a direct sum of the $\F$-linear subspaces $\mathbb{I}(1), \mathbb{I}(2),\ldots, \mathbb{I}(n_\sim)$.
\end{lem}
\begin{proof}
It is obvious to see that $\{\mathbb{D}_1, \mathbb{D}_2,\ldots, \mathbb{D}_{n_\sim}\}$ forms a partition of $\mathbb{D}$. The desired lemma follows from the combination of Theorems \ref{T;Theorem7.4}, \ref{T;Theorem7.15}, and Corollary \ref{C;Corollary7.17}.
\end{proof}
\begin{lem}\label{L;Lemma8.3}
Assume that $g\!\in\![1,n_\sim]$, $\mathbf{h}, \mathbf{i}, \mathbf{j}, \mathbf{k}\!\in\!\mathbb{E}$, $(\mathbf{h}, \mathbf{i},\mathbf{k}),(\mathbf{h}, \mathbf{j}, \mathbf{k})\!\in\!\mathbb{D}_g$. Then $\mathbf{i}\!=\!\mathbf{j}$.
\end{lem}
\begin{proof}
As $(\mathbf{h}, \mathbf{i},\mathbf{k}),(\mathbf{h}, \mathbf{j}, \mathbf{k})\in\mathbb{D}_g$, notice that   $(\mathbbm{h}\cap\mathbbm{k})^\circ\setminus\mathbbm{i}=(\mathbbm{h}\cap\mathbbm{k})^\circ\setminus\mathbbm{j}$ and $\mathbf{i}, \mathbf{j}\in\mathbb{P}_{\mathbf{h}, \mathbf{k}}$. This shows that $(\mathbbm{h}\cap\mathbbm{i}\cap\mathbbm{k})^\circ=(\mathbbm{h}\cap\mathbbm{j}\cap\mathbbm{k})^\circ$, $\mathbbm{i}=(\mathbbm{h}\triangle\mathbbm{k})\cup(\mathbbm{h}\cap\mathbbm{i}\cap\mathbbm{k})=
(\mathbbm{h}\triangle\mathbbm{k})\cup(\mathbbm{h}\cap\mathbbm{i}\cap\mathbbm{k})^\circ$, $\mathbbm{j}=(\mathbbm{h}\triangle\mathbbm{k})\cup(\mathbbm{h}\cap\mathbbm{j}\cap\mathbbm{k})=
(\mathbbm{h}\triangle\mathbbm{k})\cup(\mathbbm{h}\cap\mathbbm{j}\cap\mathbbm{k})^\circ$. These listed equations thus imply that $\mathbbm{i}=\mathbbm{j}$. The desired lemma follows from a direct application of Lemma \ref{L;Lemma2.3}.
\end{proof}
\begin{lem}\label{L;Lemma8.4}
Assume that $g\!\in\![1,n_\sim]$, $\mathbf{h}, \mathbf{i}, \mathbf{j}\in\mathbb{E}$, $(\mathbf{h}, \mathbf{i}, \mathbf{j})\in\mathbb{D}_g$. Then there are unique $\mathbf{k}, \mathbf{l}\in\mathbb{E}$ such that $\mathbb{D}_g$ contains
$(\mathbf{h}, \mathbf{k}, \mathbf{h})$ and $(\mathbf{j},\mathbf{l}, \mathbf{j})$. In particular, $\mathbf{h}, \mathbf{j}\in\mathbb{D}(g)$.
\end{lem}
\begin{proof}
By Lemma \ref{L;Lemma2.3}, there are unique $\mathbf{k}, \mathbf{l}\in\mathbb{E}$ such that $\mathbbm{k}=(\mathbbm{h}\cap\mathbbm{i})^\circ$ and $\mathbbm{l}=(\mathbbm{i}\cap\mathbbm{j})^\circ$. Then $\mathbf{k}\in\mathbb{P}_{\mathbf{h}, \mathbf{h}}$ and $\mathbf{l}\in\mathbb{P}_{\mathbf{j}, \mathbf{j}}$. As $(\mathbf{h}, \mathbf{i}, \mathbf{j})\in\mathbb{D}_g$, notice that $\mathbf{i}\in\mathbb{P}_{\mathbf{h}, \mathbf{j}}$, $p\nmid k_\mathbf{i}$, and $p\nmid k_\mathbf{k}k_\mathbf{l}$. Therefore $\mathbbm{i}=(\mathbbm{h}\triangle\mathbbm{j})\cup(\mathbbm{h}\cap\mathbbm{i}\cap\mathbbm{j})=
(\mathbbm{h}\triangle\mathbbm{j})\cup(\mathbbm{h}\cap\mathbbm{i}\cap\mathbbm{j})^\circ$ and $(\mathbf{h}, \mathbf{k}, \mathbf{h}), (\mathbf{j},\mathbf{l}, \mathbf{j})\in\mathbb{D}$. Hence $\mathbbm{h}^\circ\setminus\mathbbm{k}=\mathbbm{h}^\circ\setminus\mathbbm{i}=(\mathbbm{h}\cap\mathbbm{j})^\circ\setminus\mathbbm{i}
=\mathbbm{j}^\circ\setminus\mathbbm{i}=\mathbbm{j}^\circ\setminus\mathbbm{l}$. Therefore $(\mathbf{h}, \mathbf{k}, \mathbf{h}), (\mathbf{j},\mathbf{l}, \mathbf{j})\in\mathbb{D}_g$. The uniqueness of $\mathbf{k}$ and $\mathbf{l}$ follows from Lemma \ref{L;Lemma8.3}. The desired lemma follows.
\end{proof}
\begin{lem}\label{L;Lemma8.5}
Assume that $g\in [1, n_\sim]$, $\mathbf{h}, \mathbf{i}, \mathbf{j}, \mathbf{k}\in\mathbb{E}$, $(\mathbf{h}, \mathbf{i}, \mathbf{h}), (\mathbf{j}, \mathbf{k}, \mathbf{j})\in\mathbb{D}_g$. Then there is a unique $\mathbf{l}\in\mathbb{E}$ such that $(\mathbf{h},\mathbf{l},\mathbf{j})\in\mathbb{D}_g$.
\end{lem}
\begin{proof}
Recall that $k_\mathbf{0}=1$ and $k_\mathbf{m}=\prod_{q\in\mathbbm{m}}(|\mathbb{U}_q|-1)$ if $\mathbf{m}\in\mathbb{E}\setminus\{\mathbf{0}\}$. By Lemma \ref{L;Lemma2.3}, there is a unique $\mathbf{l}\in\mathbb{E}$ such that $\mathbbm{l}=(\mathbbm{h}\triangle\mathbbm{j})\cup(\mathbbm{h}\cap\mathbbm{j}\cap\mathbbm{k})^\circ$. This shows that $\mathbf{l}\in\mathbb{P}_{\mathbf{h}, \mathbf{j}}$. As $(\mathbf{h}, \mathbf{i}, \mathbf{h}), (\mathbf{j}, \mathbf{k}, \mathbf{j})\in\mathbb{D}_g$, notice that $\mathbbm{i}=(\mathbbm{h}\cap\mathbbm{i})=(\mathbbm{h}\cap\mathbbm{i})^\circ$, $\mathbbm{k}=(\mathbbm{j}\cap\mathbbm{k})=(\mathbbm{j}\cap\mathbbm{k})^\circ$, $\mathbbm{h}^\circ\setminus\mathbbm{i}=\mathbbm{j}^\circ\setminus\mathbbm{k}$, and $p\nmid k_\mathbf{i}k_\mathbf{k}$. Hence $\mathbbm{h}^\circ\setminus\mathbbm{j}=(\mathbbm{h}\cap\mathbbm{i})^\circ\setminus\mathbbm{j}$, $\mathbbm{j}^\circ\setminus\mathbbm{h}=(\mathbbm{j}\cap\mathbbm{k})^\circ\setminus\mathbbm{h}$, and $p\nmid k_\mathbf{l}$. So $(\mathbf{h}, \mathbf{l}, \mathbf{j})\in\mathbb{D}$. Then $(\mathbf{h}, \mathbf{l}, \mathbf{j})\in\mathbb{D}_g$ as $(\mathbbm{h}\cap\mathbbm{j})^\circ\setminus\mathbbm{l}=(\mathbbm{h}\cap\mathbbm{j})^\circ\setminus\mathbbm{k}=
\mathbbm{j}^\circ\setminus\mathbbm{k}=\mathbbm{h}^\circ\setminus\mathbbm{i}$. The uniqueness of $\mathbf{l}$ thus follows from Lemma \ref{L;Lemma8.3}. The desired lemma follows.
\end{proof}
The above lemma sequence motivates us to present the next lemma and a notation.
\begin{lem}\label{L;Lemma8.6}
Assume that $g\in[1, n_\sim]$. Then the map that sends $(\mathbf{h}, \mathbf{i}, \mathbf{j})$ to $(\mathbf{h}, \mathbf{j})$ is bijective from $\mathbb{D}_g$ to $\mathbb{D}(g)\times \mathbb{D}(g)$. In particular, the $\F$-dimension of $\mathbb{I}(g)$ is $|\mathbb{D}(g)|^2$.
\end{lem}
\begin{proof}
The desired lemma follows from the combination of Lemmas \ref{L;Lemma8.4}, \ref{L;Lemma8.3}, \ref{L;Lemma8.5}.
\end{proof}
\begin{nota}\label{N;Notation8.7}
Assume that $g\in[1, n_\sim]$. By Lemma \ref{L;Lemma8.6} and Theorem \ref{T;Theorem7.4}, there is a unique $D_{\mathbf{h},\mathbf{i}, \mathbf{j}}+\mathrm{Rad}(\mathbb{T})\in\{D_{\mathbf{a}, \mathbf{b}, \mathbf{c}}+\mathrm{Rad}(\mathbb{T}): (\mathbf{a}, \mathbf{b}, \mathbf{c})\in\mathbb{D}_g\}$ for any $\mathbf{h}, \mathbf{j}\in\mathbb{D}(g)$. This nonzero element in $\mathbb{T}/\mathrm{Rad}(\mathbb{T})$ is denoted by $D_{\mathbf{h},\mathbf{j}}(g)$ for any $\mathbf{h}, \mathbf{j}\in\mathbb{D}(g)$. For completeness, define $D_{\mathbf{h}, \mathbf{j}}(g)=O+\mathrm{Rad}(\mathbb{T})$ for any $\mathbf{h}, \mathbf{j}\in\mathbb{E}$ and $\{\mathbf{h}, \mathbf{j}\}\not\subseteq\mathbb{D}(g)$. By Lemma \ref{L;Lemma8.6} and Theorem \ref{T;Theorem7.4}, notice that $\mathbb{I}(g)$ has an $\F$-basis $\{D_{\mathbf{a}, \mathbf{b}}(g):\mathbf{a}, \mathbf{b}\in\mathbb{D}(g)\}$.
\end{nota}
The objects in Notation \ref{N;Notation8.7} allow us to present the following preliminary lemmas.
\begin{lem}\label{L;Lemma8.8}
$\mathbb{T}/\mathrm{Rad}(\mathbb{T})$ has an $\F$-basis $\{D_{\mathbf{b}, \mathbf{c}}(a)\!:\! a\!\in\![1, n_\sim], \mathbf{b}, \mathbf{c}\!\in\!\mathbb{D}(a)\}$. Moreover, if $g, h\!\in\![1, n_\sim]$,
$\mathbf{i}, \mathbf{j}, \mathbf{k},\mathbf{l}\!\in\!\mathbb{E}$, $\mathbf{i}, \mathbf{j}\!\in\!\mathbb{D}(g)$, $\mathbf{k}, \mathbf{l}\!\in\!\mathbb{D}(h)$, then $D_{\mathbf{i},\mathbf{j}}(g)D_{\mathbf{k},\mathbf{l}}(h)\!\!=\!\!\delta_{g, h}\delta_{\mathbf{j}, \mathbf{k}}D_{\mathbf{i},\mathbf{l}}(g)$.
\end{lem}
\begin{proof}
The first statement follows from combining Theorem \ref{T;Theorem7.4}, Lemmas \ref{L;Lemma8.2}, \ref{L;Lemma8.6}. By Lemma \ref{L;Lemma7.13}, there is no loss to assume that $g=h$. There are
unique $\mathbf{m}\in\mathbb{P}_{\mathbf{i}, \mathbf{j}}$, $\mathbf{q}\in\mathbb{P}_{\mathbf{k}, \mathbf{l}}$, $\mathbf{r}\in\mathbb{P}_{\mathbf{i},\mathbf{l}}$ such that $D_{\mathbf{i},\mathbf{j}}(g)=D_{\mathbf{i}, \mathbf{m}, \mathbf{j}}+\mathrm{Rad}(\mathbb{T})$, $D_{\mathbf{k},\mathbf{l}}(g)=D_{\mathbf{k}, \mathbf{q}, \mathbf{l}}+\mathrm{Rad}(\mathbb{T})$, and $D_{\mathbf{i},\mathbf{l}}(g)=D_{\mathbf{i}, \mathbf{r}, \mathbf{l}}+\mathrm{Rad}(\mathbb{T})$. Notice that all conditions in Theorem \ref{T;Theorem7.15} have already been satisfied. The desired lemma follows from an application of Theorem \ref{T;Theorem7.15}.
\end{proof}
\begin{lem}\label{L;Lemma8.9}
Assume that $g\in[1,n_\sim]$, $\mathbf{h}, \mathbf{i}\in\mathbb{E}$, $\mathbf{h}, \mathbf{i}\in\mathbb{D}(g)$. Then, via the left multiplication, $\langle\{D_{\mathbf{a}, \mathbf{h}}(g): \mathbf{a}\in\mathbb{D}(g)\}\rangle_{\mathbb{T}/\mathrm{Rad}(\mathbb{T})}$ is an irreducible $\mathbb{T}$-module. Moreover,
$\langle\{D_{\mathbf{a}, \mathbf{h}}(g)\!:\! \mathbf{a}\!\in\!\mathbb{D}(g)\}\rangle_{\mathbb{T}/\mathrm{Rad}(\mathbb{T})}\!\cong\!\langle\{D_{\mathbf{a}, \mathbf{i}}(g)\!:\! \mathbf{a}\!\in\!\mathbb{D}(g)\}\rangle_{\mathbb{T}/\mathrm{Rad}(\mathbb{T})}$ as irreducible $\mathbb{T}$-modules.
\end{lem}
\begin{proof}
The first statement follows from combining Lemmas \ref{L;Lemma7.3}, \ref{L;Lemma8.8}, Theorem \ref{T;Jacobson}. By the first statement and Lemma \ref{L;Lemma8.8}, the $\F$-linear map that sends $D_{\mathbf{j}, \mathbf{h}}(g)$ to $D_{\mathbf{j}, \mathbf{i}}(g)$ is also a $\mathbb{T}$-module isomorphism from the $\mathbb{T}$-module $\langle\{D_{\mathbf{a}, \mathbf{h}}(g): \mathbf{a}\in\mathbb{D}(g)\}\rangle_{\mathbb{T}/\mathrm{Rad}(\mathbb{T})}$ to the $\mathbb{T}$-module $\langle\{D_{\mathbf{a}, \mathbf{i}}(g): \mathbf{a}\in\mathbb{D}(g)\}\rangle_{\mathbb{T}/\mathrm{Rad}(\mathbb{T})}$. Hence the desired lemma follows.
\end{proof}
\begin{lem}\label{L;Lemma8.10}
Assume that $g\in[1, n_\sim]$. Then $\mathbb{I}(g)\cong \mathrm{M}_{|\mathbb{D}(g)|}(\F)$ as $\F$-algebras.
\end{lem}
\begin{proof}
It suffices to check that $\mathbb{I}(g)\cong\mathrm{M}_{\mathbb{D}(g)}(\F)$ as $\F$-algebras. For any $\mathbf{h}, \mathbf{i}\in\mathbb{D}(g)$, use $E_{\mathbf{h}, \mathbf{i}}$ to denote the $\{\overline{0}, \overline{1}\}$-matrix in $\mathrm{M}_{\mathbb{D}(g)}(\F)$ whose unique nonzero entry is the $(\mathbf{h}, \mathbf{i})$-entry. In particular, $E_{\mathbf{h},\mathbf{i}}E_{\mathbf{j},\mathbf{k}}=\delta_{\mathbf{i},\mathbf{j}}E_{\mathbf{h}, \mathbf{k}}$ for any $\mathbf{h}, \mathbf{i}, \mathbf{j}, \mathbf{k}\in\mathbb{D}(g)$. According to Lemma \ref{L;Lemma8.8}, the $\F$-linear map that sends $D_{\mathbf{h},\mathbf{i}}(g)$ to $E_{\mathbf{h}, \mathbf{i}}$ is an algebra isomorphism from the $\F$-subalgebra $\mathbb{I}(g)$ of $\mathbb{T}/\mathrm{Rad}(\mathbb{T})$ to $\mathrm{M}_{\mathbb{D}(g)}(\F)$. The desired lemma follows.
\end{proof}
We are now ready to get the first main result of this section and some corollaries.
\begin{thm}\label{T;Decomposition}
$$\mathbb{T}/\mathrm{Rad}(\mathbb{T})\cong\bigoplus_{g=1}^{n_\sim}\mathrm{M}_{|\mathbb{D}(g)|}(\F)\ \text{as $\F$-algebras.}$$
\end{thm}
\begin{proof}
The desired theorem follows from an application of Lemmas \ref{L;Lemma8.2} and \ref{L;Lemma8.10}.
\end{proof}
\begin{cor}\label{C;Corollary8.12}
The following are equivalent: $\mathbb{S}$ is a $p'$-valenced scheme; $\mathbb{T}$ is a semisimple $\F$-algebra; $\mathbb{T}$ must satisfy the formula $\mathbb{T}\cong\bigoplus_{g=1}^{n_\sim}\mathrm{M}_{|\mathbb{D}(g)|}(\F)$ as $\mathbb{F}$-algebras.
\end{cor}
\begin{proof}
The desired corollary directly follows from using Theorems \ref{T;Semisimplicity} and \ref{T;Decomposition}.
\end{proof}
\begin{cor}\label{C;Corollary8.13}
Assume that $g, h\!\in\![1, n_\sim]$ and $\mathbbm{Irr}(g)$ denotes the unique irreducible $\mathbb{T}$-module up to isomorphism whose definition is in Lemma \ref{L;Lemma8.9}. Then $g=h$ if and only if $\mathbbm{Irr}(g)\!\cong\!\mathbbm{Irr}(h)$ as $\mathbb{T}$-modules. In particular, $\{\mathbbm{Irr}(a)\!:\! a\!\in\![1,n_\sim]\}$ is a complete set of distinct representatives of all isomorphic classes of the irreducible $\mathbb{T}$-modules.
\end{cor}
\begin{proof}
The first statement follows from Lemmas \ref{L;Lemma8.9} and \ref{L;Lemma8.8}. The desired corollary follows from the combination of the first statement, Theorem \ref{T;Decomposition}, Lemma \ref{L;Lemma2.8}.
\end{proof}
\begin{cor}\label{C;Corollary8.14}
A $\mathbb{T}$-module is an irreducible $\mathbb{T}$-module if and only if this given $\mathbb{T}$-module is also an absolutely irreducible $\mathbb{T}$-module. In particular, the set of all irreducible $\mathbb{T}$-modules precisely equals the set of all
absolutely irreducible $\mathbb{T}$-modules.
\end{cor}
\begin{proof}
The desired corollary follows from using Theorem \ref{T;Decomposition} and Lemma \ref{L;Lemma2.9}.
\end{proof}
\begin{cor}\label{C;Corollary8.15}
An irreducible $\mathbb{T}$-module is also a self-contragredient $\mathbb{T}$-module with respect to $\alpha_T$. In particular, the set of all irreducible $\mathbb{T}$-modules is a nonempty finite set contained in the set of all self-contragredient $\mathbb{T}$-modules with respect to $\alpha_T$.
\end{cor}
\begin{proof}
Recall that $k_\mathbf{0}=1$ and $k_\mathbf{g}=\prod_{h\in\mathbbm{g}}(|\mathbb{U}_h|-1)$ if $\mathbf{g}\in\mathbb{E}\setminus\{\mathbf{0}\}$. Let $\mathbf{i}, \mathbf{j}, \mathbf{k}, \mathbf{l},\mathbf{m}\in\mathbb{E}$ and $\mathbf{j}\in\mathbb{P}_{\mathbf{i}, \mathbf{k}}$. This thus implies that $\mathbbm{j}=(\mathbbm{i}\triangle\mathbbm{k})\cup(\mathbbm{i}\cap\mathbbm{j}\cap\mathbbm{k})=(\mathbbm{i}\triangle\mathbbm{k})\cup(\mathbbm{i}\cap\mathbbm{j}\cap\mathbbm{k})^\circ$.
This thus implies that $p\nmid k_{\mathbf{i}\setminus\mathbf{k}}k_{\mathbf{k}\setminus\mathbf{i}}$ if $p\nmid k_\mathbf{j}$.
Assume that $q\in [1, n_\sim]$ and $\mathbf{l}\in\mathbb{D}(q)$. This thus implies that $p\nmid k_{\mathbf{l}\setminus\mathbf{m}}k_{\mathbf{m}\setminus\mathbf{l}}$ if $\mathbf{m}\in\mathbb{D}(q)$.
As $\mathbbm{i}\setminus\mathbbm{l}=(\mathbbm{i}\setminus(\mathbbm{l}\cup\mathbbm{m}))\cup((\mathbbm{i}\cap\mathbbm{m})\setminus\mathbbm{l})$,
$\mathbbm{l}\setminus\mathbbm{m}=(\mathbbm{l}\setminus(\mathbbm{i}\cup\mathbbm{m}))\cup((\mathbbm{i}\cap\mathbbm{l})\setminus\mathbbm{m})$, $\mathbbm{m}\setminus\mathbbm{i}\!=\!(\mathbbm{m}\setminus(\mathbbm{i}\cup\mathbbm{l}))\!\cup\!((\mathbbm{l}\cap\mathbbm{m})\setminus\mathbbm{i})$, and $\mathbbm{i}\setminus(\mathbbm{l}\cup\mathbbm{m})$,
$(\mathbbm{i}\cap\mathbbm{m})\setminus\mathbbm{l}$, $\mathbbm{l}\setminus(\mathbbm{i}\cup\mathbbm{m})$, $(\mathbbm{i}\cap\mathbbm{l})\setminus\mathbbm{m}$, $\mathbbm{m}\setminus(\mathbbm{i}\cup\mathbbm{l})$,
$(\mathbbm{l}\cap\mathbbm{m})\setminus\mathbbm{i}$ are pairwise disjoint sets, it is clear to see that $k_{\mathbf{i}\setminus\mathbf{l}}k_{\mathbf{l}\setminus\mathbf{m}}k_{\mathbf{m}\setminus\mathbf{i}}\!=\!
k_{\mathbf{i}\setminus(\mathbf{l}\cup\mathbf{m})}k_{(\mathbf{i}\cap\mathbf{m})\setminus\mathbf{l}}k_{\mathbf{l}\setminus(\mathbf{i}\cup\mathbf{m})}k_{(\mathbf{i}\cap\mathbf{l})\setminus\mathbf{m}}
k_{\mathbf{m}\setminus(\mathbf{i}\cup\mathbf{l})}k_{(\mathbf{l}\cap\mathbf{m})\setminus\mathbf{i}}$. Similarly, it is clear to see that
$k_{\mathbf{i}\setminus\mathbf{m}}k_{\mathbf{l}\setminus\mathbf{i}}k_{\mathbf{m}\setminus\mathbf{l}}\!=\!k_{\mathbf{i}\setminus(\mathbf{l}\cup\mathbf{m})}
k_{(\mathbf{i}\cap\mathbf{m})\setminus\mathbf{l}}k_{\mathbf{l}\setminus(\mathbf{i}\cup\mathbf{m})}
k_{(\mathbf{i}\cap\mathbf{l})\setminus\mathbf{m}}k_{\mathbf{m}\setminus(\mathbf{i}\cup\mathbf{l})}k_{(\mathbf{l}\cap\mathbf{m})\setminus\mathbf{i}}$.
Hence $k_{\mathbf{i}\setminus\mathbf{l}}k_{\mathbf{l}\setminus\mathbf{m}}k_{\mathbf{m}\setminus\mathbf{i}}=k_{\mathbf{i}\setminus\mathbf{m}}k_{\mathbf{l}\setminus\mathbf{i}}k_{\mathbf{m}\setminus\mathbf{l}}$.
Set $\mathbb{U}=\langle\{D_{\mathbf{a},\mathbf{l}}(q): \mathbf{a}\in\mathbb{D}(q)\}\rangle_{\mathbb{T}/\mathrm{Rad}(\mathbb{T})}$.
Lemma \ref{L;Lemma8.9} thus implies that $\mathbb{U}$ is an irreducible $\mathbb{T}$-module via the left multiplication.
By Corollary \ref{C;Corollary8.13}, it suffices to check that $\mathbb{U}$ is a self-contragredient $\mathbb{T}$-module
with respect to $\alpha_T$.

For any $\mathbf{r}\in\mathbb{D}(q)$, there is a unique $\beta_\mathbf{r}\in\mathrm{Hom}_{\F}(\mathbb{U},\F)$ such that
$\beta_\mathbf{r}(D_{\mathbf{s},\mathbf{l}}(q))=\delta_{\mathbf{r}, \mathbf{s}}$ for any $ {s}\in\mathbb{D}(q)$.
This thus implies that $\{\beta_\mathbf{a}:\mathbf{a}\in\mathbb{D}(q)\}$ is also an $\F$-basis of $\mathbb{U}^{\alpha_T}$.
Let $\Upsilon$ denote the $\F$-linear isomorphism from $\mathbb{U}^{\alpha_T}$ to $\mathbb{U}$ that sends $\beta_\mathbf{r}$ to the element
$$\overline{k_{\mathbf{l}\setminus\mathbf{r}}}(\overline{k_{\mathbf{r}\setminus\mathbf{l}}})^{-1}D_{\mathbf{r}, \mathbf{l}}(q)$$
for any $\mathbf{r}\in\mathbb{D}(q)$. Assume that $\mathbf{m}\in\mathbb{D}(q)$. If $p\mid k_\mathbf{j}$, Lemma \ref{L;Lemma3.10} and Theorem \ref{T;Jacobson} imply that
$\Upsilon(B_{\mathbf{i}, \mathbf{j}, \mathbf{k}}\beta_\mathbf{m})=O+\mathrm{Rad}(\mathbb{T})=B_{\mathbf{i}, \mathbf{j}, \mathbf{k}}\Upsilon(\beta_\mathbf{m})$.
If $p\nmid k_\mathbf{j}$ and $(\mathbf{i},\mathbf{j},\mathbf{k})\in\mathbb{D}_t$ for some $t\in[1, n_\sim]\setminus\{q\}$,
then $\Upsilon(D_{\mathbf{i}, \mathbf{j}, \mathbf{k}}\beta_\mathbf{m})=O+\mathrm{Rad}(\mathbb{T})=D_{\mathbf{i}, \mathbf{j}, \mathbf{k}}\Upsilon(\beta_\mathbf{m})$ by combining Equations \eqref{Eq;10}, \eqref{Eq;2}, Lemmas \ref{L;Lemma5.7}, \ref{L;Lemma7.13}. If $p\nmid k_\mathbf{j}$ and $(\mathbf{i}, \mathbf{j}, \mathbf{k})\in\mathbb{D}_q$, then $\mathbf{i}\!\in\!\mathbb{D}(q)$ by Lemma \ref{L;Lemma8.4}. The combination of Equations \eqref{Eq;10}, \eqref{Eq;2}, Lemmas \ref{L;Lemma5.7}, \ref{L;Lemma8.9} thus gives
$$\Upsilon(D_{\mathbf{i}, \mathbf{j}, \mathbf{k}}\beta_\mathbf{m})\!=\!\delta_{\mathbf{k}, \mathbf{m}}\overline{k_{\mathbf{i}\setminus\mathbf{m}}}\overline{k_{\mathbf{l}\setminus\mathbf{i}}}(\overline{k_{\mathbf{i}\setminus\mathbf{l}}}\overline{k_{\mathbf{m}\setminus\mathbf{i}}})^{-1}D_{\mathbf{i},\mathbf{l}}(q)\!=\!\delta_{\mathbf{k}, \mathbf{m}}\overline{k_{\mathbf{l}\setminus\mathbf{m}}}(\overline{k_{\mathbf{m}\setminus\mathbf{l}}})^{-1}D_{\mathbf{i},\mathbf{l}}(q)\!=\!D_{\mathbf{i}, \mathbf{j}, \mathbf{k}}\Upsilon(\beta_\mathbf{m}).$$
As $\mathbf{i}, \mathbf{j}, \mathbf{k}$ and $\beta_\mathbf{m}$ are chosen arbitrarily, the above discussion and Lemma \ref{L;Lemma7.3} thus imply that $\mathbb{U}^{\alpha_T}\cong\mathbb{U}$ as irreducible $\mathbb{T}$-modules. So the desired corollary follows.
\end{proof}
\begin{rem}\label{R;Remark8.16}
In general, up to isomorphism, a Terwilliger $\F$-algebra of a scheme owns at least an irreducible self-contragredient module with respect to $\alpha_T$ (see \cite{J1}).
\end{rem}
We next present a required lemma and get the second main result of this section.
\begin{lem}\label{L;Lemma8.17}
$\mathbb{T}$ is a Frobenius $\F$-algebra if and only if $\mathbb{S}$ is a $p'$-valenced scheme.
\end{lem}
\begin{proof}
Recall that $k_\mathbf{0}=1$ and $k_\mathbf{g}=\prod_{h\in\mathbbm{g}}(|\mathbb{U}_h|-1)$ if $\mathbf{g}\in\mathbb{E}\setminus\{\mathbf{0}\}$. By Theorem \ref{T;Semisimplicity} and
Lemma \ref{L;Lemma2.11}, it suffices to check that $\mathbb{T}$ is a Frobenius $\F$-algebra only if $\mathbb{S}$ is a $p'$-valenced scheme. Assume that $\mathbb{T}$ is a Frobenius $\F$-algebra and $\mathbb{S}$ is not a $p'$-valenced scheme. Let $i=|\{a: a\in[1, n], p\mid |\mathbb{U}_a|-1\}|$. Then $i>0$. By Lemma \ref{L;Lemma2.3}, there is $\mathbf{j}\in\mathbb{E}$ such that $p\mid k_\mathbf{j}$.
Set $\mathbb{U}=\langle\{B_{\mathbf{a}, \mathbf{a}, \mathbf{0}}: (\mathbf{a}, \mathbf{a}, \mathbf{0})\in\mathbb{P}, p\mid k_\mathbf{a}\}\rangle_\mathbb{T}$. Notice that $\mathbb{U}\neq\{O\}$ as $(\mathbf{j}, \mathbf{j}, \mathbf{0})\!\in\!\mathbb{P}$. Moreover, Lemma \ref{L;Lemma6.1} shows that $\mathbb{U}$ is a left ideal of $\mathbb{T}$.

Lemmas \ref{L;Lemma2.1} and \ref{L;Lemma2.2} yield $\mathbb{P}_{\mathbf{0}, \mathbf{k}}=\mathbb{P}_{\mathbf{k}, \mathbf{0}}=\{\mathbf{k}\}$ for any $\mathbf{k}\in\mathbb{E}$. This thus implies that $|\{\mathbf{a}: (\mathbf{a},\mathbf{a},\mathbf{0})\in\mathbb{P}, p\mid k_\mathbf{a}\}|=2^n-2^{n-i}$ by Lemma \ref{L;Lemma2.3}. Hence the $\F$-dimension of $\mathbb{U}$ is equal to $2^n-2^{n-i}$ by Theorem \ref{T;Theorem3.13}. Assume further that $M\in\mathrm{r}(\mathbb{U})$. Notice that $B_{\mathbf{0},\mathbf{k}, \mathbf{k}}\notin\mathrm{Supp}_{\mathbb{B}_2}(M)$ for any $\mathbf{k}\in\mathbb{E}$. Otherwise, notice that $B_{\mathbf{j}, \mathbf{j}, \mathbf{0}}M\neq O$ by Theorems \ref{T;Theorem3.23} and \ref{T;Theorem3.13}. This thus contradicts the definition of $\mathrm{r}(\mathbb{U})$. Theorem \ref{T;Theorem3.13} thus forces that $\mathrm{r}(\mathbb{U})\subseteq
\langle\{B_{\mathbf{a}, \mathbf{b}, \mathbf{c}}: (\mathbf{a}, \mathbf{b}, \mathbf{c})\in\mathbb{P}, \mathbf{a}\neq\mathbf{0}\}\rangle_\mathbb{T}$. This containment thus implies that
$\mathrm{r}(\mathbb{U})=\langle\{B_{\mathbf{a}, \mathbf{b}, \mathbf{c}}: (\mathbf{a}, \mathbf{b}, \mathbf{c})\in\mathbb{P}, \mathbf{a}\neq\mathbf{0}\}\rangle_\mathbb{T}$ by Theorems
\ref{T;Theorem3.23} and \ref{T;Theorem3.13}. Recall that $n_1=|\{a: a\in[1,n], |\mathbb{U}_a|=2\}|$, $n_2=|\{a: a\in[1,n], |\mathbb{U}_a|>2\}|$, and $n=n_1+n_2$. Then $|\{(\mathbf{a}, \mathbf{b},\mathbf{c}): (\mathbf{a}, \mathbf{b}, \mathbf{c})\in\mathbb{P}, \mathbf{a}\neq\mathbf{0}\}|=2^{2n_1}5^{n_2}-2^n$ by combining Lemmas \ref{L;Lemma2.1}, \ref{L;Lemma2.2}, \ref{L;Lemma2.3}, and Theorem \ref{T;Theorem3.13}. Hence the $\F$-dimension of $\mathrm{r}(\mathbb{U})$ is equal to $2^{2n_1}5^{n_2}-2^n$ by Theorem \ref{T;Theorem3.13}. As the sum of the $\F$-dimensions of $\mathbb{U}$ and $\mathrm{r}(\mathbb{U})$ is equal to $2^{2n_1}5^{n_2}-2^{n-i}$, notice that $\mathbb{T}$ is not a Frobenius $\F$-algebra by Lemma \ref{L;Lemma2.10}. This is a contradiction. The desired lemma follows from the above discussion.
\end{proof}
\begin{thm}\label{T;Frobenius}
The following are equivalent: $\mathbb{S}$ is a $p'$-valenced scheme; $\mathbb{T}$ is a semisimple $\F$-algebra;
$E_\mathbf{g}^*\mathbb{T}E_\mathbf{g}^*$ is a semisimple $\F$-subalgebra of $\mathbb{T}$ for any $\mathbf{g}\in\mathbb{E}$; $E_\mathbf{1}^*\mathbb{T}E_\mathbf{1}^*$ is a semisimple $\F$-subalgebra of $\mathbb{T}$; $\mathrm{Z}(\mathbb{T})$ is a semisimple $\mathbb{F}$-subalgebra of $\mathbb{T}$; $\mathbb{T}$ is a symmetric $\F$-algebra; $\mathbb{T}$ is a Frobenius $\F$-algebra; $\mathbb{T}$ must satisfy the formula
$$\mathbb{T}\cong\bigoplus_{g=1}^{n_\sim}\mathrm{M}_{|\mathbb{D}(g)|}(\F)\ \text{ as $\mathbb{F}$-algebras}.$$
In particular, one of the above eight conditions holds if and only if $p\nmid \prod_{h=1}^n(|\mathbb{U}_h|-1)$.
\end{thm}
\begin{proof}
As $\mathbb{T}$ is a semisimple $\F$-algebra only if $\mathbb{T}$ is a symmetric $\F$-algebra, the desired theorem follows from the combination of Corollaries \ref{C;Corollary5.21}, \ref{C;Corollary8.12}, and Lemma \ref{L;Lemma8.17}.
\end{proof}
We end the presentation of this paper by an example of Theorems \ref{T;Decomposition} and \ref{T;Frobenius}.
\begin{eg}\label{E;Example8.19}
Assume that $n=|\mathbb{U}_1|=2$ and $|\mathbb{U}_2|=3$. It is clear that $|\mathbb{E}|=4$. Assume that $\mathbf{g}=(0, 1)$ and $\mathbf{h}=(1,0)$. Hence $\mathbb{E}=\{\mathbf{0}, \mathbf{g}, \mathbf{h}, \mathbf{1}\}$. Assume that $p\neq 2$. Notice that  $\mathbb{D}=\mathbb{P}$ and $n_\sim=2$. Moreover, the cardinalities of the equivalence classes of $\mathbb{D}$ with respect to $\sim$ are sixteen and four. By Theorem \ref{T;Frobenius}, $\mathbb{T}\cong \mathrm{M}_4(\F)\oplus\mathrm{M}_2(\F)$ as $\F$-algebras. Assume that $p=2$. Notice that $\mathbb{D}$ contains exactly $(\mathbf{0}, \mathbf{0}, \mathbf{0})$, $(\mathbf{0},\mathbf{h},\mathbf{h})$, $(\mathbf{g}, \mathbf{0}, \mathbf{g})$, $(\mathbf{g}, \mathbf{h}, \mathbf{1})$, $(\mathbf{h}, \mathbf{0}, \mathbf{h})$, $(\mathbf{h}, \mathbf{h}, \mathbf{0})$, $(\mathbf{1}, \mathbf{0}, \mathbf{1})$, $(\mathbf{1}, \mathbf{h}, \mathbf{g})$ and $n_\sim=2$. Moreover, all cardinalities of the equivalence classes of $\mathbb{D}$ with respect to $\sim$ are equal to four. A direct application of Theorem \ref{T;Decomposition} implies that $\mathbb{T}/\mathrm{Rad}(\mathbb{T})\cong2\mathrm{M}_2(\F)$ as $\F$-algebras.
\end{eg}
\subsection*{Disclosure statement} No relevant financial or nonfinancial interests are reported.
\subsection*{Data availability statement}All used data are contained in this submitted paper.
%\subsection*{Acknowledgements}
%The authors gratefully thank Prof. Changchang Xi for some insightful comments on the theme of this paper and his constant encouragements.
%This present research was supported by the National Natural Science Foundation of China (Youth Program, No. 12301017) and Anhui Provincial Natural Science Foundation (Youth
%Program, No. \!2308085QA01). The author would like to gratefully thank %Prof. Tatsuro Ito and Prof. Gang Chen for encouragements.%constant encouragements.%Professor Tatsuro Ito for constant encouragements. %He also gratefully thank Professor Gang Chen for some comments on this present research.

\end{document}